\documentclass[]{amsart} \usepackage{amsfonts}

\usepackage{bbding}
 \usepackage[dvips]{graphicx}
\usepackage{amsfonts,amssymb,amsmath,amsthm,amscd}

\usepackage{mathrsfs}
\usepackage{xcolor}
\usepackage{empheq}
\usepackage{adforn}
\usepackage{cancel}
\usepackage{xr}
\usepackage{mdframed}
\usepackage{tikz}
\usepackage{amsfonts}
\usepackage{mathdots}
\usepackage{pinlabel}

\usepackage{enumerate}
\usepackage{lmodern}
\usepackage{mdframed}
\usepackage{stmaryrd}
\usepackage{blindtext}
\usepackage{leftidx}

\usepackage{accents}

\definecolor{slightblue}{rgb}{.8, .8, 1}
\definecolor{tif}{RGB}{10, 186, 181}
\definecolor{hair}{RGB}{100,225,190}
\definecolor{ruby}{RGB}{220,50,120}
\definecolor{grass}{RGB}{150,220,110}

\usepackage[colorlinks=true, pdfstartview=FitP, linkcolor=tif!85!black, citecolor=ruby, urlcolor=blue!80!black]{hyperref}

\usepackage{tgpagella}
\usepackage[T1]{fontenc}

\newtheorem{theorem}{Theorem}[section] 
\newtheorem*{theorem*}{Theorem}
\newtheorem{proposition}[theorem]{Proposition}
\newtheorem{lemma}[theorem]{Lemma}
 \newtheorem{corollary}[theorem]{Corollary}

\theoremstyle{definition} 
\newtheorem{definition}[theorem]{Definition}

\newtheorem{example}[theorem]{Example}

\theoremstyle{remark} \newtheorem{remark}[theorem]{Remark} \numberwithin{equation}{section}
\numberwithin{figure}{section}

\newcommand{\G}{\mathscr{G}}
\newcommand{\Diffeo}{\mathsf{Diffeo}}

\newcommand{\Ends}{\mathsf{Ends}}

\newcommand{\Hom}{\mathrm{Hom}}

\newcommand{\Aut}{\mathrm{Aut}}

\newcommand{\gjphi}{g_{{}_{\ac{J},\ve{\phi}}}}

\newcommand{\pt}{$\bullet$ }

\newcommand{\ve}[1]{{\boldsymbol{#1}}}
\newcommand{\T}{\mathsf{T}}

\newcommand{\dif}{\mathsf{d}}

\newcommand{\vol}{\mathit{vol}}

\newcommand{\ac}[1]{{\boldsymbol{#1}}} 

\newcommand{\pa}{\partial}
\newcommand{\pai}[1]{\pa_\infty^{(#1)}}
\newcommand{\bpa}{{\bar{\partial}}}

\newcommand{\dz}{{\dif z}}
\newcommand{\dbz}{{\dif\bar z}}

\newcommand{\paz}{\partial_z}
\newcommand{\pabz}{\partial_{\bar z}}
\newcommand{\pazbz}{\partial^2_{z\bar z}}

\newcommand{\ima}{\boldsymbol{i}} 

\newcommand{\SL}{\mathrm{SL}}

\newcommand{\eg}{\textit{e.g. }}
\newcommand{\cf}{\textit{c.f. }}
\newcommand{\ie}{\textit{i.e. }}
\newcommand{\etc}{\textit{etc.}}

\DeclareMathOperator{\Hess}{Hess}

\DeclareMathOperator{\Ad}{Ad}

\DeclareMathOperator{\End}{End}

\DeclareMathOperator{\im}{\mathsf{Im}}
\DeclareMathOperator{\re}{\mathsf{Re}}

\renewcommand{\H}{\mathbf{H}}
\newcommand{\CH}{\overline{\H}}

\newcommand{\para}{\ve{T}}
\renewcommand{\sl}{\mathfrak{sl}}

\newcommand{\X}[1]{X(#1)}

\newcommand{\D}{\mathsf{D}}

\begin{document}

\title{Poles of cubic differentials and ends of convex $\mathbb{RP}^2$-surfaces}
\author{Xin Nie} 
\address{Tsinghua University, Beijing, China}
 \email{nie.hsin@gmail.com}

\maketitle

\begin{abstract}
The affine sphere construction gives, on any oriented surface, a one-to-one correspondence between convex $\mathbb{RP}^2$-structures and holomorphic cubic differentials. Generalizing results of Benoist-Hulin, Loftin and Dumas-Wolf, we show that poles of order less than $3$ of cubic differentials correspond to finite volume ends of convex $\mathbb{RP}^2$-structures, and poles of order $3$ (resp. bigger than $3$) correspond to geodesic (resp. piecewise geodesic) ends. In particular, at a pole of order at least $3$, we bordify the surface by attaching to it a boundary circle in a natural way with respect to the cubic differential, and show that the $\mathbb{RP}^2$-structure extends to the boundary in a metric preserving way. 
\end{abstract}

\section{Statement of results}

\subsection{$\mathbb{RP}^2$-structures and Wang's equation}
In this paper, by a \emph{surface}, we mean an oriented $2$-manifold, without boundary unless otherwise specified, but not necessarily closed.
An \emph{$\mathbb{RP}^2$-structure} on a surface $S$ is a maximal atlas with charts taking values in $\mathbb{RP}^2$ and transition maps in the group $\SL(3,\mathbb{R})$ of projective transformations. Such a structure is determined by a \emph{developing map}, a local homeomorphism from the universal cover $\widetilde{S}$ to $\mathbb{RP}^2$ constructed by assembling the charts. An $\mathbb{RP}^2$-structure $X$ is said to be \emph{convex} if its developing map sends $\widetilde{S}$ bijectively to a bounded convex set $\Omega$ in an affine chart $\mathbb{R}^2\subset\mathbb{RP}^2$. 

One can construct an $\mathbb{RP}^2$-structure from a pair $(g,\ve{\phi})$ on $S$ consisting of a Riemannian metric $g$ and a holomorphic cubic differential $\ve{\phi}$ 
satisfying \emph{Wang's equation}
\begin{equation}\label{eqn_wangintro}
\kappa_g=-1+\|\ve{\phi}\|^2_g.
\end{equation}
In fact, such a pair determines, up to $\SL(3,\mathbb{R})$-action, an equivariant immersion $\iota$ of $\widetilde{S}$ into $\mathbb{R}^3$ as a \emph{hyperbolic affine sphere} centered at $0$, for which $g$ and $\ve{\phi}$ are the \emph{Blaschke metric} and (normalized) \emph{Pick differential}, respectively. The projectivized immersion $\mathbb{P}\circ\iota:\widetilde{S}\rightarrow\mathbb{RP}^2$ is the developing map of an $\mathbb{RP}^2$-structure, and is convex if $g$ is complete. Conversely, the unique existence theorem for complete affine spheres  \cite{cheng-yau_1} implies that every convex $\mathbb{RP}^2$-structure is obtained this way.

After the seminal works of Labourie \cite{labourie_cubic} and Loftin \cite{loftin_amer} on closed surfaces, a theme of research \cite{loftin_compactification, loftin_neck, benoist-hulin, benoist-hulin_2, dumas-wolf} consists in first showing that for certain classes of pairs $(\ac{J},\ve{\phi})$ on an open surface $S$ (where $\ac{J}$ is a conformal structure and $\ve{\phi}$ a holomorphic cubic differential under $\ac{J}$), there exists a unique complete metric $g$ conformal to $\ac{J}$ satisfying Eq.(\ref{eqn_wangintro}), then deduce from the affine sphere construction a one-to-one correspondence between those pairs and certain classes of convex $\mathbb{RP}^2$-structures. Combining existing results, we get the following fully general unique existence theorem for $g$, where we allow  $k$-differentials with any $k\geq1$ instead of cubic differentials:
\begin{theorem}[\cite{wan, au-wan, li_on}]\label{thm_intro1}
Given a Riemann surface $\Sigma$ and a nontrivial holomorphic $k$-differential $\ve{\phi}$ on $\Sigma$, there exists a unique complete conformal metric $g$ on $\Sigma$ satisfying
Eq.(\ref{eqn_wangintro}). Moreover, we have either $\kappa_g<0$ on $\Sigma$ or $\kappa_g\equiv0$.
\end{theorem}
The proof is an adaptation of the unique existence results \cite{wan, au-wan} in the case $k=2$ under the assumption $\kappa_g\leq 0$, combined with the method of Qiongling Li \cite{li_on} which removes the assumption. We also extend the result in \cite{li_on} on nonexistence of incomplete solutions for $\Sigma=\mathbb{C}$ to arbitrary Riemann surfaces:
\begin{theorem}\label{thm_introli}
Let $\ve{\phi}$ be a holomorphic $k$-differential with finitely many zeros on a Riemann surface $\Sigma$. Then $\Sigma$ does not admit any incomplete conformal metric $g$ satisfying $\kappa_g=-1+\|\ve{\phi}\|^2_g$ if and only if $\Sigma$ is obtained from a closed Riemann surface by removing a finite (possibly empty) set of punctures and $\ve{\phi}$ has a pole of order at least $k$ at every puncture.
\end{theorem}

\subsection{Geometry of convex $\mathbb{RP}^2$-surfaces at poles}\label{subsec_intro2}
Theorem \ref{thm_intro1} and the affine sphere construction yield, for any surface $S$, a one-to-one correspondence between convex $\mathbb{RP}^2$-structures and pairs $(\ac{J},\ve{\phi})$ (excluding those such that $(S,\ac{J})$ is parabolic and $\ve{\phi}=0$, see Section \ref{subsec_yau}). One then naturally asks how certain properties of convex $\mathbb{RP}^2$-structures can be read off from the cubic differential side and vice versa. 

The main purpose of this paper is to show that, for an end of $S$, the property of being a pole for $(\ac{J},\ve{\phi})$ (\ie the end is conformally a puncture for $\ac{J}$, at which $\ve{\phi}$ has a pole) corresponds to specific behaviors of the  $\mathbb{RP}^2$-structure:
\begin{theorem}\label{thm_intro2}
Let $\Sigma$ be a Riemann surface endowed with a holomorphic cubic differential $\ve{\phi}$ and $X$ be the convex $\mathbb{RP}^2$-structure on $\Sigma$ given by the complete solution to Wang's equation and the affine sphere construction. For every end $p$ of $\Sigma$, the following equivalences hold:
\begin{enumerate}
\item\label{item_introthm0}
$p$ is a pole of order at most $2$ if and only if it is a finite volume end for $X$;
\item\label{item_introthm2}
$p$ is a pole of order $m\geq 4$ if and only if $X$ extends to a convex $\mathbb{RP}^2$-structure $X'$ with piecewise geodesic boundary on a surface $\Sigma'$ obtained by attaching to $\Sigma$ a boundary circle around $p$, such that the set of vertices $\Lambda\subset\pa\Sigma'$ has $m-3$ points;
\item\label{item_introthm1}
$p$ is a third order pole if and only if $X$ extends to a convex $\mathbb{RP}^2$-structure $X'$ with geodesic boundary on a surface $\Sigma'$ obtained by attaching to $\Sigma$ a boundary circle around $p$. 
\end{enumerate}
Moreover, in the last case, the boundary holonomy of $X'$ has eigenvalues $\exp\!\big[\sqrt{2}\re(2\pi\ima\, a_j)\big]$, $j=1,2,3$, where $a_1$, $a_2$ and $a_3$ are the cubic roots of the residue
 $R$ of $\ve{\phi}$ at $p$, and we have $\re(R)>0$ if and only if $X'$ has principal geodesic boundary.

\end{theorem}
Here, the notion of \emph{finite volume end} in Part (\ref{item_introthm0}) is defined using the Hilbert metric (see Section \ref{subsec_cusps}). In Parts (\ref{item_introthm2}) and (\ref{item_introthm1}), a convex $\mathbb{RP}^2$-structure $X'$ on a surface $S$ with boundary is said to have \emph{piecewise geodesic boundary} with \emph{vertices} $\Lambda\subset\pa S$ if the $\mathbb{RP}^2$-charts of $X'$ identify a neighborhood of every point in $\pa S\setminus\Lambda$ (resp. $\Lambda$) with a neighborhood of a boundary point of a half-plane (resp. neighborhood of the vertex of an angle) in $\mathbb{RP}^2$. If $\Lambda$ is empty then $X'$ is said to have \emph{geodesic boundary}. When $\pa S$ is homeomorphic to a circle and $X'$ has geodesic boundary, $\pa S$ is developed into an open line segment $I\subset\mathbb{RP}^2$ with endpoints fixed by the boundary holonomy $H$ of $X'$, and the geodesic boundary is said to be \emph{principal} if $H$ is hyperbolic (\ie has distinct positive eigenvalues) and $I$ joins the attracting and repelling fixed points of $H$ (see \cite{goldman_convex}). It follows from the expression of eigenvalues in the theorem that $H$ is hyperbolic if and only if $\re(R)\neq0$, so the last statement of the theorem implies that if $\re(R)<0$ then the boundary is not principal, in which case the end is given by the quotient of a triangle in $\mathbb{RP}^2$ by $H$.

Parts (\ref{item_introthm0}) and (\ref{item_introthm1}) of Theorem \ref{thm_intro2} generalize results of Benoist-Hulin \cite{benoist-hulin} and Loftin \cite{loftin_compactification, loftin_neck}, respectively; whereas Part (\ref{item_introthm2}) generalizes a theorem of Dumas-Wolf \cite{dumas-wolf} in the $\Sigma=\mathbb{C}$ case. The main novelties of this paper are an analytic proof of (\ref{item_introthm0}) which applies to $k$-differentials (see Theorem \ref{thm_finitevolumeend}), and an improvement of techniques in \cite{dumas-wolf} which allows us to explicitly describe the surface $X'$ in Theorem \ref{thm_intro2}. We outline the latter in Sections \ref{subsec_intro_negative} and \ref{subsec_intro_asymp} below.

\subsection{Natural bordifications of poles}\label{subsec_intro_negative}
For the ``only if'' statements in Theorem \ref{thm_intro2}, we first construct $\Sigma'$ by specifying how the circle is attached to the pole (\ie specifying a topology on $\Sigma\cup\mathbb{S}^1$), then show that $X$ extends to $\Sigma'$. We explain in this section the idea of the construction and leave the details to Chapter \ref{sec_5}.

First consider the simplest example $(\Sigma,\ve{\phi})=(\mathbb{C},\dz^3)$, where the cubic differential  has a pole of order $6$ at infinity. The Euclidean metric $|\dz|^2$ is the complete solution to Wang's equation and gives rise to the \emph{\c{T}i\c{t}eica affine spherical embedding} $\iota_0:\mathbb{C}\to\mathbb{R}^3$, which has an explicit expression (see Section \ref{subsec_titeica}), and 
the projectivized embedding $\delta_0=\mathbb{P}\circ\iota_0$ sends $\mathbb{C}$ to a triangle $\Delta$ in $\mathbb{RP}^2$, see Figure \ref{figure_dev01}. Thus, the compactification $\mathbb{C}'=\mathbb{C}\cup\pa_\infty\mathbb{C}$ of $\mathbb{C}$ claimed in Theorem \ref{thm_intro2} (\ref{item_introthm2}) can be identified with the closure $\overline{\Delta}$.
\begin{figure}[h]
	\labellist
	\pinlabel {$\delta_0$} at 420 240
	\endlabellist
\centering
\includegraphics[width=3.2in]{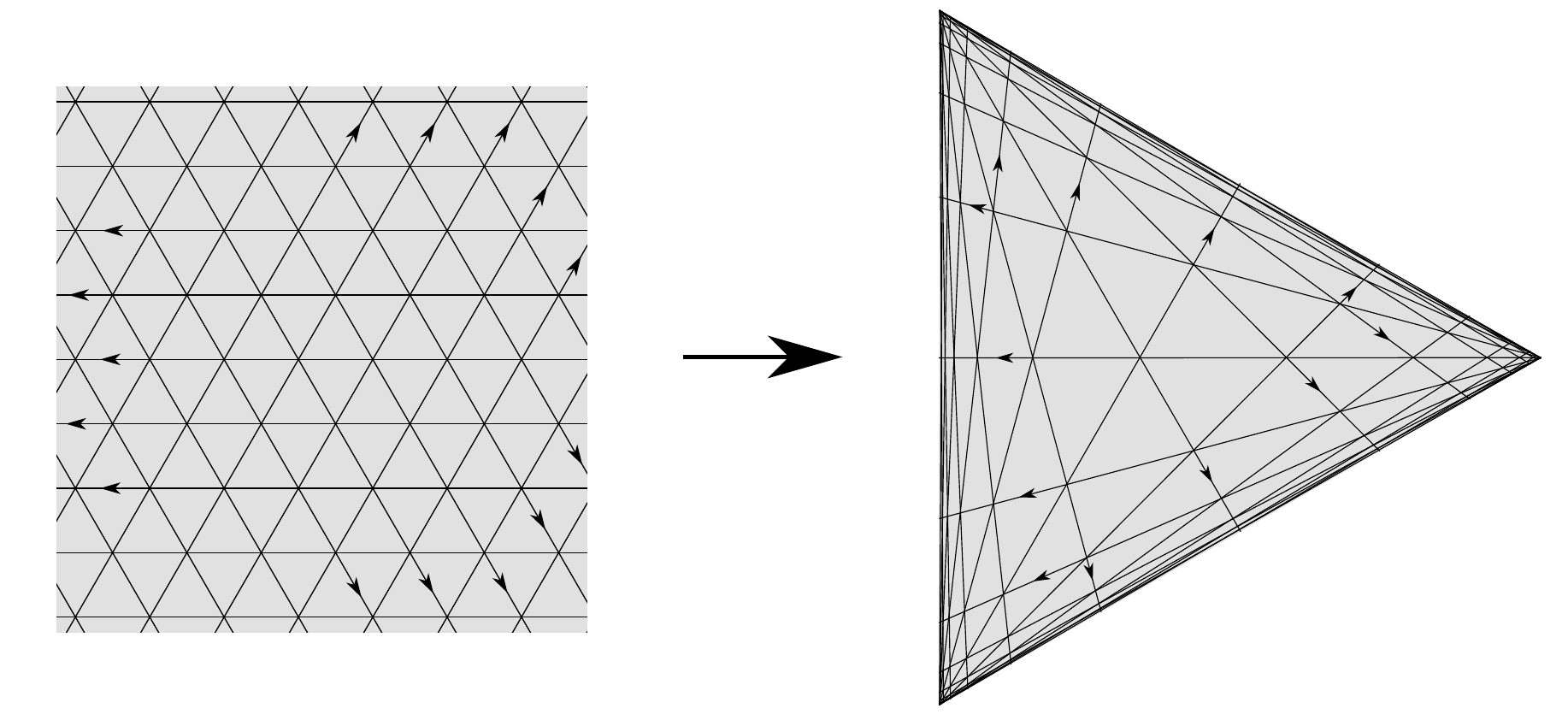}
\caption{Some negative geodesics in $(\mathbb{C},\dz^3)$ and images by $\delta_0$.}
\label{figure_dev01}
\end{figure}
As the figure shows, $\delta_0$ sends every \emph{negative geodesic} -- oriented curve such that the cubic differential evaluated on every tangent vector gives a negative value -- to a line in $\Delta$ going from a vertex to an interior point of an edge. Therefore, each non-vertex point of the boundary $\pa_\infty\mathbb{C}=\pa\Delta$ is the limit of a unique inextensible negative geodesic, while every nonnegative geodesic tends to a vertex.


For general $(\Sigma, \ve{\phi})$ with a pole $p$ of order $m\geq 4$, we construct $\Sigma'$ by attaching to $p$ a boundary circle $\pa\Sigma'$ with a subset $\Lambda$ of $m-3$ marked points in a specific way, incorporating similar properties relative to (non)negative geodesics as in the above example. Namely, the points in $\Lambda$ are limits of nonnegative geodesics rays on $\Sigma$ tending to $p$; while each connected component of $\pa\Sigma'\setminus\Lambda$ is formed by limits of negative rays parallel to each other, hence carries a natural metric given by the distances between these rays under the flat metric $|\ve{\phi}|^\frac{2}{3}$. We call $\Sigma'$ the \emph{negative ray bordification} of $(\Sigma,\ve{\phi})$ at $p$.  

When $p$ is a third order pole, a neighborhood of $p$ is a cylinder (\ie quotient of a half-plane in $(\mathbb{C},\dz^3)$ by a translation), and the bordification $\Sigma'$ that we should consider is slightly different: we pick a negative ray $\gamma$ tending to $p$, and construct $\Sigma'=\Sigma'_\gamma$ by assigning to each geodesic ray parallel to $\gamma$ (which is also negative) a boundary point. Note that there can exist either one or two parallelism families of negative rays tending to $p$, corresponding to the cases $\re(R)\geq 0$ and $\re(R)<0$, respectively (see Section \ref{subsec_bord2}), so the construction does involve a choice in the latter case. Also note that this time there is no marked boundary points, and the whole boundary circle carries the metric given by the distances between rays.

We show that the $\mathbb{RP}^2$-structure extends to the bordified surface $\Sigma'$ in a metric preserving way, which implies the ``only if'' direction in Theorem \ref{thm_intro2} (\ref{item_introthm2}) and (\ref{item_introthm1}):
\begin{theorem}\label{thm_intro3}
Let $\Sigma$, $\ve{\phi}$, $X$ and $p$ be as in Theorem \ref{thm_intro2} and suppose $p$ is a pole of order $m$.
\begin{enumerate}
\item\label{item_thmintro31}
When $m\geq 4$, $X$ extends to a convex $\mathbb{RP}^2$-structure $X'$ with piecewise geodesic boundary on the negative ray bordification $\Sigma'$ of $(\Sigma,\ve{\phi})$ at $p$. Moreover, the set $\Lambda\subset\pa\Sigma'$ given by nonnegative geodesics is the set of vertices of $X'$, and the metric on each component of $\pa\Sigma'\setminus\Lambda$ is $\frac{1}{\sqrt{6}}$ times the Hilbert metric induced by $X'$. 
\item\label{item_thmintro32}
When $m=3$, if $\gamma$ is a negative ray tending to $p$ such that the angle $\theta$ from $\gamma$ to the geodesic loops around $p$ (see Figure \ref{figure_cylinder}) is bigger than $\arcsin(\frac{\sqrt{3}}{4})\approx 0.14\pi$, then $X$ extends to a convex $\mathbb{RP}^2$-structure $X'$ with geodesic boundary on the bordification $\Sigma'_\gamma$ of $(\Sigma, \ve{\phi})$ relative to $\gamma$. The natural metric on $\pa\Sigma'_\gamma$ is $\frac{1}{\sqrt{6}}$ times the Hilbert metric induced by $X'$. 
\end{enumerate}
\end{theorem} 
\begin{figure}[h]
	\labellist
	\pinlabel {monodromy} at 9 120
	\pinlabel {$z\mapsto z+\tau$} at 18 85
	\pinlabel {$e^{\frac{\pi\ima}{3}}\mathbb{R}_+$} at 265 280
	\pinlabel {$e^{-\frac{\pi\ima}{3}}\mathbb{R}_+$} at 265 110
	\pinlabel {$\theta$} at 145 225
	\pinlabel {$\theta$} at 590 170
	\pinlabel {$\theta$} at 803 170
	\pinlabel {quotient} at 480 240
	\pinlabel {$p$} at 1182 195
	\endlabellist
	\centering
	\includegraphics[width=4in]{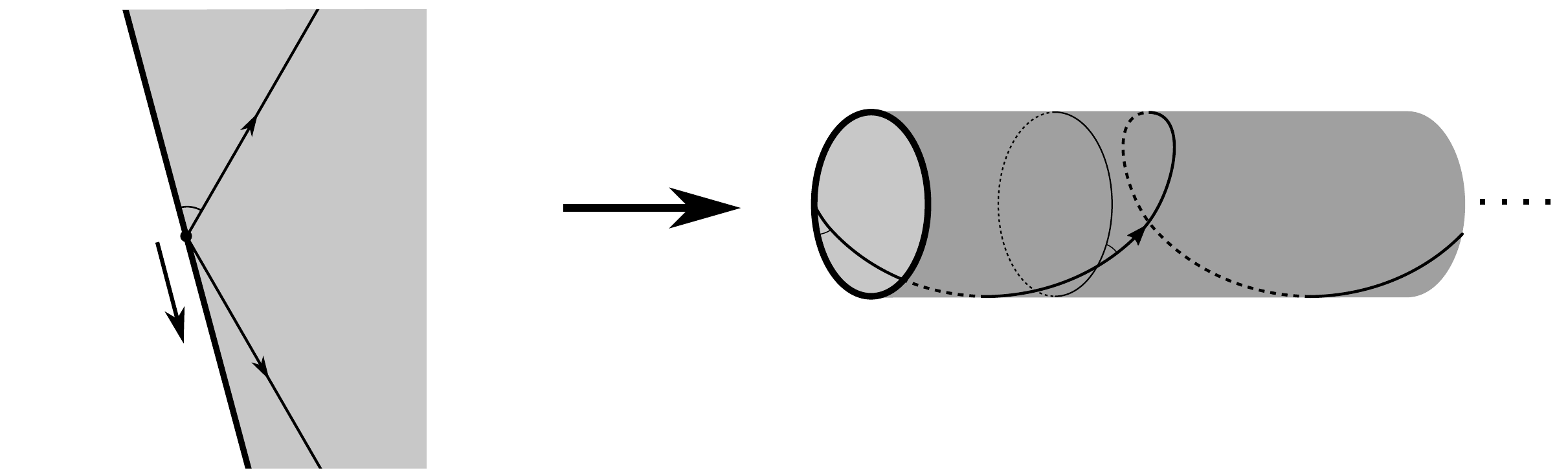}
	\caption{A neighborhood of a third order pole with $\re(R)<0$ as quotient of a half-plane in $(\mathbb{C},\dz^3)$, and two negative rays tending to the pole (only one is drawn in the quotient picture).}
	\label{figure_cylinder}
\end{figure}
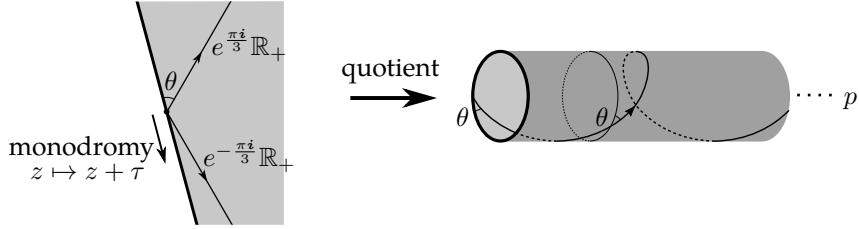
Here, the \emph{Hilbert metric} on a geodesic boundary piece of a convex $\mathbb{RP}^2$-surface is given by the metric $d_I(x,y)=\log|[a,x,y,b]|$ on the developing image $I\subset\mathbb{RP}^2$ of that piece, where $a$ and $b$ are the endpoints of the open line segment $I$ and $[a,x,y,b]$ is the cross-ratio. The assumption on $\theta$ in (\ref{item_thmintro32}) comes from an assumption required in one of our results about osculation maps, see the next section. Note that the assumption is weak enough such that the statement implies the ``only if'' part of Theorem \ref{thm_intro2}  (\ref{item_introthm1}), because one readily checks that
\begin{itemize}
	\item when $\re(R)\geq0$, the unique (up to parallelism) negative ray $\gamma$ tending to $p$  satisfies the assumption;
	\item when $\re(R)<0$, between the only two negative rays tending to $p$ (up to parallelism), at least one satisfies the assumption.
\end{itemize}

\subsection{Asymptotics of osculation map}\label{subsec_intro_asymp}
The proof of Theorem \ref{thm_intro3} is bases on the following construction:
Let $U\subset\Sigma$ be an open set conformally identified with a domain in $\mathbb{C}$ such that $\ve{\phi}|_U=\dz^3$, and let $\iota_0:U\rightarrow\mathbb{R}^3$ be the \c{T}i\c{t}eica affine sphere from the last section. Then any affine spherical immersion $\iota:U\rightarrow\mathbb{R}^3$ with normalized Pick differential $\dz^3$ tangent to $\iota_0$ at a point $z_0\in U$ can be written as
$$
\iota(z)=P(z)\iota_0(z)\ \text{ for all } z\in U,
$$ 
where the \emph{osculation map} $P:U\rightarrow\SL(3,\mathbb{R})$, measuring the deviation of $\iota$ from $\iota_0$, is the solution to a specific Pfaffian system whose coefficients are given by the Blaschke metric of $\iota$, with initial condition $P(z_0)=I$. 

A chart of the convex $\mathbb{RP}^2$-structure $X$ on $U$ is the projectivization $\delta=\mathbb{P}\circ\iota$ of an affine spherical immersion $\iota$, hence can be written as $\delta(z)=P(z)\delta_0(z)$. As $\delta_0$ is well understood, key to the proof of Theorem \ref{thm_intro3} is to understand the asymptotic behavior of $P(z)$ when $z\rightarrow+\infty$, for certain $U$ around the puncture $p$. Results of this kind are first established in \cite{dumas-wolf}, revealing a \emph{Stokes-type phenomenon}: If $p$ is a pole of order greater than $3$, then $P(z)$ converges in $\SL(3,\mathbb{R})$ when $z\rightarrow\infty$ in a sector $\overline{V}=\{z\in\mathbb{C}\mid \theta_1\leq\arg(z)\leq\theta_2\}\subset U$ as long as $\overline{V}$ does not contain any $v$ with $v^3=\pm\ima$ (called an ``unstable direction'' in \cite{dumas-wolf}); whereas if $\overline{V}$ contains exactly one such $v$ then the limits within the sub-sectors on the two sides of $v$ are related by an element in a specific $1$-dimensional unipotent subgroup $\mathcal{N}_v\subset\SL(3,\mathbb{R})$. In the latter case, we further show that $P(z)$ factorizes as
$$
P(z)=\widehat{P}(z)R(z),
$$ 
where $\widehat{P}(z)$ converges in $\SL(3,\mathbb{R})$ as $z\rightarrow+\infty$, whereas $R(z)$ is in $\mathcal{N}_v$ and has at most polynomial growth with respect to $|z|$ (see Proposition \ref{prop_reduction1}). This implies that $R(z)\delta_0(z)$ has the same asymptotic behavior as $\delta_0(z)$, allowing us to keep track of $\delta(z)$ when $z\to\infty$ along unstable directions and prove Theorem \ref{thm_intro3} (\ref{item_thmintro31}).

Assuming $g=e^{u(z)}|\dz|^2$ on $U$, a key ingredient in the proof of the above results is the exponential decay property of $u$ provided by \cite{dumas-wolf}:
$$
0\leq u(z)\leq C|z|^{-\frac{1}{2}}e^{-\sqrt{6}|z|} \text{ when $|z|$ is big enough,}
$$
where the coefficient $-\sqrt{6}$ is crucial. However, while the factor $|z|^{-\frac{1}{2}}$ is important in \cite{dumas-wolf}, it is nonessential for our factorization result, hence we will only use a weaker estimate, with $|z|^{-\frac{1}{2}}$ replaced by $|z|^{\frac{1}{2}}$. Unlike the above estimate, which is very specific to poles of order $\geq4$, the weaker estimate holds in more general settings (see Theorem \ref{thm_diskfine}) and is also used in our treatment of third order poles.

The analysis of osculation map is trickier at third order poles: Identifying a neighborhood of the pole as a cylinder as before, we only have a weaker estimate $0\leq u(z)\leq Ce^{-\sqrt{6}\sin(\theta)|z|}$ for $z$ in the vicinity of a negative ray tending to the pole, where $\theta$ is the angle in Figure \ref{figure_cylinder}. If $\sin(\theta)>\frac{\sqrt{3}}{4}$, we show a similar factorization $P(z)=\widehat{P}(z)R(z)$ for $z$ in a small sector center at the ray (Proposition \ref{prop_osculation2}), where $R(z)$ belongs to certain unipotent subgroup of $\SL(3,\mathbb{R})$ and has controlled exponential growth. This enables us to prove Part (\ref{item_thmintro32}) of Theorem \ref{thm_intro3}. 


\subsection{Moduli spaces and Fock-Goncharov framings}\label{subsec_intro_bijection}
We conclude this introductory chapter by pointing out some consequences of Theorem \ref{thm_intro2} about moduli spaces of convex $\mathbb{RP}^2$-structures, which we will not pursued further in this paper.

The bijection mentioned in Section \ref{subsec_intro2} between the space $\mathscr{P}(S)$ of convex $\mathbb{RP}^2$-structures and the space $\mathscr{C}(S)$ of pairs $(\ac{J},\ve{\phi})$ induces a bijection between the moduli spaces $\mathcal{C}(S)=\mathscr{C}(S)/\Diffeo^0(S)$ and $\mathcal{P}(S)=\mathscr{P}(S)/\Diffeo^0(S)$, where $\Diffeo^0(S)$ is the identity component of the diffeomorphism group of $S$.  In particular, when $S$ is closed with genus $g\geq 2$, $\mathcal{P}(S)$ is the \emph{Hitchin component} in the $\SL(3,\mathbb{R})$-character variety of $\pi_1(S)$, and $\mathcal{C}(S)$ is a holomorphic vector bundle of rank $5(g-1)$ over the Teichm\"uller space $\mathcal{T}_g$. The bijection $\mathcal{C}(S)\cong\mathcal{P}(S)$ in this case, due to Labourie \cite{labourie_cubic} and Loftin \cite{loftin_amer} independently, is the origin of the whole subject here. 

On the other hand, when $S=S_{g,b}$ is a surface of genus $g\geq 0$ with $b>0$ punctures, although $\mathcal{C}_{g,b}:=\mathcal{C}(S)$ and $\mathcal{P}_{g,b}:=\mathcal{P}(S)$ are infinite-dimensional spaces, Theorem \ref{thm_intro2} yields identifications between various finite dimensional submanifolds in them. Namely, given a $b$-tuple of integers $m=(m_1,\cdots, m_b)$, we consider the subset $\mathcal{C}_{g,b}^{\leqslant m}\subset\mathcal{C}_{g,b}$ given by all $(\ac{J},\ve{\phi})$ with a pole of order at most $m_i$ at the $i^\text{th}$ puncture for every $i$. When $|m|:=\sum m_i>-4(g-1)$, by Riemann-Roch Formula, $\mathcal{C}_{g,b}^{\leqslant m}$ is a complex vector bundle of rank $5(g-1)+|m|$ over the Teichm\"uller space $\mathcal{T}_{g,b}$. Further assume $m_i\geq 4$ if $b\geq 2$ and $m=m_1\geq 6$ if $b=1$, and let $\mathcal{C}_{g,b}^m$ be the open subset of $\mathcal{C}_{g,b}^{\leqslant m}$ formed by all $(\ac{J},\ve{\phi})$ with a pole of order exactly $m_i$ at the $i^\text{th}$ puncture. Then Theorem \ref{thm_intro2} suggests considering accordingly the subset $\mathcal{P}_{g,b}^m$ of $\mathcal{P}_{g,b}$ given by all convex $\mathbb{RP}^2$-structures $X$ on $S$ such that $X$ extends to a compact surface with boundary which contains $S_{g,b}$ as its interior and the extended convex $\mathbb{RP}^2$-structure has piecewise geodesic boundary with $m_i-3$ vertices on the $i^\text{th}$ boundary component. The theorem implies:
\begin{corollary}\label{coro_intromoduli}
Under the above assumption on $m$, the affine sphere construction induces a bijection between $\mathcal{C}_{g,b}^m$ and $\mathcal{P}_{g,b}^m$.
\end{corollary}
When $g=0$ and $b=1$, this is the Dumas-Wolf identification \cite{dumas-wolf} between the moduli spaces of polynomial cubic differentials of degree $m-6$ on $\mathbb{C}$ and convex polygons in $\mathbb{RP}^2$ with $m-3$ sides. Note however that while the bijection is shown to be a homeomorphism in \cite{dumas-wolf}, we do not treat the continuity issue here.

Finally, Theorem \ref{thm_intro3} implies a holomorphic description of \emph{framings} of convex $\mathbb{RP}^2$-structures in the sense of Fock-Goncharov \cite{fock-goncharov_1, fock-goncharov_2}, defined as follows: Given a surface $S$ which is the interior of a compact surface $\overline{S}$ with nonempty boundary and a convex $\mathbb{RP}^2$-structure $X$ on $S$ which extends to one with piecewise geodesic boundary on $\overline{S}$, with vertices $\Lambda\subset \pa \overline{S}$, a framing of $X$ is by definition a choice of a point in each connected component of $\pa \overline{S}\setminus\Lambda$ (see Definition 2.6 in \cite{fock-goncharov_2} and the remark following it). Thus, under the above assumptions on $m$, each convex $\mathbb{RP}^2$-structure $X$ in $\mathcal{P}_{g,b}^m$ can be endowed with framings. The moduli space $\widehat{\mathcal{P}}_{g,b}^m$ of all such framed $X$'s is the subject of study in \cite{fock-goncharov_1, fock-goncharov_2}. 

Now suppose $X$ corresponds to some $(\ac{J},\ve{\phi})$ in $\mathcal{C}_{g,b}^m$ under the bijection from Corollary \ref{coro_intromoduli}. By Theorem \ref{thm_intro3}, the compact surface $\overline{S}$ to which $X$ extends and the set $\Lambda\subset\pa \overline{S}$ of vertices are obtained from $(S,\ac{J},\ve{\phi})$ by performing the negative ray bordification at every puncture. Therefore, by the construction of negative ray bordification, a framing of $X$ is equivalent to a choice, for every $i$, of $m_i-3$ negative rays not parallel to each other tending to the $i^\text{th}$ puncture of $S$ (here two negative rays are considered the same if one is contained in the other). Defining such a choice as a framing of $(\ac{J},\ve{\phi})$ and letting $\widehat{\mathcal{C}}_{g,b}^m$ denote the moduli space of all pairs $(\ac{J},\ve{\phi})$ in $\mathcal{C}_{g,b}^m$ endowed with a framing, we deduce from Theorem \ref{thm_intro3}:
\begin{corollary}\label{coro_intromoduli2}
There is a natural bijection between $\widehat{\mathcal{C}}_{g,b}^m$ and $\widehat{\mathcal{P}}_{g,b}^m$ covering the bijection in Corollary \ref{coro_intromoduli}.
\end{corollary}
When $g=0$ and $b=1$, this identifies the moduli space of framed polynomial cubic differentials on $\mathbb{C}$ of degree $m-6$ with the moduli space of convex configurations of $m-3$ flags in $\mathbb{RP}^2$. A natural problem worth further exploration is to interpret the Fock-Goncharov coordinates on $\widehat{\mathcal{P}}_{g,b}^m$ given by cross-ratios and triple-ratios of flags \cite{fock-goncharov_1, fock-goncharov_2}, as well as the canonical Poisson structure defined under these coordinates, in terms of holomorphic cubic differentials.

%

\subsection*{Organization of the paper}
We first prove Theorems \ref{thm_intro1} and \ref{thm_introli} in Chapter \ref{sec_vortex}, then review backgrounds about affine spheres in Chapter \ref{sec_pre} and prove in Chapter \ref{sec_4} a generalized version of Part (\ref{item_introthm0}) of Theorem \ref{thm_intro2}. The detailed construction of the negative ray bordification is carried out in Chapter \ref{sec_5}, whereas in Chapter \ref{sec_6} we prove the asymptotic results on osculation maps and deduce Theorem \ref{thm_intro3} and the last statement of Theorem \ref{thm_intro2}. Since Theorem \ref{thm_intro3} implies the ``only if'' direction in Parts (\ref{item_introthm2}) and (\ref{item_introthm1}) of Theorem \ref{thm_intro2}, the proof of Theorem \ref{thm_intro2} is completed by Chapter \ref{sec_7}, where we establish the ``if'' direction in (\ref{item_introthm2}) and (\ref{item_introthm1}).

\subsection*{Acknowledgements}
This work is a revision of the author's unpublished paper \cite{nie_mero}. We are grateful to Mike Wolf for a helpful conversation in 2014 leading to that paper and to Qiongling Li for inspiring discussions on various occasions. We would also like to thank John Loftin for explaining to us his paper \cite{loftin_neck} and thank Alexandre Eremenko for helps with the proofs in Appendix \ref{subsec_regularity}. We acknowledge
support from U.S. National Science Foundation grants DMS 1107452, 1107263, 1107367 ``RNMS: GEometric structures And Representation varieties'' (the GEAR Network)


\section{The vortex equation}\label{sec_vortex}
Let $k\geq1$ be an integer, $\Sigma$ be a Riemann surface and $\ve{\phi}$ be a holomorphic $k$-differential, \ie a holomorphic section of $K^{\otimes k}$, where $K$ is the canonical bundle of $\Sigma$. We consider in this chapter conformal Riemannian metrics $g$ on $\Sigma$ satisfying 
\begin{equation}\label{eqn_vortex}
\kappa_g=-1+\|\ve{\phi}\|^2_g,
\end{equation}
where $\kappa_g$ is the curvature of $g$ and $\|\ve{\phi}\|_g$ is the pointwise norm of $\ve{\phi}$ with respect to the hermitian metric on $K^{\otimes k}$ induced by $g$.  
While Eq.(\ref{eqn_vortex}) is known as \emph{Wang's equation} when $k=3$, for general $k$ we call it the \emph{vortex equation} on $(\Sigma,\ve{\phi})$ as in \cite{dumas-wolf} and refer to $g$ as a \emph{solution} to the equation, which is alway smooth by elliptic regularity (see Remark \ref{remark_elliptic}). 

\subsection{Analytic tools}
We collect in this section some basic facts about Eq.(\ref{eqn_vortex}) and classical  PDE results needed in the sequel.

\subsubsection*{Scalar form of the equation.} Given a background conformal metric $g_0$ on $\Sigma$, assuming $g=e^ug_0$ for a function $u:\Sigma\to\mathbb{R}$, we have
$$
\kappa_{g}=e^{-u}\left(\kappa_{g_0}-\frac{1}{2}\Delta_{g_0}u\right),\quad \|\ve{\phi}\|^2_g=e^{-ku}\|\ve{\phi}\|^2_{g_0},
$$
where $\Delta_{g_0}$ is the Laplacian with respect to $g_0$. Therefore, $g$ solves the vortex equation for $(\Sigma,\ve{\phi})$ if and only if $u$ satisfies 
\begin{equation}\label{eqn_vortex2}
\frac{1}{2}\Delta_{g_0}u=e^u-e^{(1-k)u}\|\ve{\phi}\|^2_{g_0}+\kappa_{g_0}.
\end{equation}

In particular, if $z$ is a conformal local coordinate on $\Sigma$, taking  $g_0$ to be the locally defined flat metric $|\dz|^2$, we conclude that $g=e^u|\dz|^2$ is a solution to the vortex equation for $\ve{\phi}=\phi(z)\dz^k$ if and only if $u$ locally satisfies
\begin{equation}\label{eqn_vortex3}
\Delta u=2(e^u-e^{(1-k)u}|\phi|^2),
\end{equation}
where $\Delta=4\pazbz=\pa_x^2+\pa_y^2$ is the Euclidean Laplacian.
\begin{remark}\label{remark_elliptic}
We do not specify regularity conditions on the solutions $g$ or $u$, because it follows from basic elliptic theory (see \eg \cite[\S 6.3]{evans}) that any $H^1_\mathit{loc}$-weak solution to Eq.(\ref{eqn_vortex3}) is smooth.
\end{remark}

\subsubsection*{Method of sub- and super- solutions.} A sub- (resp. super-) solution to the vortex equation on $(\Sigma,\ve{\phi})$ is a continuous conformal metric $g$ on $\Sigma$ in the Sobolev class $H^1_\mathit{loc}$ such that for any conformal local coordinate $z$ on $\Sigma$, if we write $g=e^u|\dz|^2$ and $\ve{\phi}=\phi(z)\dz^k$, then $u$ is a sub- (resp. super-) solution to Eq.(\ref{eqn_vortex3}), \ie
$$
\Delta u \geq 2(e^u-e^{(1-k)u}|\phi|^2) \quad \left(\mbox{resp. }\Delta u \leq 2(e^u-e^{(1-k)u}|\phi|^2)\right)
$$
in the weak sense (see \cite[\S 9.3]{evans} or \cite[\S 8.5]{gilbarg-trudinger} for the definition). Thus, $g$ can be thought of as satisfying the inequality
$$
\kappa_g\leq -1+\|\ve{\phi}\|^2_g \quad \left(\mbox{resp. }\kappa_g\geq -1+\|\ve{\phi}\|^2_g\right)
$$
in the weak sense. The well known method of solving PDEs using sub- and super- solutions (see  \cite[Theorem 9]{wan}) implies:
\begin{theorem}\label{thm_subsuper}
Given a subsolution $g_-$ and  a supersolution $g_+$ to the vortex equation on $(\Sigma,\ve{\phi})$ such that $g_-\leq g_+$, there exists a solution $g$ satisfying $g_-\leq g\leq g_+$.
\end{theorem}

\subsubsection*{Comparison principle.} The conformal ratio between sub- and super- solutions to the vortex equation satisfies a strong maximum principle:
\begin{proposition}\label{prop_maximum}
For any subsolution $g_-$ and supersolution $g_+$ to the vortex equation on $(\Sigma,\ve{\phi})$, the conformal ratio $g_-/g_+$ does not admit local maxima greater than $1$. If $g_-/g_+$ achieves local maximum $1$ at a point $p\in\Sigma$, then $g_+=g_-$  in the vicinity of $p$.
\end{proposition}
This can be deduced from the following strong maximum principle for weak subsolutions to linear equations, which is a special case of \cite[Theorem 8.19]{gilbarg-trudinger}:
\begin{lemma}\label{lemma_maximumprinciple}
Let $\Omega\subset \mathbb{C}$ be a bounded domain and $w\in C^0(\overline{\Omega})\cap H^1(\Omega)$ be such that 
$\Delta w\geq c w$ in the weak sense for a constant $c\geq 0$ and $w(p)=\sup_\Omega w$
for a point $p\in\Omega$. If $c>0$, we further assume $\sup_\Omega w\geq 0$.
Then $w$ is a constant.
\end{lemma}
Note that if $c=0$ then the lemma is a basic property of subharmonic functions, whereas if $u\in C^2(\Omega)$ then it is the classical strong maximum principle of E. Hopf (see \cite[Theorem 3.5]{gilbarg-trudinger}).

\begin{proof}[Proof of Proposition \ref{prop_maximum}]
Suppose $w:=\log (g_-/g_+)$ achieves local maximum at $p\in \Sigma$. 
Let $U$ be a relatively compact neighborhood of $p$, conformally identified with a bounded domain in $\mathbb{C}$, such that $w(p)=\sup_U w$. If $w(p)>0$, we further assume $w>0$ on $U$.

Write $g_\pm=e^{u_\pm}|\dz|^2$ and $\ve{\phi}=\phi(z)\dz^k$, so that $u_\pm\in C^0(\overline{U})\cap H^1(U)$ satisfy
\begin{equation}\label{eqn_proofmaximum}
\Delta u_-\geq 2(e^{u_-}-e^{(1-k)u_-}|\phi|^2),\quad \Delta u_+\leq 2(e^{u_+}-e^{(1-k)u_+}|\phi|^2)
\end{equation}
in the weak sense. 

If $w(p)>0$,  since $w=u_--u_+>0$ on $U$ by assumption, we have
$$
e^{u_-}-e^{(1-k)u_-}|\phi|^2> e^{u_+}-e^{(1-k)u_+}|\phi|^2.
$$
Combining with (\ref{eqn_proofmaximum}), we get $\Delta w=\Delta (u_--u_+)> 0$ on $U$ in the weak sense, so $w$ is constant on $U$ by Lemma \ref{lemma_maximumprinciple}. But this implies $\Delta w=0$, a contradiction. Therefore, $w(p)>0$ can not happen, \ie $g_-/g_+$ does not admit local maxima greater than $1$.

On the other hand, if $w(p)= 0$, we have $w=u_--u_+\leq 0$ on $U$, hence
$$
0\leq(e^{u_+}-e^{(1-k)u_+}|\phi|^2)-(e^{u_-}-e^{(1-k)u_-}|\phi|^2)\leq c(u_+-u_-)
$$
for a constant $c>0$ only depending on $\sup_U|u_\pm|$ and $\sup_U |\phi|^2$. Combining with (\ref{eqn_proofmaximum}) gives $\Delta w\geq 2 cw$ on $U$ in the weak sense. Thus we can still apply Lemma \ref{lemma_maximumprinciple} and conclude that $w=0$ on $U$, or equivalently $g_+=g_-$ on $U$, as required.
\end{proof}

Proposition \ref{prop_maximum} implies a comparison principle for sub- and super- solutions:
\begin{corollary}\label{coro_compa}
Let $g_+$ and $g_-$ be a supersolution and a subsolution to the vortex equation on $(\Sigma,\ve{\phi})$, respectively, and let $C\subset \Sigma$ be a connected compact set such that $g_+\geq g_-$ on $\pa C$. Then $g_+\geq g_-$ on the whole $C$. If the equality holds at an interior point of $C$, then it holds on the whole $C$.
\end{corollary}
\begin{proof}
If $g_+<g_-$ at some interior point of $C$, then the restriction of $g_-/g_+$ to $C$ attends its absolute maximum, which is greater than $1$, in the interior of $C$. But this is impossible by Proposition \ref{prop_maximum}. Thus we have $g_+\geq g_-$ on $C$. If $g_+=g_-$ at an interior point $p$ of $C$, then $g_-/g_+$ attends local maximum $1$ at $p$, so Proposition \ref{prop_maximum} again implies that $g_+=g_-$ in a neighborhood of $p$. We conclude that $g_+=g_-$ on $C$ by connectedness.
\end{proof}

\subsubsection*{A Theorem of Cheng and Yau.}  The following boundedness criterion for functions on a complete Riemannian manifold is a particular case of Theorem 8 in \cite{cheng-yau_3}, obtained by taking $g(t)=t^\alpha$ in that theorem.
\begin{theorem}[Cheng-Yau]\label{thm_boundedness}
Let $u$ be a $C^2$-function on a complete Riemannian manifold $M$ with Ricci curvature bounded from below such that $\Delta u\geq f(u)$, where $f$ is a lower semi-continuous function on $\mathbb{R}$ satisfying
$$
\liminf_{t\rightarrow+\infty}\,t^{-\alpha}f(t)>0
$$
for some $\alpha>1$. Then $u$ is bounded from above on $M$ and we have $f(\sup u)\leq 0$.
\end{theorem}

\subsection{Solutions with nonpositive curvature}
The singular flat metric $|\ve{\phi}|^\frac{2}{k}$ underlying $\ve{\phi}$ is a solution to the vortex equation away from the zeros of $\ve{\phi}$. Clearly, for any solution $g$ to the vortex equation on $(\Sigma,\ve{\phi})$, 
we have the following equivalences:
$$
\kappa_g\leq 0\ \Longleftrightarrow\ \|\ve{\phi}\|_g^2\leq1 \ \Longleftrightarrow\  |\ve{\phi}|^\frac{2}{k}\leq g.
$$
Replacing ``$\leq$'' by ``$<$'', the three strict inequalities are also equivalent.  
The following lemma says that ``$\leq$'' implies either ``$<$'' everywhere or ``$=$'' everywhere: 
\begin{lemma}\label{lemma_nonpositive1}
Let $g$ be a solution to the vortex equation on $(\Sigma,\ve{\phi})$ with $\kappa_g\leq 0$. If $\kappa_g=0$ holds at a point, then $\ve{\phi}$ does not have zeros and the equality holds everywhere.
\end{lemma}
\begin{proof}
If the equality is achieved at a point, then the conformal ratio $|\ve{\phi}|^\frac{2}{k}/g=\|\ve{\phi}\|_g^\frac{2}{k}$ attends the absolute maximum $1$. Applying Proposition \ref{prop_maximum} to $g_+=g$ and $g_-=|\ve{\phi}|^\frac{2}{k}$ on the complement of the set of zeros of $\ve{\phi}$, we conclude that the equality holds everywhere except at the zeros. But in this case $\ve{\phi}$ can not have zeros.
\end{proof}

While previous works in the literature on the vortex equation mainly focus on complete solutions with nonpositive curvature, the following result in \cite{li_on} implies that the curvature assumption is dispensable as it follows from completeness. For applications later  on, we formulate the result for supersolutions:
\begin{lemma}\label{lemma_nonpositive2}
If $g$ is a complete $C^2$-supersolution to the vortex equation, then $\kappa_g\leq 0$.
\end{lemma}
\begin{proof}
Put $\tau:=\log(1+\|\ve{\phi}\|_g^2)$. Computing in a conformal local coordinate $z$ with $\ve{\phi}=\phi(z)\dz^k$ and $g=e^{u}|\dz|^2$, we get
\begin{align*}
\pazbz \tau=\pazbz\log\left(1+\frac{|\phi|^2}{e^{ku}}\right)&=-\frac{k|\phi|^2}{|\phi|^2+e^{ku}}\pazbz u+\frac{e^{ku}}{(|\phi|^2+e^{ku})^2}\big|\paz\phi-k\phi\,\paz u\big|^2\\
&\geq  -\frac{k|\phi|^2}{|\phi|^2+e^{ku}}\pazbz u\geq \frac{k|\phi|^2}{|\phi|^2+e^{ku}}\cdot\frac{e^u}{2}\left(1-\frac{|\phi|^2}{e^{ku}}\right)\\
&=\frac{e^u}{2}\cdot\frac{k(e^\tau-1)}{e^\tau}(2-e^\tau)=\frac{e^u}{2}\cdot k(e^\tau-3+2e^{-\tau}).
\end{align*}
So the Laplacian of $\tau$ with respect to $g$ is $\Delta_g\tau=\frac{4}{e^u}\pazbz\tau\geq 2k(e^\tau-3+2e^{-\tau})$.
Applying Theorem \ref{thm_boundedness} with $f(\tau)=2k(e^\tau-3+2e^{-\tau})$, we get $f(\sup\tau)\leq 0$, or equivalently, $1\leq e^{\sup\tau}\leq 2$, hence $\|\ve{\phi}\|^2_g\leq 1$. 
\end{proof}

\subsection{Unique existence of complete solutions}
In this section, we prove unique existence of complete solutions to the vortex equation. The uniqueness part is implied by the following lemma, which will be used again in Section \ref{subsec_upper}:
\begin{lemma}\label{lemma_completebound}
Let $g_+$ and $g_-$ be a supersolution and a subsolution of class $C^2$ to the vortex equation on $(\Sigma,\ve{\phi})$, respectively. If $g_+$ is complete, then $g_+\geq g_-$.
\end{lemma}
\begin{proof}
Suppose $g_-=e^wg_+$. By the scalar form (\ref{eqn_vortex2}) of the vortex equation, $g_-$ being a subsolution means
$$
\frac{1}{2}\Delta_{g_+}w\geq e^w-e^{(1-k)w}\|\ve{\phi}\|^2_{g_+}+\kappa_{g_+}.
$$
Since $g_+$ is a supersolution, we have $\kappa_{g_+}\geq -1+\|\ve{\phi}\|^2_{g_+}$. Therefore,
$$
\frac{1}{2}\Delta_{g_+}w\geq e^w-1+\left(1-e^{(1-k)w}\right)\|\ve{\phi}\|_{g_+}^2\,.
$$
By Lemma \ref{lemma_nonpositive2} we have $\kappa_{g_+}\leq0$, which implies $\|\ve{\phi}\|_{g_+}^2\leq 1$. Thus
$$
\frac{1}{2}\Delta_{g_+}w\geq f(w):=
\begin{cases}
e^w-1,&\mbox{ if }w\geq 0,\\
e^w-e^{(1-k)w}&\mbox{ if } w\leq 0.
\end{cases}
$$
By Theorem \ref{thm_boundedness}, $w$ is bounded from above and $f(\sup w)\leq 0$, hence $w\leq 0$, or equivalently, $g_-\leq g_+$.
\end{proof}
\begin{remark}\label{remark_yauas}
When $\ve{\phi}=0$, Lemma \ref{lemma_completebound} is a consequence of Yau's generalization \cite{yau_ahlfors} of the Ahlfors-Schwarz lemma, stating that a conformal map from a surface endowed with a complete Riemannian metric of curvature $\geq -1$ to a surface with a Riemannian metric of curvature $\leq -1$ is metric dominating. 
\end{remark}

\begin{theorem}\label{thm_ue}
Given a Riemann surface $\Sigma$ and a nontrivial holomorphic $k$-differential $\ve{\phi}$ on $\Sigma$, there exists a unique complete solution $g$ to the couple vortex equation on $(\Sigma,\ve{\phi})$. This $g$ has negative curvature unless $(\Sigma,\ve{\phi})$ is a quotient of $(\mathbb{C},\dz^k)$, in which case $g$ is the flat metric $|\dz|^2$.
\end{theorem}
Note that for the trivial $k$-differential $\ve{\phi}=0$, solutions to the vortex equation are exactly conformal hyperbolic metrics, so there is a unique complete solution if and only if $\Sigma$ is a hyperbolic Riemann surface.
\begin{proof}
The uniqueness follows immediately from Lemma \ref{lemma_completebound}.

Given a complete solution $g$, by Lemma \ref{lemma_nonpositive1} and Lemma  \ref{lemma_nonpositive2}, either $\kappa_g<0$ on $\Sigma$ or $\kappa_g\equiv0$. In the latter case, $\ve{\phi}$ does not have zeros and $g=|\ve{\phi}|^\frac{2}{k}$ is a complete flat metric, so the universal cover of $(\Sigma,g)$ is isometric to $(\mathbb{C},|\dz|^2)$. It follows that the universal cover of $(\Sigma, \ve{\phi})$ is isomorphic to $(\mathbb{C},\dz^k)$, whence the second statement of the theorem.

It remains to be proved that complete solutions exist. Since the Riemann sphere $\mathbb{CP}^1$ does not admit nontrivial holomorphic $k$-differentials, we may suppose $\Sigma=U/\Gamma$, where $U$ is either $\mathbb{C}$ or the unit disk $\mathbb{D}$ and $\Gamma$ is a discrete group of conformal automorphisms. If we have a complete solution $\tilde{g}$ to the vortex equation on $(U,\widetilde{\ve{\phi}})$ (where $\widetilde{\ve{\phi}}$ denotes the lift of $\ve{\phi}$), then the uniqueness statement proved above implies that $\tilde{g}$ is invariant by $\Gamma$, hence descends to a complete solution on $(\Sigma,\ve{\phi})$. Therefore,
it is sufficient to treat the cases $\Sigma=\mathbb{D}$ and $\Sigma=\mathbb{C}$ respectively.

\vspace{5pt}

\textbf{Case 1: $\Sigma=\mathbb{D}$.}   
Following \cite{wan}, we first consider a holomorphic cubic differential $\ve{\phi}$ on $\mathbb{D}$ whose pointwise norm with respect to the Poincar\'e metric $g_\mathbb{D}$ is bounded. This condition implies that if we take $M>1$ large enough and set $g_+:=Mg_\mathbb{D}$, we have
$$
\kappa_{g_+}=-\frac{1}{M^2}\geq -1+\frac{\|\ve{\phi}\|^2_{g_\mathbb{D}}}{M^k}=-1+\|\ve{\phi}\|^2_{g_+},
$$
\ie $g_+$ is a supersolution. Since $g_-=g_\mathbb{D}$ is a subsolution majorized by $g_+$, Theorem \ref{thm_subsuper} yields a solutions $g$ satisfying 
\begin{equation}\label{eqn_gdmgd}
g_\mathbb{D}=g_-\leq g\leq g_+=Mg_\mathbb{D},
\end{equation}
which is complete because $g_\mathbb{D}$ is. This establishes the required existence statement in the case of bounded norm.

For a general holomorphic $k$-differential $\ve{\phi}=\phi(z)\dz^k$, let $D_r\subset\mathbb{D}$ be the disk of radius $r<1$ centered at $0$ and $g_{D_r}$ be the complete conformal  hyperbolic metric on $D_r$. The pointwise norm of $\ve{\phi}|_{D_r}$ with respect to $g_{D_r}$ is bounded, so the above existence result in the bounded case yields a complete solution $g_r$ to the vortex equation on $(D_r, \ve{\phi}|_{D_r})$. Applying this to an increasing sequence of disks with radii $0<r_1<r_2<\cdots$, the sequence $(g_{r_n})$ of metrics obtained has following properties on each $D_{r_m}$ whenever $n>m$: 
\begin{itemize}
\item Bounded from below by $g_\mathbb{D}$. This is because $g_r\geq g_{D_r}$ by (\ref{eqn_gdmgd}), while
$$
g_{D_r}=\frac{4|\dz|^2}{r^2(1-r^{-2}|z|^2)^2}> \frac{4|\dz|^2}{(1-|z|^2)^2}=g_\mathbb{D}.
$$
\item Pointwise decreasing. To prove this, in view of the the comparison principle (Corollary \ref{coro_compa}), it is sufficient to show that if $r<r'$, then $g_{r}\geq g_{r'}$ on $\pa D_{r-\varepsilon}$ for $\varepsilon>0$ small enough. But this is true because $g_{r}\geq g_{D_{r}}$ and $g_{r'}\leq M\,g_{D_{r'}}$ by (\ref{eqn_gdmgd}), while it follows from the above expression of $g_{D_r}$ that
$$
\lim_{\varepsilon\rightarrow 0^+}\frac{g_{D_r}}{g_{D_{r'}}}\Big|_{\pa D_{r-\varepsilon}}
=+\infty.
$$
\end{itemize}

As a result, there is a (a priori not necessarily continuous) Riemannian metric $g=e^u|\dz|^2\geq g_\mathbb{D}$ on $\mathbb{D}$ such that $(g_{r_n})_{n>m}$ pointwise converges to $g$ on $D_{r_m}$ for every $m$. Write $g_{r_n}=e^{u_n}|\dz|^2$, so that $u_n$ pointwise converges to $u$ and satisfies 
$$
\Delta u_n=2\left(e^{u_n}-e^{(1-k)u_n}|\phi|^2\right)
$$
on $D_{r_n}$ (see Eq.(\ref{eqn_vortex3})). We claim that $(u_n)_{n>m}$ has a uniform $C^{2,\alpha}$-bound on $D_{r_m}$ for some $0<\alpha<1$.

To prove that claim, we fix $p>2$ and use the $L^p$-estimate for second order elliptic equations (\cite[Theorem 9.11]{gilbarg-trudinger}) to get a uniform $W^{2,p}$-bound for $(u_n)_{n>m}$ on $D_r$ for $r_m<r<r_{m+1}$. The Sobolev embedding (\cite[Corollary 7.11]{gilbarg-trudinger}) then yields a uniform $C^1$-bound on $D_r$. Finally the Schauder estimate (\cite[Theorem 6.2]{gilbarg-trudinger})  provides the required uniform $C^{2,\alpha}$-bound on $D_{r_m}$.

It follows from the claim that any subsequence of $(u_n)_{n>m}$ has a further subsequence converging to $u$ in $C^2(D_{r_m})$. Therefore, $(u_n)_{n>m}$ itself converges to $u$ in $C^2(D_{r_m})$, and thus $u$ satisfies Eq.(\ref{eqn_vortex3}) on $D_{r_m}$. Since $m$ is arbitrary, $u$ satisfies Eq.(\ref{eqn_vortex3}) on the whole $\mathbb{D}$, or equivalently, $g=e^u|\dz|^2$ is a solution to the vortex equation on $(\mathbb{D},\ve{\phi})$. Since $g\geq g_\mathbb{D}$ by construction, $g$ is complete. This finishes the construction of complete solution to the vortex equation when $\Sigma$ is hyperbolic.

\vspace{7pt}

\textbf{Case 2: $\Sigma=\mathbb{C}$.}  We construct sub- and super- solutions following \cite{au-wan}. A more general version of the construction, also used in the next section, can be stated as follows:
\begin{lemma}\label{lemma_asolution}
Let $\Sigma$ be a Riemann surface and $\ve{\phi}$ be a holomorphic $k$-differential on $\Sigma$. Let $U, F\subset\Sigma$ be open subsets such that $U\cup F=\Sigma$. Assume that the boundary $\pa F$ of $F$ in $\Sigma$ is compact and the closure $\overline{F}$ does not contain zeros of $\ve{\phi}$. Let $g_U$ be a complete solution to the vortex equation on $(U,\ve{\phi}|_U)$.
Then there exists a solution $g$ to the vortex equation on $(\Sigma,\ve{\phi})$ such that
$$
\max(C^{-1}g_U,|\ve{\phi}|^\frac{2}{k}) \leq g\leq g_U \mbox{ on } \Sigma\setminus F,
$$
$$
|\ve{\phi}|^\frac{2}{k}\leq g\leq C |\ve{\phi}|^\frac{2}{k} \mbox{ on } \Sigma\setminus U
$$
for some constant $C>1$.
\end{lemma}
\begin{proof}[Proof of Lemma \ref{lemma_asolution}]
We will construct sub- and super- solutions to the vortex equation on $(\Sigma,\ve{\phi})$ making use of the following basic facts:
\begin{itemize}
\item If $g'$ is a nonpositively curved solution and $C$ is a positive constant, then $C g'$ is a super- (resp. sub-) solution if $C\geq 1$ (resp. $C\leq 1)$;
\item The pointwise minimum (resp. maximum) of two super- (resp. sub-) solutions is a super- (resp. sub-) solution.
\end{itemize}
These follow from the definition of sub- and super- solutions, taking account of the fact that if $f_1,f_2\in C^0(\Omega)\cap H^1_\mathit{loc}(\Omega)$ (where $\Omega\subset\mathbb{C}$ is an open set), then the pointwise maximum $\max(f_1,f_2)=\frac{1}{2}\left(f_1+f_2+|f_1-f_2|\right)$ also belongs to $C^0(\Omega)\cap H^1_\mathit{loc}(\Omega)$ (see \cite{gilbarg-trudinger} Lemma 7.6), similarly for the pointwise minimum.

Since $\pa F$ is compact and disjoint from the zeros of $\ve{\phi}$, the ratio $g_U/|\ve{\phi}|^\frac{2}{k}$ is bounded from above by a constant $C>1$ on a neighborhood $V$ of $\pa F$. It follows that 
$\max(C^{-1}g_U,|\ve{\phi}|^\frac{2}{k})=|\ve{\phi}|^\frac{2}{k}$ on $V$. Therefore, by the above basic facts, 
$$
g_-:=
\begin{cases}
\max(C^{-1}g_U, |\ve{\phi}|^\frac{2}{k})&\mbox{  on  }\Sigma\setminus F\\
|\ve{\phi}|^\frac{2}{k}&\mbox{  on  }F
\end{cases}
$$
is continuous and is a subsolution on each of the open sets $F$, $V$ and $\Sigma\setminus\overline{F}$, hence a subsolution on $\Sigma$. 

On the other hand, since $g_U$ is a complete metric on $U$, the ratio $g_U/|\ve{\phi}|^\frac{2}{k}$ goes to $+\infty$ as one approaches a point in $\pa U$. As a result, $\min(g_U, C|\ve{\phi}|^\frac{2}{k})=C|\ve{\phi}|^\frac{2}{k}$ in a neighborhood $W$ of $\pa U$, while $\min(g_U, C|\ve{\phi}|^\frac{2}{k})=g_U$ on the neighborhood  $V$ of $\pa F$ by the earlier construction. Therefore
$$
g_+:=
\begin{cases}
g_U &\mbox{ on } \Sigma\setminus F\\ 
\min(g_U, C|\ve{\phi}|^\frac{2}{k}) &\mbox{ on } F\cap U\\
C|\ve{\phi}|^\frac{2}{k} &\mbox{ on }\Sigma\setminus U
\end{cases}
$$
is continuous and is a supersolution on each of the open sets $\Sigma\setminus\overline{F}$, $V$, $F\cap U$, $W$ and $\Sigma\setminus\overline{U}$, hence a supersolution on $\Sigma$. Moreover, since $g_U\geq |\ve{\phi}|^\frac{2}{k}$ by  Lemma \ref{lemma_nonpositive2}, one readily checks that $g_+\geq g_-$. Therefore, Theorem \ref{thm_subsuper} applies and yields the required solution.

\end{proof}

When $\Sigma=\mathbb{C}$, we let $F\subset\mathbb{C}$ be an open disk whose closure does not contain zeros of $\ve{\phi}$ and let $U\subset\mathbb{C}$ be the complement of a closed disk contained in $F$. Since $U$ is a hyperbolic Riemann surface, the above existence result in the hyperbolic case yields the solution $g_U$ needed in the hypothesis of Lemma \ref{lemma_asolution}. The lemma then provides a solution $g$ to the vortex equation on $(\mathbb{C},\ve{\phi})$, which is complete because it is minorized by a constant multiple of $g_U$ outside $F$ while $g_U$ is complete. This establishes the existence of complete solution in the case $\Sigma=\mathbb{C}$ and finishes the proof of the theorem.
\end{proof}

%

\subsection{(Non-)existence of incomplete solutions}\label{subsec_nonexistence}
The following theorem is a more detailed statement of Theorem \ref{thm_introli}. It shows that for certain $(\Sigma,\ve{\phi})$, the completeness assumption in Theorem \ref{thm_ue} can be dropped, resulting in unconditional uniqueness. The case $\Sigma=\mathbb{C}$ is proved by Q. Li \cite{li_on}. 

\begin{theorem}\label{thm_incomplete}
Let $\Sigma$ be a Riemann surface and $\ve{\phi}$ be a holomorphic $k$-differential on $\Sigma$ with finitely many zeros.  Then the following statements are equivalent:
\begin{enumerate}[(a)]
\item\label{item_incomplete1}
The vortex equation on $(\Sigma,\ve{\phi})$ does not admit incomplete solutions.
\item\label{item_incomplete2}
The singular flat metic $|\ve{\phi}|^\frac{2}{k}$  is complete.
\item\label{item_incomplete3}
$\Sigma=\overline{\Sigma}\setminus P$ for a closed Riemann surface $\overline{\Sigma}$ and a finite (possibly empty) subset $P\subset\overline{\Sigma}$, such that every $p\in P$ is a pole of $\ve{\phi}$ of order at least $k$.
\end{enumerate}
\end{theorem}
\begin{proof}
(\ref{item_incomplete1})$\Rightarrow$(\ref{item_incomplete2}): Assuming $|\ve{\phi}|^\frac{2}{k}$ is incomplete, we shall show that the vortex equation on $(\Sigma,\ve{\phi})$ has an incomplete solution. Take a relatively compact neighborhood $V$ of the set of zeros of $\ve{\phi}$ and put $F:=\Sigma\setminus\overline{V}$. Let $U$ be a relatively compact neighborhood of $\overline{V}$. Then $U$ and $F$ satisfy the assumptions in Lemma \ref{lemma_asolution}, while Theorem \ref{thm_ue} provides the complete solution $g_U$ needed in the lemma. The solution $g$ given by the lemma is incomplete, as required,  because $g$ is majorized by $C|\ve{\phi}|^\frac{2}{k}$ outside the compact set $\overline{U}$ while $|\ve{\phi}|^\frac{2}{k}$ is incomplete.  

(\ref{item_incomplete2})$\Rightarrow$(\ref{item_incomplete1}):
Assume $|\ve{\phi}|^\frac{2}{k}$  is complete and let $g_0$ be a smooth conformal metric on $\Sigma$ such that $g_0=|\ve{\phi}|^\frac{2}{k}$ outside a relatively compact neighborhood $V$ of the set of zeros of $\ve{\phi}$. Then $g_0$ is complete and has curvature bounded from below. Let $g=e^ug_0$ be any solution to the vortex equation on $(\Sigma,\ve{\phi})$. Then $u$ satisfies Eq.(\ref{eqn_vortex2}), or equalently,
$$
\frac{1}{2}\Delta_{g_0}(-u)=e^{(k-1)(-u)}\|\ve{\phi}\|^2_{g_0}-e^{-(-u)}-\kappa_{g_0}.
$$
Outside $U$, the right hand side is $e^{(k-1)(-u)}-e^{-(-u)}$. Therefore, if we let $t_0$ and $f_0$ be the maximum of $-u$ and the minimum of $\frac{1}{2}\Delta_{g_0}(-u)$ on $\overline{U}$, respectively, and put
$$
f(t):=
\begin{cases}
f_0 &\mbox{ if } t\leq t_0 \\[5pt]
e^{(k-1)t}-e^{-t}& \mbox{ if }  t> t_0
\end{cases}
$$
then we have $\frac{1}{2}\Delta_{g_0}(-u)\geq f(-u)$.
By Theorem \ref{thm_boundedness}, $-u$ is bounded from above by a positive constant $M$. Thus $g\geq e^{-M}g_0$ is complete as $g_0$ is.

(\ref{item_incomplete2})$\Rightarrow$(\ref{item_incomplete3}): The complete conformal metric $g_0$ constructed above has finite total curvature, so $\Sigma$ is a punctured Riemann surface of finite type by \cite[Theorem 1.1]{hulin-troyanov}. By Lemma \ref{lemma_huber}, each puncture is pole of order at least $k$.
 
(\ref{item_incomplete3})$\Rightarrow$(\ref{item_incomplete2}): Conformally identifying a closed neighborhood of each  $p\in P$ with the punctured disk $\{z\in\mathbb{C}\mid |z|\leq 1\}$ such that $p$ corresponds to $z=0$, it is elementary to check that Condition (\ref{item_incomplete3}) implies $|\ve{\phi}|^\frac{2}{k}$ is complete on the punctured disk, hence also complete on the whole $\Sigma$.
\end{proof}

\subsection{Upper bound at the center of a disk}\label{subsec_upper}

Every point of $\Sigma$ that is not a zero of $\ve{\phi}$ has a neighborhood which can be conformally identified with a disk
$$
D_r:=\{z\in\mathbb{C}\mid |z|<r\}
$$
in such a way that $\ve{\phi}$ restricts to $\dz^k$. A solution $g$ to the vortex equation on $(\Sigma,\ve{\phi})$ can be written in the disk as $g=e^u|\dz|^2$, with $u:D_r\to\mathbb{R}$ satisfying
\begin{equation}\label{eqn_simplevortex}
\Delta u=2(e^u-e^{(1-k)u}).
\end{equation}
The curvature $\kappa_g=-\frac{1}{2e^u}\Delta u$ is nonpositive in the disk if and only if $u$ is nonnegative. In this section, we establish upper bounds for the value of any such $u$ at the center of the disk in terms of the radius $r$.
\begin{proposition}[\textbf{Coarse bound}]
\label{prop_coarse}
Given $k\geq1$, there is a constant $C>0$ such that for any nonnegative $u$ satisfying Eq.(\ref{eqn_simplevortex}) on the disk $D_r$ with radius $r\geq1$, we have
$$
u(0)\leq \frac{C}{r^2}.
$$
\end{proposition}
\begin{proof}
Given $t\geq 0$, we let $\xi_k(t)$ denote the unique $\xi\geq 1$ such that $\xi^k-\xi^{k-1}=t$. In other words, $t\mapsto\xi_k(t)$ is the inverse of the strictly increasing function
$\xi\mapsto \xi^k-\xi^{k-1}$ defined on $\xi\in[1,+\infty)$. The definition implies $\xi_k(t)\geq t^\frac{1}{k}$ and 
\begin{equation}\label{eqn_xik}
\xi_k(t)-t^\frac{1}{k}=\xi_k(t)\left[1-\left(1-\frac{1}{\xi_k(t)}\right)^\frac{1}{k}\right]=\frac{1}{k}+O\big(\xi_k(t)^{-1}\big) \mbox{ as }  t\rightarrow+\infty.
\end{equation}

On the other hand, the complete conformal hyperbolic metric on $D_r$ is
$$
g_{D_r}:=\frac{4r^2|\dz|^2}{(r^2-|z|^2)^2}.
$$
Denoting $M:=\xi_k((r/2)^{2k})$, we claim that $g'=Mg_{D_r}$ is a supersolution to the vortex equation on $(D_r, \dz^k)$. Indeed, since the pointwise norm $\|\dz^k\|^2_{g_{D_r}}$ attends its maximum $(r/2)^{2k}$ at $z=0$, we have $\sup_{D_r}\|\dz^k\|^2_{g'}=(r/2)^{2k}/M^k$, so 
$$
\kappa_{g'}=-\frac{1}{M}=-1+\frac{(r/2)^{2k}}{M^k}= -1+\sup_{D_r}\|\dz^k\|^2_{g'}\geq -1+\|\dz^k\|^2_{g'},
$$
whence the claim. 

Eq.(\ref{eqn_simplevortex}) is the vortex equation for $g=e^{u(z)}|\dz|^2$ on $(D_r,\dz^k)$, so the above claim and Lemma \ref{lemma_completebound} implies $g\leq g'$. At $z=0$, this gives 
$$
u(0)\leq \log \frac{\xi_k((r/2)^{2k})}{(r/2)^2}\sim\frac{4}{kr^2} \mbox{ as }r\rightarrow+\infty
$$
 (``$f_1(x)\sim f_2(x)$ as $x\rightarrow x_0$'' means $\lim_{x\rightarrow x_0}\frac{f_1(x)}{f_2(x)}\rightarrow 1$), hence the required inequality, where the last asymptotic equivalence is obtained through
$$
\log\frac{\xi_k(t)}{t^\frac{1}{k}}=\frac{1}{k}\log\left(1+\frac{\xi_k(t)^{k-1}}{t}\right)\sim\frac{\xi_k(t)^{k-1}}{kt}\sim \frac{1}{k}t^{-\frac{1}{k}} \mbox{ as }t\rightarrow+\infty,
$$
using the relation $\xi_k(t)^k-\xi_k(t)^{k-1}=t$ and the asymptotic equivalence
$\xi_k(t)\sim t^\frac{1}{k}$ implied by Eq.(\ref{eqn_xik}).
\end{proof}

Proposition \ref{prop_coarse} can be improved to the following theorem, where the upper bound decays exponentially with respect to $r$:
\begin{theorem}[\textbf{Fine bound}]
\label{thm_diskfine}
Given $k\geq1$, there are constants $C, r_0>0$ such that for any nonnegative solution $u$ to Eq.(\ref{eqn_simplevortex}) on the disk $D_r$ with radius $r\geq r_0$, we have
$$
u(0)\leq Cr^\frac{1}{2}e^{-\sqrt{2k}\,r}.
$$
\end{theorem}
See Remark \ref{remark_stronger} below for a further improvement around poles of order $>3$.
\begin{proof}
We first give a proof for all $k\geq 3$ and then point out the modifications needed for the simpler cases $k=1,2$. The idea is to construct supersolutions based on the following observation (\cite[Lemma 5.9]{dumas-wolf}): If a $C^2$-function $h$ satisfies
\begin{equation}\label{eqn_h2k}
\Delta h=2kh
\end{equation}
and has the bounds $0\leq h\leq \frac{2}{k-2}$, then $v:=h-\frac{k-2}{2}h^2\geq0$ is a supersolution to Eq.(\ref{eqn_simplevortex}). This is an immediate consequence of the following facts, which one can check by computations:
\begin{itemize}
\item 
For any $u\geq0$ we have $e^u-e^{(1-k)u}\geq ku-\frac{k(k-2)}{2}u^2$.
Therefore, every nonnegative supersolution to the equation
\begin{equation}\label{eqn_u2k}
\Delta u=2ku-k(k-2)u^2
\end{equation}
is a supersolution to Eq.(\ref{eqn_simplevortex}). 
\item 
If $h$ satisfies Eq.(\ref{eqn_h2k}), then $v:=h-\frac{k-2}{2}h^2$ satisfies
$$
\Delta v-2kv+k(k-2)v^2=-k(k-2)^2\tau^3+\frac{1}{4}k(k-2)^3h^4-(k-2)|\nabla h|^2.
$$ 
If furthermore $0\leq h\leq \frac{4}{k-2}$, then the right hand side above is nonpositive, hence $v$ is a supersolution to Eq.(\ref{eqn_u2k}).
\end{itemize}

We consider \emph{rotationally symmetric} solution to Eq.(\ref{eqn_h2k}). They are given by the modified Bessel function of the second kind $I_0(x):=\frac{1}{\pi}\int_0^\pi e^{x\cos\theta}\dif\theta$, which can be characterized as the even function on $\mathbb{R}$ satisfying
\begin{equation}\label{eqn_i0}
I_0''(x)+\frac{1}{x}I_0'(x)=I_0(x), \quad I_0(0)=1.
\end{equation}
Note that $I_0$ is increasing on $[0,+\infty)$. Given $R>0$, we set
$$
h_R(z):=\frac{I_0(\sqrt{2k}|z|)}{(k-2)I_0(\sqrt{2k}R)},\quad v_R:=h_R-\frac{k-2}{2}h_R^2.
$$ 
Then $h_R$ satisfies Eq.(\ref{eqn_h2k}) because of (\ref{eqn_i0}), and has the bounds $0\leq h_R\leq\frac{1}{k-2}$ on $D_R$. Therefore, $v_R$ is a supersolution to Eq.(\ref{eqn_simplevortex}) on $D_R$, with boundary value $v_R|_{\pa D_R}\equiv\frac{1}{2(k-2)}$.

By Proposition \ref{prop_coarse} and the translation invariance of Eq.(\ref{eqn_simplevortex}), we can let $r_1>0$ be sufficiently large such that $u(z_0)\leq \frac{1}{2(k-2)}$ for any $C^2$-function $u$ satisfies Eq.(\ref{eqn_simplevortex}) on $D_{r_1}(z_0):=\{z\in\mathbb{C}\mid |z-z_0|<r_1\}$. As a result, if $u$ satisfies Eq.(\ref{eqn_simplevortex}) on $D_{r}$ with $r>r_1$ then $u\leq\frac{1}{2(k-2)}$ holds on $D_{r-r_1}\subset D_r$. We can then apply Comparison Principle (Corollary \ref{coro_compa}) to the solution $e^u|\dz|^2$ and the above supersolution $e^{v_R}|\dz|^2$ (with $R=r-r_1$) of the vortex equation on $(D_{r-r_1},\dz^k)$ and get
\begin{equation}\label{eqn_prooffine}
u(0)\leq v_R(0)=\frac{1}{(k-2)I_0\big(\sqrt{2k}(r-r_1)\big)}\left(1-\frac{1}{2I_0\big(\sqrt{2k}(r-r_1)\big)}\right).
\end{equation}
It is well known (by the stationary phase method) that
$I_0(x)=\frac{1}{\sqrt{2\pi x}}e^{x} (1+O(x^{-1}))$ as $x\rightarrow+\infty$. So for any $x_1>0$ there is a constant $C_1>0$ such that $I_0(x)\geq C_1e^x/\sqrt{x}$ whenever $x\geq x_1$. As a result, there are constants $C_2,C_3$ such that 
$$
\frac{1}{I_0(\sqrt{2k}(r-r_1))}\leq C_2(r-r_1)^\frac{1}{2}e^{-\sqrt{2k}(r-r_1)}\leq C_3r^\frac{1}{2}e^{-\sqrt{2k}\,r}
$$
when $r-r_1\geq 1$. Combining with (\ref{eqn_prooffine}) and noting that $I_0\geq1$, we obtain the required inequality in the case $k\geq 3$.

The cases $k=1,2$ are simpler because we have $e^u-e^{(1-k)u}\geq ku$ for $u\geq0$, hence any solution $h$ to Eq.(\ref{eqn_h2k}) is already a supersolution to Eq.(\ref{eqn_simplevortex}). Thus, we can proceed as above with $v_R(z):=I_0(\sqrt{2k}R)^{-1}I_0(\sqrt{2k}|z|)$.
\end{proof}


\section{Affine spheres, cubic differentials and $\mathbb{RP}^2$-structures}
\label{sec_pre}
In this chapter, we briefly review backgrounds on affine spheres and $\mathbb{RP}^2$-structures. See \cite{nomizu, loftin_survey} for thorough surveys.
\subsection{Affine spheres}\label{subsec_affine}
We view $\mathbb{R}^3$ as a vector space endowed with a volume form, say, the determinant $\det:\bigwedge^3\mathbb{R}^3\rightarrow\mathbb{R}$, and study properties of surfaces in $\mathbb{R}^3$ invariant under the group $\SL(3,\mathbb{R})$ of linear transformations preserving $\det$.

An embedded surface $M\subset\mathbb{R}^3\setminus\{0\}$ is said to be \emph{centro-affine} if for each $p\in M$ the position vector $\overrightarrow{0p}$ is transversal to the tangent plane $\T_pM$. Such an $M$ carries the following intrinsic data in an $\SL(3,\mathbb{R})$-equivariant way:
\begin{itemize} 
\item
a volume form $\omega$, called the \emph{induced volume form};
\item
a symmetric $2$-tensor $g$, called the \emph{centro-affine second fundamental form};
\item
a torsion-free affine connection $\nabla$, called the \emph{induced affine connection}.
\end{itemize}
They are determined by the following equalities, at every $p\in M$:
$$
\omega(X,Y)=\det(X,Y,\overrightarrow{0p}),\quad \D_XY=\nabla_XY+g(X,Y)\overrightarrow{0p},
$$ 
where $X$ and $Y$ are any vector fields tangent to $M$. We always equip $M$ with the orientation given by the induced volume form $\omega$.  

A centro-affine surface $M\subset\mathbb{R}^3\setminus\{0\}$ is locally convex with the origin $0$ lying on the concave side if and only if the centro-affine second fundamental form $g$ of $M$ is a Riemann metric. $M$ is called an \emph{hyperbolic affine sphere centered at $0$}, or simply an \emph{affine sphere}, if furthermore the volume form $\dif\vol_g$ of $g$ coincides with the induced volume form $\omega$. The affine second fundamental form $g$ is known as the \emph{Blaschke metric} in this case. 
\begin{remark}
The condition $\dif\vol_g=\omega$ signifies that the position vector $\overrightarrow{0p}$ for each $p\in M$ coincides with the \emph{affine normal} of $M$ at $p$.
\end{remark}

A centro-affine immersion of an oriented surface $S$ into $\mathbb{R}^3$ induces intrinsic data $(\nabla,g,\omega)$ on $S$. In the same way as how immersions of a surface into the Euclidean $3$-space $\mathbb{E}^3$ are locally determined up to isometries of $\mathbb{E}^3$ by a pair of $2$-tensors (the first and second fundamental form) satisfying a compatibility condition (Gauss and Codazzi equations), centro-affine immersions are locally determined up to $\SL(3,\mathbb{R})$ by these intrinsic data, with a compatibility condition.

For affine spherical immersions, the intrinsic data with compatibility condition are equivalent to the data $(g,\ve{\phi})$ consisting of the Blaschke metric $g$ and a holomorphic cubic differential $\ve{\phi}$ (with respect to the complex structure given by $g$ and the orientation) satisfying the vortex equation $\kappa_g=-1+\|\ve{\phi}\|^2_g$, also known as \emph{Wang's equation} in this setting. In fact, $\nabla$ is given by $g$ and $\ve{\phi}$ through
\begin{equation}\label{eqn_pickdef}
g(\nabla_XY-\widehat{\nabla}_XY,Z)=\tfrac{1}{\sqrt{2}}\re \ve{\phi}(X,Y,Z),
\end{equation}
where $\widehat{\nabla}$ denotes the Levi-Civita connection of $g$. We call $\ve{\phi}$ the \emph{normalized Pick differential} of the affine sphere.
\begin{remark}\label{remark_normalized}
More commonly considered in the literature is the \emph{Pick differential} $\ve{\psi}=\frac{1}{\sqrt{2}}\ve{\phi}$, for which the relation (\ref{eqn_pickdef}) does not have the $\frac{1}{\sqrt{2}}$ fact, but Wang's equation becomes $\kappa_g=-1+2\|\ve{\psi}\|^2_g$. The advantage of the equation we use is that the flat metric $g=|\ve{\phi}|^\frac{2}{3}=|\phi(z)|^\frac{2}{3}|\dz|^2$ provides a solution away from the zeros of $\ve{\phi}$.
\end{remark}

Given an oriented surface $S$, we let $\underline{\mathbb{R}}=S\times \mathbb{R}$ denote the trivial line bundle over $S$ endowed with a global section, denoted by $\underline{1}$, and consider the rank $3$ vector bundle $E:=\T S\oplus\underline{\mathbb{R}}$.
For a pair $(g,\ve{\phi})$ on $S$ as above, affine spherical immersions with Blaschke metric $g$ and normalized Pick differential $\ve{\phi}$ can be reconstructed by integrating a specific flat connection $\D=\D_{g,\ve{\phi}}$ on $E$, defined as follows. In a conformal local coordinate $z$, we write $g=e^u|\dz|^2$, $\ve{\phi}=\phi(z)\dz^3$ and define $\D$ by specifying its matrix expression under the local frame $(\paz,\pabz,\underline{1})$ of the complexified vector bundle $E\otimes\mathbb{C}=\T_\mathbb{C}S\oplus\underline{\mathbb{C}}$ 
as
\begin{equation}\label{eqn_Dlocal}
\D=\dif+
\begin{pmatrix}
\pa u&\frac{1}{\sqrt{2}}e^{-u}\bar\phi\,\dbz&\dz\\[6pt]
\frac{1}{\sqrt{2}}e^{-u}\phi\,\dz&\bpa u&\dbz\\[6pt]
\frac{1}{2}e^{u}\dbz&\frac{1}{2}e^{u}\dz&0
\end{pmatrix}
\end{equation}
(here $\D$ is viewed as a connection on $E_\mathbb{C}$ by $\mathbb{C}$-linear extension). It can be checked that $\D$ is coordinate-independent, flat and preserves the volume form $\dif\vol_g\wedge\underline{1}^*$ on $E$ (where $\underline{1}^*$ is the section of $E^*$ dual to $\underline{1}$). We call $\D$ the \emph{affine sphere connection} associated to $(g,\ve{\phi})$. 

%

The reconstruction of affine spherical immersion from $(g,\ve{\phi})$ (see \eg \cite[\S 4.2]{dumas-wolf}, \cite[\S 3.4]{loftin_neck}) can be formulated as follows:
\begin{proposition}[\textbf{Wang's developing map}]
\label{prop_recon}
Let $U\subset\mathbb{R}^2$ be a simply connected domain. Then the assignment of the Blaschke metric $g$ and normalized Pick differential $\ve{\phi}$ to each affine spherical immersion $\iota:U\rightarrow\mathbb{R}^3$ gives a one-to-one correspondence between $\SL(3,\mathbb{R})$-equivalence classes of such immersions and pairs $(g,\ve{\phi})$ on $U$ satisfying Wang's equation. Given such a pair $(g,\ve{\phi})$, let $\para(z_2,z_1):\T_{z_1} U\oplus\mathbb{R}\rightarrow\T_{z_2} U\oplus\mathbb{R}$ denote the parallel transport of the affine sphere connection associated to $(g,\ve{\phi})$ along any path from $z_1$ to $z_2$ (see Appendix \ref{subsec_connection}), then a corresponding $\iota$ is given by
\begin{equation}\label{eqn_Ddevelop}
\iota:U\rightarrow\T_{z_0}U\oplus\mathbb{R}\cong\mathbb{R}^3,\ \ \iota(x)=\para(z_0,z)\underline{1}_z,
\end{equation}
where $z_0$ is any point in $U$ and the identification $\T_{z_0}U\oplus\mathbb{R}\cong\mathbb{R}^3$ is such that $\dif\vol_g\wedge\underline{1}^*$ corresponds to the volume form $\det$ on $\mathbb{R}^3$.
\end{proposition}
If the surface $S$ is not simply connected, applying Proposition \ref{prop_recon} to the universal cover $\widetilde{S}$ gives a one-to-one correspondence between pairs $(g,\ve{\phi})$ on $S$ satisfying Wang's equation and affine spherical immersions $\iota:\widetilde{S}\rightarrow\mathbb{R}^3$ that are \emph{equivariant} with respect to a representation $\pi_1(S)\rightarrow\SL(3,\mathbb{R})$.

\subsection{$\mathbb{RP}^2$-structures and convexity}\label{subsec_convexity}
Given a pair $(g,\ve{\phi})$ satisfying Wang's equation on a surface $S$,
the corresponding $\iota$ from the last paragraph, composed with the projection $\mathbb{P}:\mathbb{R}^3\setminus\{0\}\rightarrow\mathbb{RP}^2$, is an equivariant local diffeomorphism from $\widetilde{S}$ to $\mathbb{RP}^2$, hence defines an $\mathbb{RP}^2$-structure on $S$. Denote this $\mathbb{RP}^2$-structure by $\X{g,\ve{\phi}}$.

A. M. Li \cite{li_calabi1, li_calabi2} clarified an earlier proof of Cheng-Yau \cite{cheng-yau_1, cheng-yau_4} to show that under the settings of Proposition \ref{prop_recon}, $\iota$ is a proper embedding if and only if $g$ is complete, in which case $\iota(U)$ is inscribed in a convex cone $C\subset\mathbb{R}^3$ (like a hyperboloid inscribed in a quadratic cone) and $\mathbb{P}$ gives a diffeomorphism from $\iota(U)$ to the properly convex domain $\Omega=\mathbb{P}(C\setminus\{0\})$. 
As a consequence, the above $\mathbb{RP}^2$-structure $\X{g,\ve{\phi}}$ is convex if $g$ is complete. 

But Theorem \ref{thm_ue} yields a pair $(g,\ve{\phi})$ with complete $g$ once $\ve{\phi}$ is given. Therefore, letting $\mathscr{C}(S)$ denote the space of pairs $(\ac{J},\ve{\phi})$, where $\ac{J}$ is a $C^\infty$-complex structure on $S$ compatible with the orientation and $\ve{\phi}$ a holomorphic cubic differential on $(S,\ve{J})$ such that either $(S,\ac{J})$ is hyperbolic or $\ve{\phi}$ is nontrivial, we have a natural map from $\mathscr{C}(S)$ to the space $\mathscr{P}(S)$ of $C^\infty$-convex $\mathbb{RP}^2$-structures on $S$:
\begin{equation}\label{equation_mapcs}
\mathscr{C}(S)\rightarrow\mathscr{P}(S),\quad (\ac{J},\ve{\phi})\mapsto \X{\gjphi, \ve{\phi}},
\end{equation}
where $\gjphi$ denotes the complete solution to Wang's equation on $(S,\ac{J},\ve{\phi})$ given by Theorem \ref{thm_ue}. We proceed to explain that this map is bijective.

\subsection{Unique existence of complete affine spheres}\label{subsec_yau}
Given a bounded convex domain $\Omega\subset\mathbb{R}^2\subset\mathbb{RP}^2$  and a negative valued convex function $u$ on $\Omega$, the surface
$$
M:=\left\{-\frac{1}{u(x,y)}
(x,y, 1)\,\Big|\,
(x,y)\in \Omega\right\}\subset\mathbb{R}^3
$$ 
is convex, centro-affine and projects to $\Omega$ in an orientation preserving manner (where $M$ carries a natural orientation as a centro-affine surface). It turns out that $M$ is an affine sphere if and only if $u$ satisfies the Mong-Amp\`ere equation
\begin{equation}\label{eqn_mongeampere}
\det(\Hess(u))=u^{-4}.
\end{equation}
Moreover, $u$ has vanishing boundary values if and only if $M$ is properly embedded. The follows theorem shows that such $u$ uniquely exists, hence implies there is a unique affine sphere projecting to $\Omega$:
\begin{theorem}[\cite{cheng-yau_1}]
\label{thm_chengyauaffinesphere}
For every bounded convex domain $\Omega\subset\mathbb{R}^2$, there exists a unique convex function $u\in C^0(\overline{\Omega})\cap C^\infty(\Omega)$ satisfying Eq.(\ref{eqn_mongeampere}) in $\Omega$ with $u|_{\pa\Omega}=0$. As a consequence, for every properly convex domain $\Omega\subset\mathbb{RP}^2$ endowed with an orientation, there exists a unique complete affine sphere $M_\Omega\subset\mathbb{R}^3$ projectivizing to $\Omega$ in an orientation preserving manner.
\end{theorem}
As a result, every oriented properly convex domain $\Omega\subset\mathbb{RP}^2$ carries a complete metric $g_\Omega$ and a holomorphic cubic differential $\ve{\phi}_\Omega$ satisfying Wang's equation, given by pushing forward the Blaschke metric and normalized Pick differential on $M_\Omega$ through the projection $\mathbb{P}:M_\Omega\overset\sim\rightarrow\Omega$. The uniqueness part of the above theorem implies that $g_\Omega$ and $\ve{\phi}_\Omega$ are $\SL(3,\mathbb{R})$-equivariant in the sense that $a^*g_\Omega=g_{a(\Omega)}$ for all $a\in\SL(3,\mathbb{R})$, and similarly for $\ve{\phi}_\Omega$.

Let $S$ be an oriented surface. Given a convex $\mathbb{RP}^2$-structure $X$ with developing map $\delta:\widetilde{S}\overset\sim\rightarrow\Omega\subset\mathbb{RP}^2$, the pullbacks $\delta^*g_\Omega$ and $\delta^*\ve{\phi}_\Omega$, which are just the Blaschke metric and normalized Pick differential of the unique affine spherical immersion $\iota$ satisfying $\delta=\mathbb{P}\circ\iota$, are invariant under deck transformations, hence can be viewed as defined on $S$. We simply call them the Blaschke metric and normalized Pick differential of the convex $\mathbb{RP}^2$-structure $X$. Clearly, the map $\mathscr{P}(S)\rightarrow\mathscr{C}(S)$ assigning to each $X=[\delta]\in\mathscr{P}(S)$ the pair $(\delta^*\ac{J}_\Omega,\delta^*\ve{\phi}_\Omega)\in\mathscr{C}(S)$ (where $\ac{J}_\Omega$ is the complex structure underlying $g_\Omega$) is inverse to the above map (\ref{equation_mapcs}). Therefore, we obtain:
 
\begin{corollary}\label{coro_mapcs}
The map (\ref{equation_mapcs}) is bijective.
\end{corollary}

\subsection{\c{T}i\c{t}eica affine sphere}\label{subsec_titeica}
As an example of Proposition \ref{prop_recon}, we give in this section an explicit affine spherical embedding $\iota_0:\mathbb{C}\rightarrow\mathbb{R}^3$ with Blaschke metric $|\dz|^2$ and normalized Pick differential $\dz^3$. Since $|\dz|^2$ is complete, by the discussions in Section \ref{subsec_convexity}, $\iota_0$ is a proper embedding and the projectivized embedding $\delta_0:=\mathbb{P}\circ\iota_0$ is a diffeomorphism from $\mathbb{C}$ to a properly convex domain in $\mathbb{RP}^2$, which turn out to be a triangle, as shown in Figure \ref{figure_dev01}.

The affine sphere connection $\D_0$ associated to $(|\dz|^2,\dz^3)$, defined on the vector bundle $E:=\T\mathbb{C}\oplus\underline{\mathbb{R}}$ and extended to $E\otimes\mathbb{C}=T_\mathbb{C}\mathbb{C}\oplus\underline{\mathbb{C}}$, is expressed under the local frame $(\paz,\pabz,\underline{1})$ of $E\otimes\mathbb{C}$ as
$$
\D_0=\dif+
\begin{pmatrix}
0&\frac{1}{\sqrt{2}}\dbz&\dz\\[5pt]
\frac{1}{\sqrt{2}}\dz&0&\dbz\\[5pt]
\frac{1}{2}\dbz&\frac{1}{2}\dz&0
\end{pmatrix}.
$$
It turns out that $\D_0$ is diagonalized under the new frame $(e_1,e_2,e_3)$ defined by
$$
e_1:=\sqrt{2}\pa_x+\underline{1}\,,\ e_2:=\sqrt{2}\left(-\frac{1}{2}\pa_x+\frac{\sqrt{3}}{2}\pa_y\right)+\underline{1}\,,\ e_3:=\sqrt{2}\left(-\frac{1}{2}\pa_x-\frac{\sqrt{3}}{2}\pa_y\right)+\underline{1}.
$$
Indeed, this frame is related to the original one by $(e_1,e_2,e_3)=(\paz,\pabz,\underline{1})B$, where the change-of-basis matrix $B$ and its inverse are the constant matrices
$$
B=
\begin{pmatrix}
\sqrt{2}&&\\
&\sqrt{2}&\\
&&1
\end{pmatrix}
\begin{pmatrix}
1&\omega&\omega^2\\
1&\omega^2&\omega\\
1&1&1
\end{pmatrix}
,\
B^{-1}=
\frac{1}{3}
\begin{pmatrix}
1&1&1\\
\omega^2&\omega&1\\
\omega&\omega^2&1\\
\end{pmatrix}
\begin{pmatrix}
\frac{1}{\sqrt{2}}&&\\
&\frac{1}{\sqrt{2}}&\\
&&1
\end{pmatrix}
$$
with $\omega:=e^{\frac{2\pi\ima}{3}}$. A computation then gives the expression of $\D_0$ under $(e_1,e_2,e_2)$ as
$$
\D_0=\dif+B^{-1}
\begin{pmatrix}
0&\frac{1}{\sqrt{2}}\dbz&\dz\\[5pt]
\frac{1}{\sqrt{2}}\dz&0&\dbz\\[5pt]
\frac{1}{2}\dbz&\frac{1}{2}\dz&0
\end{pmatrix}
B=\dif+
\sqrt{2}
\re
\begin{pmatrix}
\dz&&\\
&\omega^2\dz&\\
&&\omega\dz
\end{pmatrix}.
$$

In view of Appendix \ref{subsec_connection}, one checks that the parallel transport $\para_0(0,z)$ of $\D_0$ along any path from $z$ to $0$ is expressed under the frame $(e_1,e_2,e_3)$ as the matrix
$$
\para_0(0,z)=
\begin{pmatrix}
e^{\sqrt{2} \re(z)}\hspace{-7pt}&&\\
&e^{\sqrt{2} \re(\omega^2z)}\hspace{-7pt}&\\
&&e^{\sqrt{2} \re(\omega z)}
\end{pmatrix}.
$$

Let $\iota_0: \mathbb{C}\rightarrow E_0=\T_0\mathbb{C}\oplus\mathbb{R}\cong\mathbb{R}^3$ be the affine spherical embedding given by Proposition \ref{prop_recon} with $z_0=0$. Since $\underline{1}=e_1+e_2+e_3$, we have the expression
$$
\iota_0(z)=\para_0(0,z)
\begin{pmatrix}
1\\
1\\
1
\end{pmatrix}
=
\begin{pmatrix}
e^{\sqrt{2} \re(z)}\\[5pt]
e^{\sqrt{2} \re(\omega^2z)}\\[5pt]
e^{\sqrt{2} \re(\omega z)}
\end{pmatrix}.
$$
The image $\iota_0(\mathbb{C})=\{(x_1,x_2,x_3)\in\mathbb{R}^3\mid x_1x_2x_3=1\}$ is known as the \emph{\c{T}i\c{t}eica affine sphere}. Its projectivization $\delta_0(\mathbb{C})$ is the triangle 
$\Delta:=\{[x_1:x_2:x_3]\in\mathbb{RP}^2\mid x_i>0\}$. 

\begin{remark}
The identification $E_0\oplus\mathbb{R}\cong\mathbb{R}^3$ used in the above coordinate expression of $\iota_0$ is induced by the frame $(e_1,e_2,e_3)$, which is \emph{not} unimodular with respect to the volume form $\dif\vol_g\wedge\underline{1}^*=\dif x\wedge\dif y\wedge\underline{1}^*$ on $E$ given by the Blaschke metric $g=|\dz|^2$ (see Section \ref{subsec_affine}). Therefore, $\iota_0$ gives a hyperbolic affine sphere not with respect to the usual volume form on $\mathbb{R}^3$ given by the determinant, but with respect to a scaled one. Nevertheless, this issue is insignificant for our purpose as we are mainly concerned with the projectivized embedding $\delta_0$.
\end{remark}

\section{Finite volume ends and low-order poles}\label{sec_4}
In this chapter, we show in Theorem \ref{thm_finitevolumeend} that the poles of a holomorphic $k$-differential of order less than $k$ are exactly the finite volume ends for the complete solution to the vortex equation. When $k=3$, such ends have a projective-geometric description (Theorem \ref{thm_marquis}) and the above result implies Part (\ref{item_introthm0}) of Theorem \ref{thm_intro2}. 

\subsection{Hausdorff continuity}\label{subsec_benoisthulin}
Recall from Section \ref{subsec_yau} that the properly embedded affine sphere $M_\Omega$ projecting to a bounded convex domain $\Omega\subset\mathbb{R}^2$ is given by the convex function $u_\Omega\in C^0(\overline{\Omega})\cap C^\infty(\Omega)$ satisfying the equation $\det(\Hess(u_\Omega))=u^{-4}_\Omega$ with vanishing boundary values. The following theorem of Benoist and Hulin \cite[Corollary 3.3]{benoist-hulin} (see also \cite[Theorem 4.4]{dumas-wolf}) shows that $u_\Omega$ and its derivatives depend continuously on $\Omega$: 
\begin{theorem}[\cite{benoist-hulin}]
\label{thm_bh}
Let $\Omega\subset\mathbb{R}^2$ be a bounded convex domain,  $K\subset \Omega$ be a compact subset and let $\varepsilon>0$. Then for any  integer $k\geq0$ and any convex domain $\Omega'\subset\mathbb{R}^2$ sufficiently close to $\Omega$ in terms of Hausdorff distance, we have $K\subset \Omega'$ and 
$$
\|u_\Omega-u_{\Omega'}\|_{K,k}<\varepsilon.
$$
Here $\|\cdot\|_{K,k}$ denote the $C^k$-norm on $K$.
\end{theorem}
The Riemannian metric $g_\Omega$ and holomorphic cubic differential $\ve{\phi}_\Omega$ on $\Omega$, given by the Blaschke metric and normalized Pick differential on $M_\Omega$ (see Section \ref{subsec_yau}), have coordinate expressions with coefficients given by partial derivatives of $u_\Omega$ of order up to $3$. So the curvature $\kappa_{g_\Omega}=-1+\|\ve{\phi}_\Omega\|^2_{g_\Omega}:\Omega\rightarrow\mathbb{R}$ has the same continuity property with respect to $\Omega$ as $u_\Omega$. We denote $\kappa_{g_\Omega}$ simply by $\kappa_\Omega$ for clearness. Note that $\kappa_{\Omega}\leq0$ by Theorem \ref{thm_ue}. 

Unlike $u_\Omega$, whose definition relies on the affine chart, $\kappa_{\Omega}$ only depends on the properly convex domain $\Omega$ and is $\SL(3,\mathbb{R})$-equivariant in the sense that $a^*\kappa_{\Omega}=\kappa_{{a(\Omega)}}$ for all $a\in\SL(3,\mathbb{R})$.  Therefore, letting $\mathfrak{C}$ denote the space of all properly convex domains equipped with the Hausdorff topology and $\mathfrak{C}_*$ denote the topological subspace of $\mathfrak{C}\times \mathbb{RP}^2$ consisting of those $(\Omega, x)\in \mathfrak{C}\times \mathbb{RP}^2$ such that $x\in\Omega$, we can state the continuity properly of $\kappa_{\Omega}$ as:
\begin{corollary}\label{coro_khausdorff}
The function on $\mathfrak{C}_*$ given by $(\Omega, x)\mapsto \kappa_\Omega(x)=-1+\|\ve{\phi}_\Omega\|^2_{g_\Omega}(x)$
is continuous.
\end{corollary}

\subsection{Hilbert metric}\label{subsec_hilbert}
Besides the Riemannian metric $g_\Omega$, any properly convex set $\Omega\subset\mathbb{RP}^2$ also carries a natural Finsler metric $g^\mathsf{H}_\Omega$, the \emph{Hilbert metric}. Given an affine chart $\mathbb{R}^2$ containing $\overline{\Omega}$ and a Euclidean norm $|\cdot|$ on $\mathbb{R}^2$, $g^\mathsf{H}_\Omega$ is defined by
$$
g^\mathsf{H}_\Omega(v):=\left(\frac{1}{|x-a|}+\frac{1}{|x-b|}\right)|v|\ \text{ for all } x\in\Omega,\,v\in\T_{x}\Omega\cong\mathbb{R}^2,
$$
where $a,b\in\pa\Omega$ are the intersections of $\pa\Omega$ with the line passing through $x$ along the direction $v$. The distance from any other $y\in\Omega$ on the line to $x$ is given by 
$$
d_\Omega^\mathsf{H}(x,y)=|\log[a,x,y,b]|,
$$
where $[a,x,y,b]$ is the cross-ratio, expressed as $\frac{(a-y)(b-x)}{(a-x)(b-y)}$ if we identify the line with $\mathbb{R}$. It follows from projective invariance of the cross-ratio that $g_\Omega^\mathsf{H}$ is independent of the choice of $\mathbb{R}^2$ and $|\cdot|$ and is $\SL(3,\mathbb{R})$-equivariant in the sense that $a^*g_\Omega^\mathsf{H}=g_{a(\Omega)}^\mathsf{H}$ for all $a\in\SL(3,\mathbb{R})$. 

Associated to $g_\Omega$ is a natural \emph{Buseman volume form} $\nu^\mathsf{H}_\Omega$ on $\Omega$, whose density with respect to the Lebesgue measure underlying $|\cdot|$ is $\mathsf{Area}(B_x(1))^{-1}$, where $B_x(1)$ denotes the unit disk $\{v\mid g^\mathsf{H}_\Omega(v)\leq 1\}\subset \T_x\Omega\cong\mathbb{R}^2$ and $\mathsf{Area}(B_x(1))$ is its area with respect to the Lebesgue measure.

Clearly, $g^\mathsf{H}_\Omega$ and $\nu^\mathsf{H}_\Omega$ at a point $x\in\Omega$ depends continuously on $\Omega$, hence the density $\nu^\mathsf{H}_\Omega/\nu_\Omega$ of $\nu^\mathsf{H}_\Omega$ with respect to the volume form $\nu_\Omega$ of the Blaschke metric $g_\Omega$ has the same continuity property as $\kappa_{\Omega}$. Applying the Benz\'ecri Compactness Theorem (see \cite[Theorem 2.7]{benoist-hulin}), which states that the quotient of the above topological space $\mathfrak{C}_*$ by the $\SL(3,\mathbb{R})$-action is compact, we get:
\begin{corollary}\label{coro_bh}
The function on $\mathfrak{C}_*$ given by $(\Omega,x)\mapsto \frac{\nu_\Omega^\mathsf{H}}{\nu_\Omega}(x)$
is continuous and is bounded from above and below by positive constants.
\end{corollary}

\subsection{Definitions of ends and poles}\label{subsec_puncture}
The surfaces that we consider include those of infinite type. In this section, we define precisely what we mean by an \emph{end} of a  surface, a \emph{puncture} of a Riemann surface, \etc, although in most of the paper except the proof of Theorem \ref{thm_marquis} below, we only need to consider the ``trivial'' case of these definitions, \ie ends homeomorphic to an annulus.

Given a topological surface $S$ (without boundary), an \emph{end} of $S$ is defined intrinsically as an equivalence class of nested sequences of open subsets $U_1\supset U_2\supset\cdots$ in $S$ such that
\begin{itemize}
\item
the boundary $\pa U_i$ of $U_i$ in $S$ is compact for every $i$;
\item
any compact subset of $S$ is disjoint from $U_i$ when $i$ is sufficiently large;
\end{itemize}
while two sequences $(U_i)$ and $(V_i)$ are equivalent if each $U_i$ contains some $V_j$ and vice versa. A \emph{neighborhood} of an end $p=[(U_i)]$ is by definition a subset of $S$ containing some $U_i$. A sequence of points or an oriented path on $S$ is said to \emph{tend to $p$} if it is eventually contained in every neighborhood of $p$, or equivalently, in every $U_i$. 

The space $\Ends(S)$ of all ends of $S$ carries a natural totally disconnected, separable and compact topology (see \cite{richards}). An end $p\in \Ends(S)$ is said to be \emph{planar} if it has a neighborhood homeomorphic to a domain in $\mathbb{R}^2$. An end $p$ is both planar and isolated in $\Ends(S)$ if and only if a neighborhood is homeomorphic to an annulus. This is equivalent to the existence of a surface $S'$ with boundary such that
\begin{itemize}
\item the boundary $\pa S'$ is homeomorphic to a circle;
\item the interior $S'\setminus\pa S$ is  identified homeomorphically with $S$;
\item every sequence of points on $S$ tending to $p$ has a subsequence converging in $S'$ to a point on $\pa S'$. 
\end{itemize}
In intuitive terms, $S'$ is obtained from $S$ by attaching a boundary circle to $p$. We call such an $S'$ a \emph{bordification} of $S$ at $p$. Note that by the Hausdorff property of $\Ends(S)$, the last condition above can be replace by ``there exists a sequence on $S$ tending to $p$ which converges in $S'$ to a point on $\pa S'$''.


If $\Sigma$ is a Riemann surface, an end $p\in \Ends(\Sigma)$ is said to be a \emph{puncture} if it has a neighborhood conformally equivalent to the punctured disk $\{z\in\mathbb{C}\mid 0<|z|\leq 1\}$.  We view $z$ as a conformal local coordinate of $\Sigma$ ``centered at $p$''. 

A holomorphic $k$-differential $\ve{\phi}$ on $\Sigma$ is said to have \emph{non-essential singularity of degree $d\in\mathbb{Z}$} at a puncture $p$ if the expression of $\ve{\phi}$ in a conformal local coordinate $z$ centered at $p$ is $z^df(z)$, where $f$ is a holomorphic function defined around $z=0$ with $f(0)\neq0$. Thus, a pole of order $m\in\mathbb{Z}$ is by definition a non-essential singularity of degree $-m$.
\begin{remark}
We allow poles to have nonpositive order, although this case is more commonly referred to as \emph{removable singularities}. In particular, they are included when we talk about ``poles of order less than $k$'' in Theorem \ref{thm_intro2} (\ref{item_introthm0}) and Section \ref{subsec_finitvolume}.
 \end{remark}
 
\subsection{Finite volume ends of convex $\mathbb{RP}^2$-surfaces}\label{subsec_cusps}
A surface $S$ endowed with a convex $\mathbb{RP}^2$-structure is referred to as a \emph{convex $\mathbb{RP}^2$-surface}. Such an $S$ is identified via a developing map with the quotient of a properly convex domain $\Omega\subset\mathbb{RP}^2$ by a torsion-free subgroup $\Gamma\subset\SL(3,\mathbb{R})$ acting properly discontinuous on $\Omega$. The Hilbert metric $g^\mathsf{H}_\Omega$, Buseman volume form $\nu^\mathsf{H}_\Omega$ and Blaschke metric $g_\Omega$ on $\Omega$ are all invariant under $\Gamma$, hence descent to objects on $S$ bearing the same names (compare Section \ref{subsec_yau}).

By Corollary \ref{coro_bh}, the proportion between the Buseman volume form and the volume form of the Blaschke metric on every convex $\mathbb{RP}^2$-surface $S$ is bounded from above and below by universal positive constants.
A \emph{finite volume end} of $S$ is by definition an end which has a neighborhood with finite volume under either of these volume forms.

The works of L. Marquis \cite{marquis} give a characterization of finite volume ends:
\begin{theorem}[\cite{marquis}]
\label{thm_marquis}
Let $\Omega\subset\mathbb{RP}^2$ be a properly convex domain and $\Gamma$ be a torsion-free subgroup of $\SL(3,\mathbb{R})$ acting properly discontinuously on $\Omega$. Then an end $p$ of the convex $\mathbb{RP}^2$-surface $S=\Omega/\Gamma$ is a finite volume end if and only if there is a parabolic projective transformation $a\in\Gamma$ with fixed point $x\in\pa\Omega$ and a closed ellipse $E\subset\mathbb{RP}^2$ preserved by $a$, such that $E\setminus\{x\}$ is contained in $\Omega$ and its quotient by $a$ is embedded in $S$ as a neighborhood of $p$.
\end{theorem}
In other words, an end is a finite volume one if and only if it is projectively equivalent to a cusp of a hyperbolic surface. In particular, such an end is topologically trivial, \ie an isolated planar end.
\begin{proof}[Proof]
The ``if'' part: The interior $E^\circ$ of $E$ equipped with its Hilbert metric $g^\mathsf{H}_{E^\circ}$ is isometric to the hyperbolic plane, with $a$ corresponding to a parabolic isometry. So any $a$-invariant ellipse strictly contained in $E$ gives a neighborhood of $p$ with finite volume under $g^\mathsf{H}_{E^\circ}$, which has finite volume under the Hilbert metric of $S$ as well because $g^\mathsf{H}_{E^\circ}$ dominates $g^\mathsf{H}_\Omega$ by virtue of the fact that $E^\circ\subset \Omega$.

To prove the ``only if'' part, we first show that $p$ is an isolated planar end. 
Suppose by contradiction that it is not the case, so that every neighborhood of $p$ is a surface of infinite type.  Let $U$ be a neighborhood with finite volume. We can pick a non-elementary simple loop $\gamma$ on $S$ (where a simple loop $\gamma$ is said to be \emph{elementary} if $S\setminus\gamma$ has a connected component homeomorphic to a punctured disk) such that a connected component  $V$ of $S\setminus \gamma$ is a neighborhood of $p$ contained in $U$.  By \cite[Proposition 5.12]{marquis}, the holonomy along $\gamma$ is a hyperbolic projective transformation with principal segment contained in $\Omega$. Thus, by \cite[Proposition 5.9]{marquis}, we can deform $\gamma$ by free homotopy to a geodesic loop $\gamma'$, so that $V$ gets deformed into a neighborhood $V'$ of $p$ bounded by $\gamma'$, which still has finite volume. Since $V'$ is itself a convex $\mathbb{RP}^2$-surface of infinite type, the proof of \cite[Th\'eor\`eme 5.18]{marquis} shows that $V'$ contains either infinitely many disjoint embedded ideal geodesic triangles or infinitely many disjoint pairs of pants with non-elementary geodesic boundary loops. But it is shown in \cite[Propositions 1.13 and 5.17]{marquis} that each such triangle or pair of pants has Buseman volume bounded from below by a universal constant. Therefore, $V'$ has infinite volume, a contradiction.

Now that $p$ is an isolated planar end, we can pick a neighborhood $U$ of $p$ homeomorphic to a punctured disk with boundary $\pa U$ a simple loop. The holonomy $a\in\SL(3,\mathbb{R})$ of $S$ along $\pa U$ is either hyperbolic, quasi-hyperbolic or parabolic and a case-by-case check (see \cite[Propositions 5.21, 5.27]{marquis}) shows that $U$ can have finite volume only when $a$ is parabolic and $p$ corresponds to the fixed point of $a$.  It follows (see \cite[Proposition 4.16]{marquis}) that $p$ admits the required description.
\end{proof}

\subsection{Finite volume ends and low-order poles}\label{subsec_finitvolume}
In view of the definition of finite volume ends from the previous section, Part (\ref{item_introthm0}) of Theorem \ref{thm_intro2} is a consequence of the following more general result:
\begin{theorem}
\label{thm_finitevolumeend}
Let $\Sigma$ be a Riemann surface, $\ve{\phi}$ be a nontrivial holomorphic $k$-differential on $\Sigma$ and $p$ be an end of $\Sigma$. Then $p$ has a neighborhood with finite volume under the complete solution $g$ to the vortex equation on $(\Sigma,\ve{\phi})$ (see Theorem \ref{thm_ue}) if and only if $p$ is a puncture of $\Sigma$ and $\ve{\phi}$ has a pole of order less than $k$ at $p$.
\end{theorem}
Note that if $\Sigma$ is hyperbolic and $\ve{\phi}$ is trivial, so that the hyperbolic metric given by Uniformization solves the vortex equation, then the statement of the theorem boils down to the well known fact that finite volume ends of complete hyperbolic surfaces are cusps. 


The mains tools in the proof are the following local estimates at low-order poles:
\begin{lemma}\label{lemma_vortexcusp}
Let $\ve{\phi}$ be a holomorphic $k$-differential on $\mathbb{D}^*:=\{z\in\mathbb{C}\mid 0<|z|<1\}$ with a pole of order less than $k$ at $0$. Let $g_1$ be the complete conformal hyperbolic metric on $\mathbb{D}^*$ and $g$ be a solution to the vortex equation on $(\mathbb{D}^*,\ve{\phi})$ which is complete at $0$ (\ie rectifiable paths tending to $0$ have infinite length under $g$). Then we have
$$
\lim_{z\rightarrow0}\kappa_g(z)=-1,\quad \limsup_{z\rightarrow0}\frac{g}{g_1}(z)\leq 1.
$$
\end{lemma}
\begin{proof}
Let $g'$ be a conformal metric on $\mathbb{D}^*$ such that $g'=g$ on $\{0<|z|<1/3\}$ and $g'=g_1$ on $\{2/3<|z|<1\}$. Noting that $\kappa_g\geq-1$,  we can pick a sufficiently large constant $M$ such that the curvature of $Mg'$ is bounded from below by $-1$. Since $g'$ is complete, Yau's Ahlfors-Schwarz lemma \cite{yau_ahlfors} (see Remark \ref{remark_yauas}) implies that $Mg'\geq g_1$. Therefore, on $\{0<|z|<1/3\}$ we have
$$
-1\leq \kappa_g=-1+\|\ve{\phi}\|^2_g\leq -1+M^k\|\ve{\phi}\|^2_{g_1}.
$$
But it follows from the expressions 
$$
g_1=\frac{|\dz|^2}{|z|^2\left(\log |z|\right)^2},\quad \ve{\phi}=\frac{f(z)}{z^{k-1}}\dz^k
$$
(where $f$ is holomorphic on $\{|z|<1\}$) that $\|\ve{\phi}\|^2_{g_1}=|f(z)|^2|z|^2(\log|z|)^{2k}$ tends to $0$ as $z\rightarrow0$, whence $\lim_{z\rightarrow0}\kappa_g(z)=-1$.

The second assertion is deduced from the first one as follows. Given $r>0$, we consider the complete conformal hyperbolic metric
$$
g_r=\frac{|\dz|
^2}{|z|^2\left(\log |z|-\log r\right)^2}
$$
on the punctured disk $\{0<|z|<r\}$. Since $\lim_{z\rightarrow0}\kappa_g(z)=-1$, for any $\lambda>1$ there is $0<r<1$ such that  $-1\leq\kappa_g(z)\leq-\lambda^{-1}$ for $0<|z|<r$, so that we can apply the Ahlfors-Schwarz lemma on $\{0<|z|<r\}$ to the metrics $g_r$ and $\lambda^{-1} g$, and conclude that $g_r\geq \lambda^{-1} g$. As a result,
$$
\frac{g}{g_1}(z)=\frac{g}{g_r}(z)\frac{g_r}{g_1}(z)\leq \lambda\left(\frac{\log|z|}{\log|z|-\log r}\right)^2
$$
whenever $0<|z|<r$. This implies the required $\limsup$ because $\lambda>1$ is arbitrary and the fraction in the parentheses tends to $1$ as $z$ goes to $0$.
\end{proof}

\begin{proof}[Proof of Theorem \ref{thm_finitevolumeend}]
Suppose $p$ has a neighborhood with finite volume under $g$. We first show that $p$ is a puncture. To this end, we may suppose that $\Sigma$ is a hyperbolic Riemann surface (otherwise $\Sigma$ is parabolic and $p$ must be a puncture) and let $g_0$ be the hyperbolic metric on $\Sigma$ given by Uniformization. We have $g\geq g_0$ by Yau's Ahlfors-Schwarz lemma or alternative by Lemma \ref{lemma_completebound}. So $p$ is a finite-volume end with respect to $g_0$, hence a puncture (see the paragraph following Theorem \ref{thm_finitevolumeend}).

By Theorem \ref{thm_ue}, the curvature $\kappa_g=-1+\|\ve{\phi}\|^2_g$ is nonpositive, so the conformal ratio $|\ve{\phi}|^\frac{2}{k}/g=\|\ve{\phi}\|^\frac{2}{k}_g$ between singular flat metric $|\ve{\phi}|^\frac{2}{k}$ and $g$ is at most $1$. Thus, $p$ also has a neighborhood with finite volume under the volume form of $|\ve{\phi}|^\frac{2}{k}$. By Lemma \ref{lemma_area} in the appendix, $\ve{\phi}$ has a pole or order at most $k-1$ at $p$. This proves the ``only if'' part.

The ``if'' part follows from Lemma \ref{lemma_vortexcusp}: Identifying a neighborhood of $p$ with the punctured disk $\mathbb{D}^*$, the lemma implies that $g\leq Cg_1$ on $\{0<|z|<\varepsilon\}$ for some constants $C>1$ and $0<\varepsilon<1$. But $\{0<|z|<\varepsilon\}$ has finite volume with respect to $g_1$, hence with respect to $g$ as well.
\end{proof}

The $k=3$ case of Theorem \ref{thm_finitevolumeend} is a slightly more general version of the results of Benoist and Hulin in \cite{benoist-hulin}. Indeed, the ``only if'' statement for $k=3$ is essentially contained in \cite{benoist-hulin}, with a proof different from the one above.

\section{Bordification of cubic differentials around poles}\label{sec_5}
In this chapter, after reviewing some facts about the local geometry of holomorphic $k$-differentials around poles, we show that for a cubic differential, any pole of order $n+3$ ($n\geq1$) has a neighborhood formed by $n$ copies of a half-plane. We then define the negative ray bordification $\Sigma'$ in Theorem \ref{thm_intro3}.
\subsection{Normal form and invariants of $k$-differentials at poles}\label{subsec_normalform}
Around a pole, we can always write a $k$-differential in normal form after a change of coordinate:
\begin{theorem}[\textbf{Normal form}]\label{thm_norm}
Given a holomorphic $k$-differential $\ve{\phi}$ on a punctured neighborhood of $0$ in $\mathbb{C}$ with non-essential singularity of degree $d\in\mathbb{Z}$ at $0$ (see Section \ref{subsec_puncture}), there exists a conformal local coordinate $w$ centered at $0$ such that
\begin{equation}\label{eqn_normalform}
\ve{\phi}
=\begin{cases}
w^d\dif w^k\quad  &\mbox{ if } d\notin\{-k,-2k,\cdots\},\\[4pt]
Rw^{-k}\dif w^k\quad(R\in\mathbb{C}^*) &\mbox{ if }d=-k,\\[4pt]
\left(w^{-l}+A\right)^kw^{-k}\dif w^k\quad (A\in\mathbb{C})  &\mbox{ if }d=-(l+1)k\mbox{ with }l\in\mathbb{Z}_+.
\end{cases}
\end{equation}
Moreover, $R\in\mathbb{C}^*$ and $A^k\in\mathbb{C}$ are invariants of $\ve{\phi}$, independent of the choice of $w$.
\end{theorem}
Note that $A$ itself is not coordinate-independent, because if $\lambda^{kl}=1$ then the last normal form is written in the coordinate $\tilde{w}=\lambda w$ as $(\tilde{w}^{-l}+\lambda^lA)^k\tilde{w}^{-k}\dif\tilde{w}^k$. The invariant $R$, defined at poles of order $k$, is called the \emph{residue}, and the theorem implies that $d$, $R$ and $A^k$ give full local information of $\ve{\phi}$ at a pole.

A proof of Theorem \ref{thm_norm} in the $k=2$ case is given in \cite[\S 6]{strebel}. In this section, we interpret the invariants $R$ and $A^k$ from a geometric point of view, which will be used in the next section. The interpretation, combined with Appendix \ref{sec_flatend}, also yield in this section a proof of the theorem for general $k$.

The geometric point of view, well-known for $k=1,2$ (see \cite{zorich}), is to consider a Riemann surface $\Sigma$ endowed with a holomorphic $k$-differential $\ve{\phi}$ as a \emph{$\frac{1}{k}$-translation surface} with singularities. Namely, the complement of the zeros of $\ve{\phi}$ in $\Sigma$ carries an atlas with charts taking values in $\mathbb{C}$ and transition maps in the group
$$
\Aut(\mathbb{C},\dz^k)=\big\{z\mapsto az+b\mid a^k=1,\ b\in\mathbb{C}\big\}
$$ 
of automorphisms of $\mathbb{C}$ preserving the $k$-differential $\dz^k$. We further define:
\begin{definition}[\textbf{$\frac{1}{k}$-translation ends}]\label{def_translation_ends}
A \emph{$\frac{1}{k}$-translation end} is a $\frac{1}{k}$-translation surfaces homeomorphic to the half-open annulus $\mathbb{S}^1\times[0,1)$. Two $\frac{1}{k}$-translation ends are said to be \emph{equivalent} if, after removing a compact subset from each of them if necessary, they are isomorphic.
A $\frac{1}{k}$-translation end is said to be \emph{regular} with \emph{degree $d$} if it is conformally a punctured disk on which the $k$-differential does not have zeros and has a non-essential singularity of degree $d$ at the puncture. 
\end{definition}
\begin{remark} 
This definition of regularity is nonstandard and comes from our treatment of ends of flat surfaces in Appendix \ref{sec_flatend}. It should not be confused with the convention in some literature of calling poles of order $\leq k$ regular.
\end{remark} 

These definitions transform the analytic problem of finding a normal form for holomorphic $k$-differentials into the geometric problem of classifying regular $\frac{1}{k}$-translation ends $F$ up to equivalence. For the latter problem, we consider invariants of $F$ coming from its \emph{monodromy}, \ie the element in $\Aut(\mathbb{C},\dz^k)$, well-defined up to conjugation, given by the orienting generator of $\pi_1(F)$ and the \emph{holonomy representation} $\pi_1(F)\rightarrow\Aut(\mathbb{C},\dz^k)$ from a \emph{developing pair} (in the sense of \cite{goldman_gx}). But every element $\sigma:z\mapsto az+b$ in $\Aut(\mathbb{C},\dz^k)$ is determined up to conjugation by following invariants:
\begin{itemize}
	\item the \emph{rotation} $\rho(\sigma):=a\in\langle e^\frac{2\pi\ima}{k}\rangle\cong\mathbb{Z}/k\mathbb{Z}$;
	\item the \emph{translation} $\tau(\sigma)$, defined by $\tau(\sigma):=0$ if $a\neq 1$ and $\tau(\sigma):=b$ if $a=1$ (note that if $a\neq 1$ then $z\mapsto az+b$ is conjugate to $z\mapsto az$). It is well-defined only up to rotation by $\langle e^\frac{2\pi\ima}{k}\rangle$, hence should be viewed as taking values in $\mathbb{C}/\langle e^\frac{2\pi\ima}{k}\rangle$. Alternatively, $\tau(\sigma)^k\in\mathbb{C}$ is a well-defined conjugation invariant.
\end{itemize}
Abusing the notation, we let $\rho(F)$ and $\tau(F)$ denote the rotation and translation, respectively, of the monodromy of a $\frac{1}{k}$-translation end $F$. They are invariant under the equivalence relation, and are related to the invariants from the normal form by:
\begin{lemma}\label{lemma_invariants}
For a $\frac{1}{k}$-translation end $F=(\{0<|w|\leq\varepsilon\}, \ve{\phi})$ such that $\ve{\phi}$ has the normal form (\ref{eqn_normalform}), we have
$$
\rho(F)=e^{\frac{2d\pi\ima}{k}}
,\quad \left(\frac{\tau(F)}{2\pi\ima}\right)^k=
\begin{cases}
0 &\text{if } d\notin\{-k,-2k,\cdots\},\\
R &\text{if }d=-3,\\
A^k&\text{if }d=-(l+1)k,\ l\geq1.
\end{cases}
$$
\end{lemma}
\begin{proof}
For $\ve{\phi}$ in each of the three cases, an explicit developing map $\widetilde{F}\to(\mathbb{C},\dz^3)$ and  the corresponding monodromy are given in the following table, from which $\sigma(F)$ and $\tau(F)$ can be read off immediately. 

\vspace{5pt}

\begin{tabular}{|c|c|c|c|}
	\hline
	expression of $\ve{\phi}$& developing map & monodromy\\
	\hline
	\rule[-4pt]{0pt}{16pt} $w^d\dif w^k$&$w\mapsto z(w)=\frac{k}{d+k}w^{\frac{d+k}{k}}$ & $z\mapsto e^{\frac{2(d+k)\pi\ima}{k}}z$\\
	\hline
		\rule[-4pt]{0pt}{16pt} $Rw^{-k}\dif w^k$& $
		w\mapsto z(w)=\sqrt[k]{R}\,\log w.
		$& $z\mapsto z+2\pi\ima\sqrt[k]{R}$\\
	\hline
		\rule[-4pt]{0pt}{16pt} $(w^{-l}+A)^kw^{-k}\dif w^k$& $
		w\mapsto z(w)=-\frac{1}{l}w^{-l}+A\log w
		$& $z\mapsto z+2\pi\ima A$\\
	\hline
\end{tabular}

\vspace{5pt}

(In the language of complex analysis, each $w\mapsto z(w)$ in the second column is a multi-valued function on $\{0<|w|\leq\varepsilon\}$, and the corresponding monodromy describes how we move from a single-valued branch to another when the variable $w$ goes counterclockwise around $0$.)
\end{proof}
\begin{proof}[Proof of Theorem \ref{thm_norm}]
Lemma \ref{lemma_normalform} in Appendix \ref{sec_flatend} gives a coordinate $\tilde{w}$ such that the flat metric underlying $\ve{\phi}$ coincides with the flat metric underlying some $k$-differential $\widetilde{\ve{\phi}}$ whose expression in $\tilde{w}$ has the form (\ref{eqn_normalform}). It follows that $\ve{\phi}=\lambda\widetilde{\ve{\phi}}$ for some $\lambda\in\mathbb{C}$ with $|\lambda|=1$. If $d=-k$, this already means that $\ve{\phi}$ has the normal form in $\tilde{w}$; otherwise, we can check that $\ve{\phi}$ has the normal form in the coordinate $w=\lambda^\frac{1}{d+k}\tilde{w}$.
We have thus established the existence of $w$. The coordinate-independence of $R$ and $A^k$ follows from Lemma \ref{lemma_invariants}, because $\tau(F)^k$ is coordinate-independent.
\end{proof}
Using Lemma \ref{lemma_invariants}, we can reformulate Theorem \ref{thm_norm} as:
\begin{corollary}
\label{coro_model}
The degree $d=d(F)$ and monodromy translation $\tau=\tau(F)$ of any regular $\frac{1}{k}$-translation end $F$ satisfy:
\begin{itemize}
\item
if $d\notin\{-k,-2k,\cdots\}$ then $\tau=0$;
\item
if $d=-k$ then $\tau\neq0$.
\end{itemize}
Conversely, any $d\in\mathbb{Z}$ and $\tau\in\mathbb{C}/\langle e^\frac{2\pi\ima}{k}\rangle$ satisfying these constraints are the degree and monodromy translation of a regular $\frac{1}{k}$-translation end, which is unique up to equivalence. 
In particular, any two regular $\frac{1}{k}$-translation ends of the same degree $d\notin\{-k,-2k,\cdots\}$ are equivalent to each other.
\end{corollary}
\begin{remark}
The existence of $F$ with prescribed $d$ and $\tau$ in the corollary follows from the analytic expressions in (\ref{eqn_normalform}) through Lemma \ref{lemma_invariants}. However, we can also construct $F$ geometrically as in Theorem \ref{thm_classification} of the appendix, where geometric models for \emph{flat ends} are built. Those models can be endowed with $\frac{1}{k}$-translation structures (as long as the parameter $\theta$ is a multiple of $\frac{2\pi}{k}$) to give the required $F$ here. In particular, when $d=-k$, we can take a cylinder as $F$, see Section \ref{subsec_bord2} below for detailed discussions in the $k=3$ case. When $d<-k$, another model will be built in the next section.
\end{remark}

Finally, note that although the analytic invariants $R$ and $A^k$ are recovered from the monodromy by Lemma \ref{lemma_invariants}, the degree $d$ is recovered only modulo $k$. Nevertheless, as the next lemma shows, $d$ still has a geometric interpretation through the \emph{total boundary curvature} $\theta(F)$, which is an invariant of $F$ only involving the underlying flat metric. We refer to Appendix \ref{subsec_classification} for the definition. 
\begin{lemma}\label{lemma_degree}
For any regular $\frac{1}{k}$-translation end $F$ of degree $d\in\mathbb{Z}$, we have
$$
\theta(F)=2\pi\left(\frac{d}{k}+1\right).
$$
\end{lemma}
\begin{proof}
We can assume $F=(\{0<|w|\leq\varepsilon\},\ve{\phi})$ with $\ve{\phi}$ given by (\ref{eqn_normalform}). The definition of $\theta(F)$ only involves the flat metric $|\ve{\phi}|^\frac{2}{k}$ underlying $\ve{\phi}$, which has the expression
$$
|\ve{\phi}|^\frac{2}{k}
=\begin{cases}
|w|^\frac{2d}{k}|\dif w|^2\quad  &\mbox{ if } d\notin\{-k,-2k,\cdots\},\\[4pt]
|R|^\frac{2}{k}|w|^{-2}|\dif w|^2 &\mbox{ if }d=-k,\\[4pt]
\left|w^{-l}+A\right|^2|w|^{-2}|\dif w|^2=\left|\tilde{w}^{-l}+|A|\right|^2|\tilde{w}|^{-2}|\dif\tilde{w}|^2&\mbox{ if }d=-(l+1)k,
\end{cases}
$$
where in the last case we change the coordinate $w$ to $\tilde{w}:=e^{\arg(A)/l}w$. Using Theorem \ref{thm_classification} in the appendix, we get the value of $\theta(F)$ from the expression.
\end{proof}

\subsection{Half-planes around a pole of order $\geq4$}\label{subsec_halfplanemodel}
From now on we only consider holomorphic \emph{cubic} differentials, \ie the case $k=3$. For applications in the next section and Chapter \ref{sec_6}, we shall show that a pole of order $n+3$ ($n\geq1$) has a neighborhood form by $n$ copies of the right half-plane
$$
\overline{\H}:=\{z\in\mathbb{C}\mid\re(z)\geq0\},
$$ 
endowed with the cubic differential $\dz^3$, patched together through $\frac{2\pi}{3}$-rotations as in Figure \ref{figure_halfplanes} below. Using Definition \ref{def_translation_ends}, we formulate the result as:
\begin{theorem}\label{thm_halfplanes}
Any regular $\frac{1}{3}$-translation end of degree $-n-3$ ($n\geq 1$) is equivalent to one from Figure \ref{figure_halfplanes}.
\end{theorem}
\begin{figure}[h]
\includegraphics[width=4in]{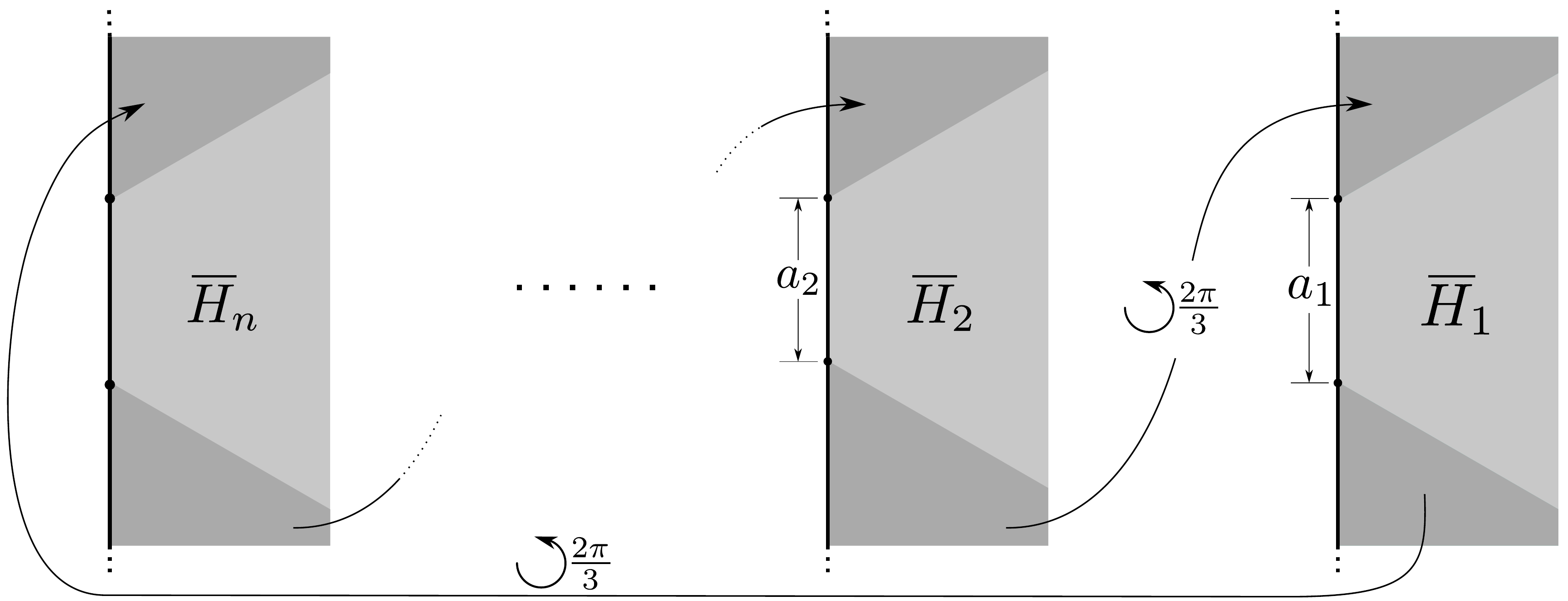}
\caption{$\frac{1}{3}$-translation end constructed by assembling $n$ copies of $\overline{\H}$: the shaded region at the bottom of each $\overline{H}_{i+1}$ is glued with the one at the top of $\overline{H}_{i}$  (the indices counted mod $n$) through a map of the form $z\mapsto e^{2\pi\ima/3}z+z_0$. The distance $a_i>0$ between the two shaded regions in each $\overline{H}_i$ is a customizable parameter.}
\label{figure_halfplanes}
\end{figure}
Note that the overbars in our notations just indicate that we are considering \emph{closed} half-planes. In the sequel we often remove the bars to denote the open ones.

This theorem generalizes the result in \cite[Appendix A]{dumas-wolf}. The statement and the proof below also generalize easily to $\frac{1}{k}$-translation ends for any $k\geq2$.
\begin{proof}
Let $F(a_1,\cdots, a_n)$ denote the $\frac{1}{3}$-translation end in Figure \ref{figure_halfplanes}.
We claim that this is a regular $\frac{1}{3}$-translation end with degree $-n-3$. To prove the claim, first note that a $\frac{1}{k}$-translation end is regular if and only if its underlying \emph{flat end} is regular (see Appendix \ref{subsec_regularity}); but the flat metric on $F(a_1,\cdots, a_n)$ is clearly complete, hence regular by Lemma \ref{lemma_huber}. This shows that $F(a_1,\cdots,a_n)$ is regular. Its total boundary curvature is $-2n\pi/3$ (because the boundary is piecewise geodesic with $n$ vertices, each contributing $-2\pi/3$; see Appendix \ref{subsec_classification}), hence Lemma \ref{lemma_degree} gives the required degree count, proving the claim.

To prove the theorem, we need to show that any regular $\frac{1}{3}$-translation end $F$ of degree $-n-3$ is equivalent to some $F(a_1,\cdots, a_n)$. In view of the above claim and Corollary \ref{coro_model}, this is trivial if
$n$ is not divisible by $3$, while in the case $n=3l$ ($l\in\mathbb{Z}_+$) we only need to find some parameters $a_1,\cdots, a_n$ such that $F(a_1,\cdots, a_n)$ has the correct monodromy translation $\tau(F(a_1,\cdots, a_n))=\tau:=\tau(F)$. 

Figure \ref{figure_modelend} below gives a way to find such parameters. First, consider the region $L(a)\subset\mathbb{C}$ showed in the first picture, whose boundary consists of a vertical segment of length $a$ and two rays with slope $\pm\sqrt{3}$, so that $F(a_1,\cdots,a_n)$ is obtained alternative by gluing $L(a_1),\cdots, L(a_n)$ successively along their boundary rays. Now, we start the construction with $F(a,\cdots, a)$ (with $n=3l$ identical parameters), which is the $l$-fold cyclic cover of the complement of a equilateral triangle in $(\mathbb{C},\dz^3)$, and has trivial monodromy. If $\tau=0$, this is already what we need. Otherwise, since $\tau$ is well-defined up to rotation by $\langle e^\frac{2\pi\ima}{3}\rangle$, we can assume $\im(\tau)<0$ and consider the strip $B\subset(\mathbb{C},\dz^3)$ bounded by the segment from $0$ to $\tau$ and the rays $\mathbb{R}_{\geq0}$ and $\tau+\mathbb{R}_{\geq0}$. Then we graft $F(a,\cdots, a)$ by inserting $B$ into one of the $n$ copies of $L(a)$, see the second picture. This yields a $\frac{1}{3}$-translation end $F'$ with the prescribed monodromy translation $\tau$. Finally, we remove a trapezoid from $F'$ as in the last picture. The resulting surface is formed by some $L(a_1),\cdots,L(a_n)$, with $a_i=a$ for all but two $i$'s. This is the required $F(a_1,\cdots, a_n)$.

\begin{figure}[h]
\centering\includegraphics[width=5in]{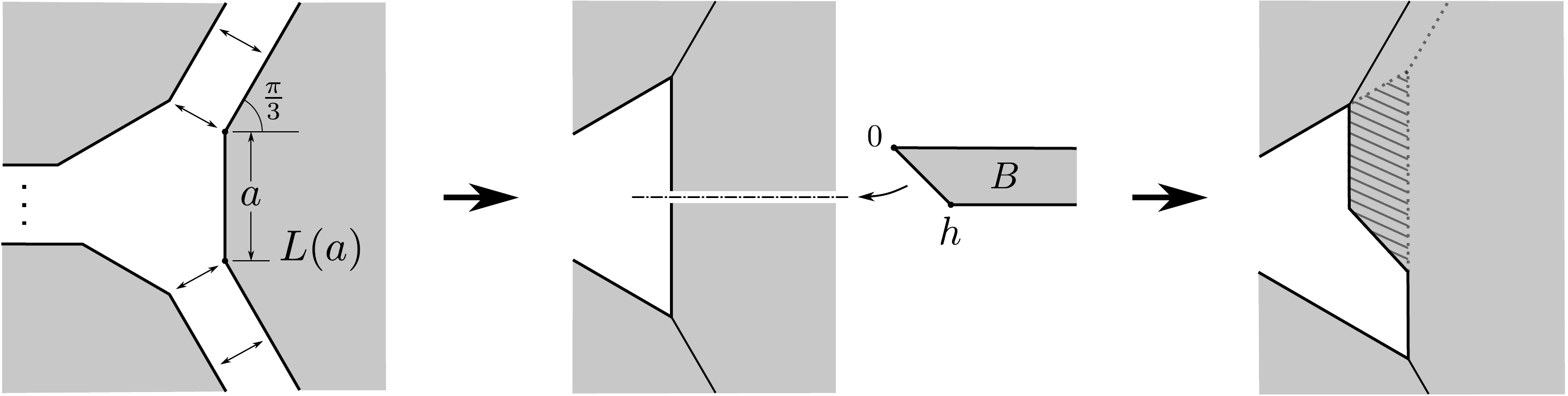}
\caption{Construction of $F(a_1,\cdots,a_n)$ with prescribed monodromy translation $\tau$ from $F(a,\cdots, a)$.}
\label{figure_modelend}
\end{figure}
\end{proof}

Using Theorem \ref{thm_diskfine}, we can estimate nonpositively curved solutions to the vortex equation in terms of the coordinate on each half-plane: 
\begin{corollary}\label{coro_halfplaneestimate}
Given a $\frac{1}{3}$-translation end $F=\overline{H}_1\cup\cdots\cup\overline{H}_n$ as in Figure \ref{figure_halfplanes}, there are constants $C,r_0>0$ with the following property: For any Riemann surface $\Sigma$ with a cubic differential $\ve{\phi}$ such that $(\Sigma,\ve{\phi})$ contains $F$, and any  solution $g$ with nonpositive curvature to the vortex equation on $(\Sigma,\ve{\phi})$,
writing $g=e^u|\dz|^2$ on $\overline{H}_i\cong\CH$, we have
$$
0\leq u(z)\leq C|z|^{\frac{1}{2}}e^{-\sqrt{6}|z|} \ \text{ for all } z\in \overline{H}_i,\,|z|\geq r_0.
$$
\end{corollary}
\begin{proof}
Around $\overline{H}_i$, $F$ locally looks like Figure \ref{figure_estimate}, 
	\begin{figure}[h]
	\centering\includegraphics[width=1.4in]{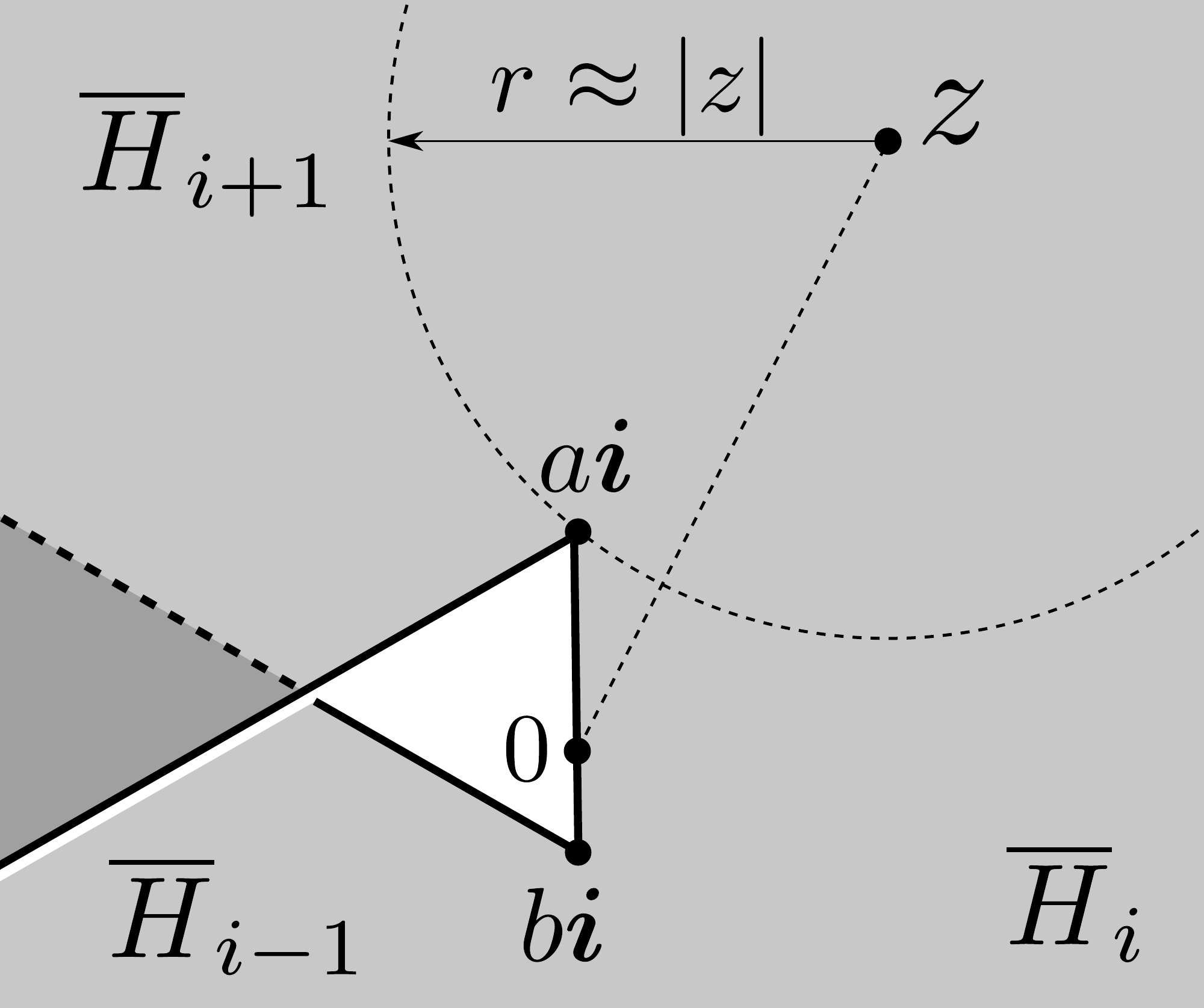}
	\caption{The maximal disk centered at $z\in\overline{H}_i\cong\CH$ contained in $\bigcup_{j=1}^n\overline{H}_j$. 
	}
	\label{figure_estimate}
\end{figure}
from which it is clear that there is a constant $c>0$,
only depending on the points $a\ima, b\ima\in\pa H_i\cong\ima\mathbb{R}$ where $\pa H_{i-1}$ and $\pa H_{i+1}$ intersect $\pa H_i$, such that for every $z\in\overline{H}_i$, the disk of radius $|z|-c$ centered at $z$ is contained in $F$. Applying Theorem \ref{thm_diskfine} to $u$ on this disk, we obtain
$$
0\leq u(z)\leq C'(|z|-c)^{\frac{1}{2}}e^{-\sqrt{6}(|z|-c)}\ \text{ for all } z\in\overline{H}_i \text{ with }|z|-c\geq r_0',
$$
where the constants $C',r_0'>0$ are universal. The required inequality follows.
\end{proof}
\begin{remark}\label{remark_stronger}
A stronger estimate is given in \cite[Theorem 5.7]{dumas-wolf} with the upper bound $C|z|^{-\frac{1}{2}}e^{-\sqrt{6}|z|}$. However, the exponent of the $|z|$-factor is nonessential for our method in Chapter \ref{sec_6}.
\end{remark}

\subsection{Bordification at a pole of order $\geq4$}\label{subsec_bord1}
In this section, we give a formal definition of the negative ray bordification $\Sigma$, used in Theorem \ref{thm_intro3}, for high-order poles. First, we introduce the following terminologies for a Riemann surface $\Sigma$ equipped with a nontrivial holomorphic cubic differential $\ve{\phi}$:

\pt By a \emph{geodesic ray} on $(\Sigma,\ve{\phi})$, we mean a unit-speed geodesic ray with respect to the flat metric $|\ve{\phi}|^\frac{2}{3}$ avoiding the zeros of $\ve{\phi}$, or equivalently, a curve $\gamma:[0,+\infty)\rightarrow \Sigma$ such that $\ve{\phi}(\dot{\gamma})$ (the cubic differential evaluated at the velocity vector $\dot{\gamma}$) is a unimodular complex constant. If furthermore $\ve{\phi}(\dot{\gamma})\equiv-1$, we call $\gamma$ a \emph{negative ray}.

\pt A \emph{sector} (resp.  \emph{half-band}) in $(\Sigma,\ve{\phi})$ is an open set whose closure is isometric to a convex region in the Euclidean plane bounded by two rays issuing from a point (resp. two parallel rays and the line segment joining their endpoints).


\pt Geodesic rays $\alpha$ and $\beta$ are said to be \emph{coterminal} if they coincide up to a shift of parameter, \ie if $\alpha(t)=\beta(t+t_0)$ for some $t_0\in\mathbb{R}$ and all $t\geq\max\{-t_0,0\}$. If $\alpha$ and $\beta$ are either coterminal to each other, or coterminal to the two boundary rays of a half-band respectively, then they are said to be \emph{parallel}. 

The rough idea of defining the bordification $\Sigma'$ at a high-order pole $p$ is to make every geodesic ray tending to $p$ converge to a limit in $\pa\Sigma'$, in such a way that the limit distinguishes negative rays but not non-negative rays. A formal definition is:

\begin{definition}[\textbf{Bordification at  pole of order $\geq4$}]\label{def_bord1}
Given a Riemann surface $\Sigma$, a puncture $p$ of $\Sigma$ and a holomorphic cubic differential $\ve{\phi}$ with a pole of order $n+3$ at $p$ ($n\geq1$), we let $\G_p$ denote the set of all geodesic rays on $(\Sigma,\ve{\phi})$ tending to $p$, and ``$\sim$'' denote the equivalence relation defined on $\G_p$ by:
\begin{itemize}
	\item[-] if $\alpha$ is negative, $\alpha\sim\beta$ means $\alpha$ and $\beta$ are coterminal;
	\item[-] otherwise, $\alpha\sim\beta$ means $\alpha$ and $\beta$ are parallel, respectively, to the two boundary rays of a sector not containing any negative ray.
\end{itemize}

The \emph{negative ray bordification} $\Sigma'$ of $(\Sigma,\ve{\phi})$ at $p$ is the set $\Sigma\cup(\G_p/\sim)$ endowed with the topology generated by the open subsets of $\Sigma$ together with all subsets of $\Sigma' $ of the form $V\cup\pa_\infty V$, where 
\begin{itemize}
\item[-]
$V$ is a half-band or a sector such that both boundary rays of $V$ are negative rays tending to $p$;
\item[-]
$\pa_\infty V$ is the subset of $\G_p/\sim$ given by all geodesic rays contained in $V$. 
\end{itemize}
\end{definition}

To justify the definition, we shall show:
\begin{proposition}\label{prop_bordification}
$\Sigma' $ is a bordification of $\Sigma$ at $p$ (in the sense of Section \ref{subsec_puncture}), and has the following properties:
\begin{enumerate}
\item\label{item_bord1}
Every $\gamma\in\G_p$, viewed as a path on $\Sigma$, converges in $\Sigma'$ to the point $[\gamma]\in\G_p/\sim=\pa\Sigma'$ represented by $\gamma$ itself.
\item\label{item_bord2}
The set $\Lambda:=\{[\gamma]\mid \gamma\in\G_p\mbox{ is not negative}\}$ has $n$ points. 
\item\label{item_bord3}
Given negative rays $\alpha,\beta\in\G_p$, their limits $[\alpha],[\beta]\in\pa\Sigma'\setminus\Lambda$ are in the same connected component of $\pa\Sigma'\setminus\Lambda$ if and only if $\alpha$ and $\beta$ are parallel.
\end{enumerate}
\end{proposition}
The proof is given in the next section. By the last property, every connected component of $\pa\Sigma'\setminus\Lambda$ carries a natural metric isometric to the real line, given by 
specifying the distance between $[\alpha]$ and $[\beta]$ to be the width of a half-band whose boundary rays are coterminal to $\alpha$ and $\beta$ (this might not be the plain distance between $\alpha$ and $\beta$, as they can start from the same point but still be parallel with positive distance). 

The negative ray bordification $\mathbb{C}'$ of $(\mathbb{C},\dz^3)$ at $\infty$, which is actually a compactification, is essentially so defined that the projectivized \c{T}i\c{t}eica affine spherical embedding $\delta_0:\mathbb{C}\overset\sim\rightarrow \Delta\subset\mathbb{RP}^2$ (see Section \ref{subsec_titeica}) extends to a homeomorphism from $\mathbb{C}'$ to the closed triangle $\overline{\Delta}$. In the rest of this section we give details of this example.
\begin{example}\label{example_cprime}
The cubic differential $\dz^3$ on $\mathbb{C}$ has a pole of order $6$ at  $\infty$. The set $\G=\G_\infty$ of geodesic rays tending to $\infty$ consists of all rays of the form
$$
\gamma_{z,v}:[0,+\infty)\rightarrow\mathbb{C},\ \  \gamma_{z,v}(t)=z+vt,
$$ 
where $v\in\mathbb{S}^1:=\{v\in\mathbb{C}\mid |v|=1\}$. The definition of negative/nonnegative rays and the equivalence relation ``$\sim$'' amounts to:
\begin{itemize}
\item
$\gamma_{z,v}$ is negative if $v^3=-1$, \ie if $v=e^{\pm\pi\ima/3}$ or $-1$; 
\item
two nonnegative rays $\gamma_{z,v}$ and $\gamma_{z',v'}$ are $\sim$-equivalent if $v$ and $v'$ are in the same connected component of $\mathbb{S}^1\setminus\{v\in\mathbb{C}\mid v^3=-1\}$. 
\end{itemize}
The first picture of Figure \ref{figure_continuity} shows how to understand $\pa_\infty\mathbb{C}:=\G/\sim$. Here, given $v\in\mathbb{S}^1$, $\pai{v}\mathbb{C}$ denotes the set of points in $\pa_\infty\mathbb{C}$ represented by all $\gamma_{z,v}$. Then $\pa_\infty^{(v)}\mathbb{C}$ is a single point exactly when $v^3\neq-1$, and we have  
$$
\Lambda:=\{[\gamma]\in\G/\sim\mid \mbox{ $\gamma\in\G$ is nonnegative} \}=\{\pai{1}\mathbb{C},\,\pai{\omega}\mathbb{C},\,\pai{\omega^2}\mathbb{C}\},
$$
where $\omega:=e^{2\pi\ima/3}$, whereas each $\pa_\infty^{(v)}\mathbb{C}$ with $v^3=-1$ can  be identified with the real line, with points in one-to-one correspondence with coterminal classes of negative rays in the direction $v$. The identification can be written explicitly as
\begin{equation}\label{eqn_parametrization}
\mathbb{R}\overset\sim\rightarrow \pai{v}\mathbb{C},\quad s\mapsto [\gamma_{\ima sv ,v}(t)].
\end{equation}
 
On the other hand, we have the following coordinate expression of the projectivized \c{T}i\c{t}eica affine sphere embedding (see Section \ref{subsec_titeica}):
\begin{align*}
\delta_0:\mathbb{C}&\overset\sim\longrightarrow\Delta=\{[x_1:x_2:x_3]\mid x_i>0\}\subset\mathbb{RP}^2\\ 
\delta_0(z)&=[e^{\sqrt{2}\re(z)}:e^{\sqrt{2}\re(\omega^2 z)}:e^{\sqrt{2}\re(\omega z)}].
\end{align*}

Let us find the limit of $\delta_0$ along each $\gamma_{z,v}$. 
Note that $\delta_0$ maps every negative ray $\gamma_{z,v}$ with $v=-1$ (resp. $v=e^{-\pi\ima/3}$, $v=e^{\pi\ima/3}$) to a straight ray in $\Delta$ issuing from $[1:0:0]$ (resp. $[0:1:0]$, $[0:0:1]$), as Figure \ref{figure_dev01} shows. Therefore, the limit along all such rays form an edge of $\Delta$. More explicitly, for any $s\in\mathbb{R}$, one checks that
\begin{equation}\label{eqn_limit2}
\lim_{t\rightarrow+\infty}\delta_0(\gamma_{\ima sv ,v}(t))=
\begin{cases}
[1:e^{\sqrt{6}s}:0] &\mbox{ if }v=e^{\pi\ima/3}\\
[0:1:e^{\sqrt{6}s}] &\mbox{ if }v=-1\\
[e^{\sqrt{6}s}:0:1] &\mbox{ if }v=e^{5\pi\ima/3}
\end{cases}
\end{equation}
When $v^3\neq-1$, one checks that the limit of $\delta_0(\gamma_{z,v}(t))$ is a vertex of $\Delta$:
$$
\lim_{t\rightarrow+\infty}\delta_0(\gamma_{z ,v}(t))=
\begin{cases}
[1:0:0] &\mbox{ if }\arg(v)\in (-\tfrac{\pi}{3},\tfrac{\pi}{3}),\\
[0:1:0] &\mbox{ if }\arg(v)\in (\tfrac{\pi}{3},\pi),\\
[0:0:1] &\mbox{ if }\arg(v)\in (\pi,\tfrac{5\pi}{3}).
\end{cases}
$$
Therefore, we can define a natural bijective extension of $\delta_0$ to $\mathbb{C}'=\mathbb{C}\cup\pa_\infty\mathbb{C}$ by
$$
\overline{\delta}_0:\mathbb{C}'\rightarrow\overline{\Delta},\quad \overline{\delta}_0|_{\mathbb{C}}:=\delta_0,\ \  \overline{\delta}_0([\gamma]):= \lim_{t\rightarrow+\infty}\delta_0(\gamma(t))\text{ for all }\gamma\in\G.
$$
See Figure \ref{figure_continuity}. Equipping $\mathbb{C}'$ with the topology from Definition \ref{def_bord1}, we shall show:
\begin{lemma}\label{lemma_deltaprime}
The map $\overline{\delta}_0$ is a homeomorphism. 
\end{lemma}
This implies Proposition \ref{prop_bordification} in the case $(\Sigma,\ve{\phi})=(\mathbb{C},\dz^3)$.
\begin{proof}
The topology defined on $\mathbb{C}'$ is clearly Hausdorff. Since a continuous bijection from a compact space to a Hausdorff space is a homeomorphism, we only need to show that the inverse $\overline{\delta}_0^{-1}$ is continuous, or equivalently, $\overline{\delta}_0$ maps every $V\cup\pa_\infty V\subset\mathbb{C}'$ in Definition \ref{def_bord1} to an open subset of $\overline{\Delta}$. Since modifying $V$ within a bounded set of $\mathbb{C}$ does not effect the latter statement, we may suppose that $V\subset\mathbb{C}$ is either an angular region or a half-band bounded by negative geodesics, typical examples of which are shown as $V_1$ and $V_2$ in Figure \ref{figure_continuity}.
\begin{figure}[h]
\centering\includegraphics[width=4.7in]{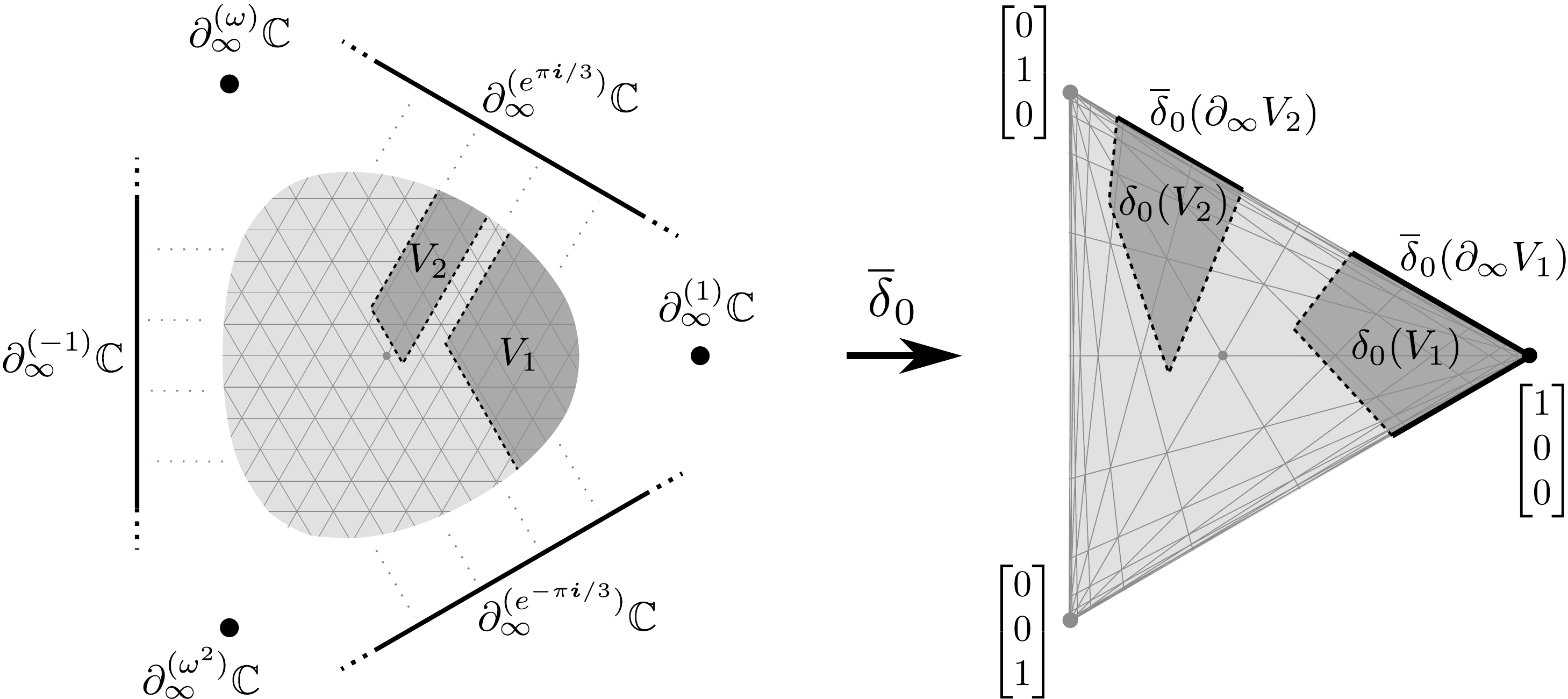}
\caption{The compactification $\mathbb{C}'$ and the proof of Lemma \ref{lemma_deltaprime}.}
\label{figure_continuity}
\end{figure}
By virtue of the fact that $\delta_0$ maps every negative geodesic to a line in $\Delta$ with the limit behavior computed above, $\overline{\delta}_0(V\cup\pa_\infty V)$ is formed by the region $\delta_0(V)\subset \Delta$ bounded by lines together with the open boundary part of the region, as shown by the second picture in Figure \ref{figure_continuity}. So $\overline{\delta}_0(V\cup\pa_\infty V)$ is open, as required.
\end{proof}
\end{example}

The extension $\overline{\delta}_0$ has the property that it preserves boundary metrics. Indeed, each $\pai{v}\mathbb{C}$ with $v^3=-1$ is a connected component of $\pa_\infty\mathbb{C}\setminus\Lambda$ and the parametrization (\ref{eqn_parametrization}) is isometric with respect to the metric defined in the paragraph following Proposition \ref{prop_bordification}.
On the other hand, every open line segment $I\subset\mathbb{RP}^2$ with endpoints $a\neq b$ carries a metric given by the logarithm of cross-ratio (see Section \ref{subsec_hilbert}). If we consider the rescaled metric
$$
\widetilde{d}_I(x,y):=\tfrac{1}{\sqrt{6}}|\log[a,x,y,b]| \ \text{ for all } x,y\in I,
$$
then, using the limit expressions (\ref{eqn_limit2}), one checks that $\overline{\delta}_0$ restricts to an isometry from $\pai{v}\mathbb{C}$ to an open edge $I$ of $\Delta$ equipped with the metric $d_I$.

\subsection{Bordified half-plane charts and proof of Proposition \ref{prop_bordification}}\label{subsec_bordified}
Let $\H'$ denote the open subset of $\mathbb{C}'$ formed by the right half-plane $\H=\{z\in\mathbb{C}\mid \re(z)>0\}$ and the boundary part 
$$
\pa_\infty\H:=\pai{e^{\pi\ima/3}}\mathbb{C}\cup \{\pai{1}\mathbb{C}\}\cup\pai{e^{-\pi\ima/3}}\mathbb{C}\subset\pa_\infty\mathbb{C}.
$$
The homeomorphism $\overline{\delta}_0:\mathbb{C}'\rightarrow\overline{\Delta}$ maps $\H'$ to an open subset of the closed triangle $\overline{\Delta}$, sending $\pai{e^{\pi\ima/3}}\mathbb{C}$ and $\pai{e^{-\pi\ima/3}}\mathbb{C}$ isometrically to the open edges 
$$
I_+:=\{[1:a:0]\mid a>0\},\quad I_-:=\{[a:0:1]\mid a>0\}
$$
of the triangle, respectively, and sending the point $\pai{1}\mathbb{C}$ to the vertex $[1:0:0]$. 
In other to prove proposition \ref{prop_bordification}, we shall show that the local model of $(\Sigma,\ve{\phi})$ around $p$ formed by $n$ copies of $\H$ (see Theorem \ref{thm_halfplanes}) extends to a local model of $\Sigma'$ around the boundary formed by $n$ copies of the bordified half-plane $\H'$. In particular, this provides $\Sigma'$ with manifold charts.

\begin{proof}[Proof of Proposition \ref{prop_bordification}]
By Theorem \ref{thm_halfplanes}, a closed neighborhood of $p$ is formed by the half-planes $\overline{H}_1,\cdots, \overline{H}_n$. We view every geodesic ray in $\overline{H}_j$, which has the form $\gamma_{z,v}:[0,+\infty)\rightarrow \CH\cong \overline{H}_j$ with $z\in\CH$, $\arg(v)\in[-\pi,\pi]$ (see Example \ref{example_cprime} for the notation), as an element of $\G_p$. Conversely, one sees from the gluing pattern of the half-planes that every $\gamma\in\G_p$ is coterminal to a ray $\gamma_{z,v}$ in some $H_j$ with $\arg(v)\in(-\frac{\pi}{3}, \frac{\pi}{3}]$.

To each $H_j$ is associated a subset $\pa_\infty H_j$ of $\G_p/\sim$ similar to the set $\pa_\infty\H\subset\pa_\infty\mathbb{C}$ defined above. Namely, let $\pa_\infty^\pm H_j$ (resp.  $\pai{1}H_j$) denote the subset (resp. the point) in $\G_p/\sim$ given by the rays $\gamma_{z,v}$ in $H_j$ with $v=e^{\pm\pi\ima/3}$ (resp. $v=1$). We then put 
$$
\pa_\infty H_j:=\pa_\infty^+ H_j\cup \pai{1} H_j\cup \pa_\infty^- H_j\,,\quad H'_j:=H_j\cup\pa_\infty H_j.
$$
By definition of the topology on $\Sigma'$, $H_j'$ is open in $\Sigma'$  because it is the union of all subsets of the form $V\cup\pa_\infty V$, where $V\subset H_j$ is a translation of the open sector in $H_j$ bounded by $\gamma_{0,e^{\pi\ima/3}}$ and $\gamma_{0,e^{-\pi\ima/3}}$. The definition of the topology also implies that the obvious identification between $H'_j$ and the bordified half-plane $\H'=\H\cup\pa_\infty\H$ is a homeomorphism.  

Thus, $H_1',\cdots, H_n'$ form an open covering of a neighborhood of $\G_p/\sim$. Since $H_j'$ is homeomorphic through $\overline{\delta}_0$ to the open subset $\overline{\delta}_0(\H')$ of $\overline{\Delta}$, which is in turn homeomorphic to a closed half-plane, $\Sigma'$ is a topological manifold with boundary.

To show that the boundary $\pa\Sigma'=\G_p/\sim=\bigcup_{j=1}^n\pa_\infty H_j$ is homeomorphic to a circle, note that on one hand, each $H_j'\cong\H'$ is assembled with $H_{j+1}'\cong\H'$ through the self-homeomorphism of $\mathbb{C}'$ induced by an automorphism of $(\mathbb{C},\dz^3)$ of the form $z\mapsto e^{2\pi\ima/3}z+z_0$, which sends the boundary part  $\pa_\infty^- H_{j+1}\cong\pai{e^{-\pi\ima/3}}\mathbb{C}$ to $\pa_\infty^+ H_j\cong\pai{e^{\pi\ima/3}}\mathbb{C}$; on the other hand, $\pa_\infty H_j$ is topologically identified through $\overline{\delta}_0$ with the ``$V$ shape'' on the boundary of $\Delta$ formed by the open edges $I_\pm$ and the vertex $[1:0:0]$, wtih $\pa_\infty^\pm H_j$ corresponding to $I_\pm$. Therefore, $\pa\Sigma'$ is topologically the $1$-manifold obtained by assembling $n$ copies of the $V$ shape in such a way that the $I_+$-part in each copy is glued with the $I_-$-part in the next copy by a rotation. Thus, $\pa\Sigma'$ is a circle.
 
Property (\ref{item_bord1}) in the statement of the proposition is clear from definition. Therefore, $\Sigma'$ is a bordification of $\Sigma$ at $p$ in the sense of Section \ref{subsec_puncture}. A geodesic ray $\gamma\in\G_p$ is nonnegative if and only if $\gamma$ is coterminal to some $\gamma_{z,v}$ in $H_j$ with $\arg(v)\in(-\frac{\pi}{3}, \frac{\pi}{3})$, which implies $[\gamma]=\pa_\infty^{(1)}H_j$. So the set $\Lambda\subset\pa\Sigma'$ given by nonnegative rays meets each $H_j'$ exactly at the point $\pa_\infty^{(1)}H_j$, whence Property (\ref{item_bord2}). Finally, $\alpha,\beta\in\G_p$ are parallel negative rays if and only if they are coterminal to rays in the same $H_j$ with direction $v=e^{\pi\ima/3}$, which means $[\alpha]$ and $[\beta]$ both belong to the component $\pa_\infty^+H_j$ of $\pa\Sigma'\setminus\Lambda$. This proves Property (\ref{item_bord3}) and completes the proof of the proposition.
\end{proof}
\subsection{Bordification at a third order pole}\label{subsec_bord2}
Now assume $\ve{\phi}$ has a third order pole at $p$. We define in this section the bordifications $\Sigma'_\gamma$ of $\Sigma$ used in Part (\ref{item_thmintro32}) of Theorem \ref{thm_intro3}. To this end, we first note
that $(\Sigma,\ve{\phi})$ locally looks like a cylinder around $p$:
\begin{proposition}\label{prop_cylinder}
	Any regular $\frac{1}{3}$-translation end of degree $-3$ is equivalent to the cylinder
	$\tau\overline{\mathbb{H}}/\tau\mathbb{Z}$
	for some $\tau\in\mathbb{C}^*$. Two such cylinders $\tau\overline{\mathbb{H}}/\tau\mathbb{Z}$ and $\tau'\overline{\mathbb{H}}/\tau'\mathbb{Z}$ are equivalent if and only if $\tau^3=\tau'^3$.
\end{proposition}
Here $\overline{\mathbb{H}}$ is the closure of the upper half-plane $\mathbb{H}\subset\mathbb{C}$, hence $\tau\overline{\mathbb{H}}$ stands for the closed half-plane $\{\tau z\mid \im(z)\geq 0\}$ in $(\mathbb{C},\dz^3)$. Quotienting by the group $\tau\mathbb{Z}$ means quotienting by the translation $z\to z+ \tau$, which is the monodromy of the cylinder. Therefore, the above proposition is an immediate consequence of Corollary \ref{coro_model}. Note that in this case we can write the cubic differential into normal form (see Theorem \ref{thm_norm}) as
$$
\tau\overline{\mathbb{H}}/\tau\mathbb{Z}\cong(\{0<|w|\leq1\},Rw^{-3}\dif w^3),
$$
where the residue $R$ equals $\big(\tau/2\pi\ima\big)^3$ by Lemma \ref{lemma_invariants}. For our purpose, a crucial remark is that the negative rays in the cylinder behave differently when the real part of $R$ is negative, zero and positive, respectively, as one can see from Figure \ref{figure_thirdorder}: When $\re(R)<0$, there exist two negative rays tending to the puncture at infinity that are not parallel to each other, but when $\re(R)\geq0$ all such rays are parallel. The case $\re(R)=0$ is characterized by the existence of a closed negative loop.
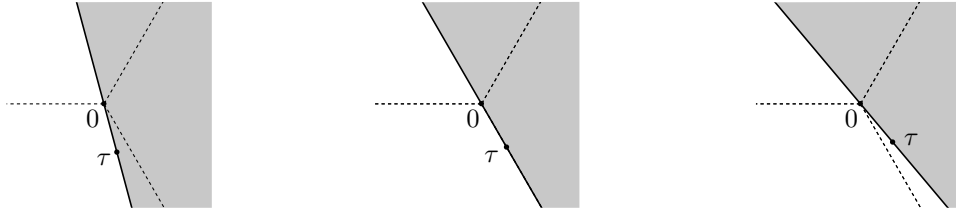
\begin{figure}[h]
	\labellist
	\pinlabel {$0$} at 134 134
	\pinlabel {$0$} at 667 134
	\pinlabel {$0$} at 1197 134
	\pinlabel {$\tau$} at 150 76
	\pinlabel {$\tau$} at 693 80
	\pinlabel {$\tau$} at 1281 104
	\endlabellist
	\centering
	\includegraphics[width=5.1in]{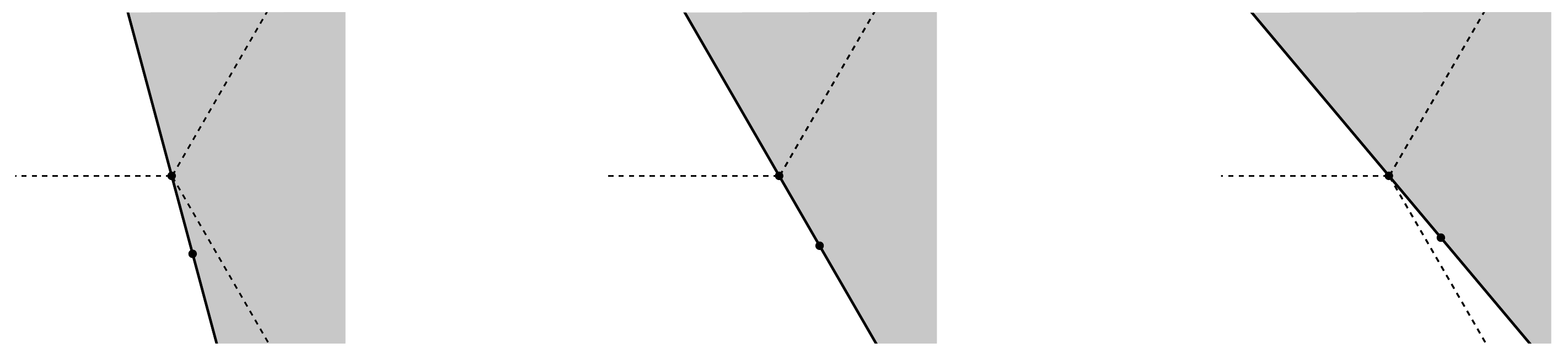}
	\caption{The negative rays in $(\mathbb{C},\dz^3)$ issuing from $0$ and examples of $\tau\overline{\mathbb{H}}$ with $\re(R)=\re((\tau/2\pi\ima)^3)$  negative, zero and positive, respectively.}
	\label{figure_thirdorder}
\end{figure}


In view of this remark and the terminologies from Section \ref{subsec_bord1}, we define: 
\begin{definition}[\textbf{Bordification at third order pole}]
\label{def_bord2}
Given a Riemann $\Sigma$, a puncture $p$ of $\Sigma$, a holomorphic cubic differential $\ve{\phi}$ with a third order pole at $p$, and a negative ray $\gamma$ tending to $p$, we let $\G_\gamma$ denote the set of all geodesic rays on $(\Sigma,\ve{\phi})$ parallel to $\gamma$, and let ``$\sim$'' denote the coterminal equivalence relation on $\G_\gamma$. The \emph{negative ray bordification} $\Sigma'_\gamma$ of $(\Sigma,\ve{\phi})$ at $p$, relative to $\gamma$, is the set $\Sigma\cup(\G_\gamma/\sim)$ endowed with the topology generated by the open subsets of $\Sigma$ and all subsets of $\Sigma'_\gamma$ of the form $V\cup\pa_\infty V$, where $V$ is a half-band whose boundary rays are both in $\G_\gamma$, and $\pa_\infty V\subset\G_\gamma/\sim$ is given by rays contained in $V$.
\end{definition}
Thus, unlike the case of higher order poles, where every geodesic ray tending to $p$ converges to a boundary point of the bordification, in the current case we only assign limits to a chosen family of parallel negative rays, and when $\re(R)<0$ (resp. $\re(R)\geq0$) there are two (resp. only one) choice(s).

Similarly as in the case of higher order poles, $\Sigma'_\gamma$ is locally modeled on the compactification $\mathbb{C}'$ of $(\mathbb{C},\dz^3)$. More precisely, given a cubic root  $v$ of $-1$ in an open half-plane $\tau\mathbb{H}$ as above, we can bordify the half-plane by attaching to it the boundary segment $\pa_\infty^{(v)}\mathbb{C}$ of $\mathbb{C}'$(see Section \ref{subsec_bord1}), such that the quotient $(\tau\mathbb{H}\cup \pa_\infty^{(v)}\mathbb{C})/\tau\mathbb{Z}$
of the bordified half-plane is a bordification of the cylinder $\tau\mathbb{H}/\tau\mathbb{Z}$. This bordified cylinder provides a local model for $\Sigma'_\gamma$:
\begin{proposition}\label{prop_bord2}
In the setting of Definition \ref{def_bord2}, if we identify a neighborhood of $p$ with the (open) cylinder $\tau\mathbb{H}/\tau\mathbb{Z}$ and let $v\in\tau\mathbb{H}$, $v^3=-1$ be such that $\gamma$ is given by $t\mapsto z_0+tv$ in the cylinder, then $\Sigma'_\gamma$ is the surface with boundary obtained by replacing the cylinder with the bordified cylinder $(\tau\mathbb{H}\cup\pa_\infty^{(v)}\mathbb{C})/\tau\mathbb{Z}$.
\end{proposition}
The proof is similar is spirit to the previous two sections and we omit the details.

\section{$\mathbb{RP}^2$-structures around higher order poles}
\label{sec_6}
In this chapter, we fix a Riemann surface $\Sigma$, a puncture $p$ of $\Sigma$, a holomorphic cubic differential $\ve{\phi}$ on $\Sigma$ with a pole of order $m\geq 3$ at $p$, and let $g$ be the complete solution to the vortex equation on $(\Sigma,\ve{\phi})$ (see Theorem \ref{thm_ue}).
Our goal is to prove Theorem \ref{thm_intro3} about extensibility of the convex $\mathbb{RP}^2$-structure $X:=X(g,\ve{\phi})$ (see Section \ref{subsec_convexity}) to the bordifications of $\Sigma$ from the last chapter, as well as the last statement in Theorem \ref{thm_intro2} about eigenvalues and principality of geodesic boundary.

\subsection{Reduction to local statements}\label{subsec_reduction}
First assume $m=n+3\geq 4$, in which case the boundary $\pa\Sigma'\approx\mathbb{S}^1$ has a subset $\Lambda$ of $n$ marked points, given by the limits of non-negative rays. By Theorem \ref{thm_halfplanes}, $p$ has a neighborhood formed by $n$ copies $H_1,\cdots, H_n$ of the right half-plane $\H$, and it is showed in Section \ref{subsec_bordified} that we can attach  to each $H_i$ a topological open interval $\pa_\infty H_i\subset\pa\Sigma'$ of the negative ray bordification, so that $H_i'=H_i\cup\pa_\infty H_i$ ($i=1,\cdots,n$) form an open covering of a neighborhood of $\pa\Sigma'$ in $\Sigma'$, with the following properties: 

\pt Each $\pa_\infty H_i$ is separated by the point $\pai{1}H_i=\Lambda\cap\pa_\infty H_i$ into sub-intervals $\pa_\infty^+H_i$ and $\pa_\infty^-H_i$, such that $\pa_\infty^+ H_i=\pa_\infty^- H_{i+1}$ is a connected component of $\pa\Sigma'\setminus\Lambda$ and carries the metric given by parallel distances of negative rays.

\pt The identification $H_i\cong \H$ extends to $H'_i=H_i\cup\pa_\infty H_i\cong\H'=\H\cup\pa_\infty\H\subset\mathbb{C}'$, identifying $\pa_\infty^\pm H_i$ isometrically with $\pai{e^{\pm\pi\ima/3}}\mathbb{C}\subset\pa_\infty\mathbb{C}$ (here $\mathbb{C'}=\mathbb{C}\cup\pa_\infty\mathbb{C}$ is the negative ray compactification of $(\mathbb{C},\dz^3)$ and $\pa_\infty\H$ is the boundary part of $\H$ in $\pa_\infty\mathbb{C}$, see Sections \ref{subsec_bord1} and \ref{subsec_bordified}).

\pt Writing $g=e^u|\dz|^2$ on $H_i\cong\H$, we have $0\leq u(z)\leq C|z|^\frac{1}{2}e^{-\sqrt{6}|z|}$ when $|z|$ is large enough (see Corollary \ref{coro_halfplaneestimate} and Remark \ref{remark_stronger}).

In view of these properties, Part (\ref{item_thmintro31}) of Theorem \ref{thm_intro3} is an immediate consequence of the following local result, proved in Section \ref{subsec_reduction1} below, applied to every $H_i$:
\begin{proposition}\label{prop_reduction1}
Let $g=e^u|\dz|^2$ be a solution to Wang's equation on $(\H,\dz^3)$ satisfying
$$
0\leq u(z)\leq C|z|^\alpha e^{-\sqrt{6}|z|} \ \text{ for all } z\in\H, |z|\geq r
$$  
for some constants $C,\alpha, r>0$ and $\iota:\H\rightarrow\mathbb{R}^3$ be an affine spherical embedding with Blaschke metric $g$ and normalized Pick differential $\dz^3$ such that $\delta:=\mathbb{P}\circ\iota:\H\rightarrow\mathbb{RP}^2$ is injective. Then $\delta$ extends to a continuous injection $\overline{\delta}:\H'\rightarrow\mathbb{RP}^2$ sending $\pai{e^{\pi\ima/3}}\mathbb{C}$ and $\pai{e^{-\pi\ima/3}}\mathbb{C}$ isometrically to open line segments in $\mathbb{RP}^2$ not collinear to each other.
\end{proposition}
Here, we endowed every open line segment in $\mathbb{RP}^2$ with the metric given by $\frac{1}{\sqrt{6}}$ times the logarithm of cross-ratio, as at the end of Section \ref{subsec_bord1}. 

We postpone the proof of this proposition and the next two propositions to the subsequent sections, and proceed with the case $m=3$. In this case, we identify a neighborhood of $p$ with the cylinder $\tau\mathbb{H}/\tau\mathbb{Z}$, where $\tau\in\mathbb{C}^*$ is related to the residue $R$ of $\ve{\phi}$ at $p$ by $R=(\tau/2\pi\ima)^3$ (see Section \ref{subsec_bord2}). Around $p$, the convex $\mathbb{RP}^2$-structure $X$ is given by an affine spherical embedding  $\iota:\tau\mathbb{H}\rightarrow\mathbb{R}^3$, 
and there is some $H\in\SL(3,\mathbb{R})$, namely the holonomy of $X$ around $p$, such that
$$
\iota(z+\tau)=H(\iota(z))\ \text{ for all }z\in\tau\mathbb{H}.
$$ 

A technical issue in this case is that we now only have a weaker upper bound for the function $u$ in the Blaschke metric $e^u|\dz|^2$, with the term $|z|$ in the above estimate replaced by the distance from $z$ to the boundary $\tau\mathbb{R}$ or $\pa\mathbb{H}$, namely,
$$
u(z)\leq C|d(z,\tau\mathbb{R})|^\frac{1}{2}e^{-\sqrt{6}\,d(z,\tau\mathbb{R})},
$$ 
Thus, when $z$ goes to $\infty$ along a ray in $\tau\mathbb{H}$, the smaller the angle from the ray to $\tau\mathbb{R}$, the weaker the decay. Our techniques only enable us to prove a similar extensibility result as Proposition \ref{prop_reduction1} when this angle is bigger than $\arcsin(\frac{\sqrt{3}}{4})\approx 0.14\pi$:
\begin{proposition}\label{prop_reduction2}
Suppose $\theta_1\in(\arcsin(\frac{\sqrt{3}}{4}),\frac{\pi}{2}]$ and $v\in\mathbb{C}$, $v^3=-1$. Let $g=e^u|\dz|^2$ ($u\geq0$) be a solution to Wang's equation on the sector 
$$
V_1:=\{te^{\theta\ima}v\mid t>0,\ |\theta|<\theta_1\}
$$
with cubic differential $\dz^3$, and let $\iota:V_1\rightarrow\mathbb{R}^3$ be an affine spherical embedding with Blaschke metric $g$ and normalized Pick differential $\dz^3$, such that $\delta:=\mathbb{P}\circ\iota:V_1\to\mathbb{RP}^2$ is injective. Then $\delta$ extends to a continuous injection $\overline{\delta}$ from $V_1\cup\pai{v}\mathbb{C}\subset\mathbb{C}'$ to $\mathbb{RP}^2$, sending $\pai{v}\mathbb{C}$ isometrically to an open line segment.
\end{proposition}
Returning to the cylinder neighborhood  $\tau\mathbb{H}/\tau\mathbb{Z}$ of $p$, we may suppose that the negative ray $\gamma$ in Theorem \ref{thm_intro3} (\ref{item_thmintro32}) is given by $t\mapsto z_0+tv$ in the cylinder (where $v^3=-1$). The assumption on $\gamma$ in the theorem means exactly that the above sector $V_1$, defined from this $v$ and some $\theta_1$ verifying the assumption of the proposition, is contained in $\tau\mathbb{H}$. The proposition then yields an extension $\overline{\delta}:\tau\mathbb{H}\cup\pai{v}\mathbb{C}\rightarrow\mathbb{RP}^2$ of $\delta=\mathbb{P}\circ\iota$ whose restriction to $V_1\cup\pai{v}\mathbb{C}$ is continuous. But then $\overline{\delta}$ itself is continuous because $\tau\mathbb{H}\setminus V_1$ does not have adherent points in $\pai{v}\mathbb{C}$. Since the bordification $\Sigma'_\gamma$ is obtained by replacing $\tau\mathbb{H}/\tau\mathbb{Z}$ with $(\tau\mathbb{H}\cup\pa_\infty^{(v)}\mathbb{C})/\tau\mathbb{Z}$ (Proposition \ref{prop_bord2}), $\overline{\delta}$ gives the required extension of the convex $\mathbb{RP}^2$-structure $X$ and proves Theorem \ref{thm_intro3} (\ref{item_thmintro32}).

Finally, we need the following result on $H$ for the last statement in Theorem \ref{thm_intro2}:
\begin{proposition}\label{prop_reduction3}
Let $\tau\in\mathbb{C}^*$ and let $v$ be a cubic root of $-1$ contained in the half-plane $\tau\mathbb{H}$, with distance to the boundary $\tau\mathbb{R}$ of the half-plane greater than $\frac{\sqrt{3}}{4}$. Suppose
\begin{itemize}
\item
$g=e^{u(z)}|\dz|^2$ is a solution to Wang's equation on $(\tau\mathbb{H},\dz^3)$  such that $u\geq0$ and $u(z+\tau)=u(z)$ for all $z$.
\item
$\iota:\tau\mathbb{H}\rightarrow\mathbb{R}^3$ is an affine spherical embedding with Blaschke metric $g$ and normalized Pick differential $\dz^3$, such that $\delta:=\mathbb{P}\circ\iota:\tau\mathbb{H}\rightarrow\mathbb{RP}^2$ is injective. 
\item
$\overline{\delta}:\tau\mathbb{H}\cup\pai{v}\mathbb{C}\rightarrow\mathbb{RP}^2$ is the extension of $\delta$ provided by Proposition \ref{prop_reduction2}. 
\item
$x^+$ and $x^-$ are the endpoints of the open line segment $I:=\overline{\delta}(\pai{v}\mathbb{C})$ in $\mathbb{RP}^2$, and more specifically $x^+$ (resp. $x^-$) is the limit of $\overline{\delta}(\zeta)$ as $\zeta\in\pai{v}\mathbb{C}$ tends to $\pai{e^{\pi\ima/3}v}\mathbb{C}$ (resp. $\pai{e^{-\pi\ima/3}v}\mathbb{C}$).
\end{itemize}
Then $\overline{\delta}$ is injective and there is a unique $H\in\SL(3,\mathbb{R})$ such that $\iota(z+\tau)=H(\iota(z))$ for all $z\in \tau\mathbb{H}$, which preserves $I$ and has eigenvalues $\exp\big(\sqrt{2}\re(e^{\mp\frac{\pi\ima}{3}}v^{-1}\tau)\big)$ at $x^\pm$.
\end{proposition}
\begin{proof}[Proof of the last statement in Theorem \ref{thm_intro2}]
  Since $e^{\frac{\pi\ima}{3}}v^{-1}$ and $e^{-\frac{\pi\ima}{3}}v^{-1}$ are different cubic roots of unity, Proposition \ref{prop_reduction3} implies that the three eigenvalues of $H$ are $e^{\sqrt{2}\re(\tau)}$, $e^{\sqrt{2}\re(\omega\tau)}$ and $e^{\sqrt{2}\re(\omega^2\tau)}$ (where $\omega=e^{\frac{2\pi\ima}{3}}$), which are exactly the values claimed in Theorem \ref{thm_intro2} because of the relation $R=\left(\frac{\tau}{2\pi\ima}\right)^3$. These eigenvalues are distinct if and only if $\tau^3\neq\pm1$, or equivalently, $\re(R)\neq0$. 
%
%


It is elementary to check, for $v=e^{\frac{\pi\ima}{3}},-1$ and $e^{-\frac{\pi\ima}{3}}$ respectively, that if $\re(R)>0$, or equivalently, if $\arg(\tau)\in (-\frac{\pi}{3},0)\cup(\frac{\pi}{3}, \frac{2\pi}{3})\cup(\pi, \frac{4\pi}{3})$, then $\exp\big(\sqrt{2}\re(e^{\frac{\pi\ima}{3}}v^{-1}\tau)\big)$ and $\exp\big(\sqrt{2}\re(e^{-\frac{\pi\ima}{3}}v^{-1}\tau)\big)$ are the biggest and smallest eigenvalues of $H$; whereas if $\re(R)< 0$ then one of the two is the intermediate eigenvalue. This shows that $X'$ has principal geodesic boundary if and only if $\re(R)>0$, completing the proof.
\end{proof}

To explain the idea of proof of the above propositions, we first note that if $g=|\dz|^2$, so that the affine spherical embedding in question and its projectivization are the maps $\iota_0$ and $\delta_0$ studied in Sections \ref{subsec_titeica} and \ref{subsec_bord1}, the propositions can be verified with the explicit expressions. For general $g$, we study the \emph{osculation map} introduced in \cite{dumas-wolf}, which measures the deviation of $\iota$ from $\iota_0$. The precise definition of this map goes as follows. Given a simply connected domain $U\subset\mathbb{C}$ and a solution $g$ to Wang's equation on $(U,\dz^3)$, we let 
$$
\para_0(z_2,z_1),\para(z_2,z_1):\T_{z_1}U\oplus\mathbb{R}\rightarrow\T_{z_2}U\oplus\mathbb{R}
$$ 
denote the parallel transports of the affine sphere connections $\D_0$ and $\D$, associated to $(|\dz|^2,\dz^3)$ and $(g,\dz^3)$, respectively, along a path from $z_1$ to $z_2$ (see Section \ref{subsec_affine}; be cautious about the order of $z_1$ and $z_2$ in the notations, \cf Appendix \ref{subsec_connection}). By Proposition \ref{prop_recon}, given a base point $z_0\in U$, we can write
$$
\iota_0(z)=\para_0(z_0,z)\underline{1}_z,\ \iota(z)=\para(z_0,z)\underline{1}_z.
$$
We define the \emph{osculation map associated to $g$ at $z_0$} as
\begin{equation}\label{eqn_osc}
P:U\rightarrow\SL(\T_{z_0} U\oplus\mathbb{R}),\ P(z):=\para(z_0,z)\para_0(z,z_0),
\end{equation}
so that $\iota(z)=P(z)\iota_0(z)$ for all $z\in U$. Since $\delta_0=\mathbb{P}\circ\iota_0$ is already well understood, in other to prove the required statements about $\delta(z)=\mathbb{P}(\iota(z))=P(z)\delta_0(z)$, we shall analyze the asymptotic behavior of $P(z)$ as $z$ goes to infinity. This will be done in Sections \ref{subsec_part1} and \ref{subsec_part2} below, for poles of order $m\geq4$ and $m=3$ respectively, after we collect some preliminary facts in Section \ref{subsec_setups}. Propositions \ref{prop_reduction1} -- \ref{prop_reduction3} above are deduced from asymptotic results on $P$ in Sections \ref{subsec_reduction1} and \ref{subsec_reduction2}.

\subsection{Preliminaries}\label{subsec_setups}
By the definition (\ref{eqn_Dlocal}) of affine sphere connections, in the above setting, the connections $\D_0$ and $\D$, extended to the complexified vector bundle $(\T U\oplus\underline{\mathbb{R}})\otimes\mathbb{C}=\T_\mathbb{C}U\oplus\underline{\mathbb{C}}$, are expressed under the frame $(\paz,\pabz,\underline{1})$ as
$$
\D_0=\dif+
\begin{pmatrix}
0&\frac{1}{\sqrt{2}}\dbz&\dz\\[6pt]
\frac{1}{\sqrt{2}}\dz&0&\dbz\\[6pt]
\frac{1}{2}\dbz&\frac{1}{2}\dz&0
\end{pmatrix},\ \ 
\D=\dif+
\begin{pmatrix}
\pa u&\frac{1}{\sqrt{2}}e^{-u}\dbz&\dz\\[6pt]
\frac{1}{\sqrt{2}}e^{-u}\dz&\bpa u&\dbz\\[6pt]
\frac{1}{2}e^{u}\dbz&\frac{1}{2}e^{u}\dz&0
\end{pmatrix}
$$

In the sequel, we work with the global frame $(e_1,e_2,e_3)$ of $\T U\oplus\underline{\mathbb{R}}$ introduced in Section \ref{subsec_titeica}, under which we can write 
$$
\D_0=\dif+
\sqrt{2}
\re
\begin{pmatrix}
\dz&&\\
&\omega^2\dz&\\
&&\omega\dz
\end{pmatrix},\quad
\D=\D_0+\Xi,
$$
where $\Xi:=\D-\D_0\in\Omega^1(U,\End(\T U\oplus\underline{\mathbb{R}}))$ is expressed as a matrix valued $1$-form. By the above expressions of $\D$ and $\D_0$ under the old frame $(\paz,\pabz,\underline{1})$ and the fact that $(e_1,e_2,e_3)$ is related to $(\paz,\pabz,\underline{1})$ by a constant change-of-basis matrix (see Section \ref{subsec_titeica}), the coefficients of every entry of $\Xi$ are $\mathbb{R}$-linear combinations of the four functions $e^u-1$,  $e^{-u}-1$, $\pa_x u$ and $\pa_y u$.

Parallel transports of $\D$ and $\D_0$ are expressed as matrices in $\SL(3,\mathbb{R})$ under the frame $(e_1,e_2,e_3)$. In particular, the osculation map (\ref{eqn_osc}) is considered as $\SL(3,\mathbb{R})$-valued through the identification $\SL(\T_{z_0} U\oplus\mathbb{R})\cong\SL(3,\mathbb{R})$ given by the frame. 

Recall from Section \ref{subsec_titeica} that the parallel transport of $\D_0$ along a path from $z\in U$ to $0$ (assuming $0\in U$) is 
$$
\para_0(0,z)=
\begin{pmatrix}
e^{\sqrt{2} \re(z)}&&\\
&e^{\sqrt{2} \re(\omega^2z)}&\\
&&e^{\sqrt{2} \re(\omega z)}
\end{pmatrix},\ \mbox{ where }\omega=e^{\frac{2\pi\ima}{3}}.
$$
We need some information on eigenvalues of the  automorphism $\Ad_{\para_0(0,z)}$ of $\sl_3\mathbb{R}$. Given $i,j\in\{1,2,3\}$, let $\mu_{ij}$ be the function on $\mathbb{S}^1=\{v\in\mathbb{C}\mid |v|=1\}$ defined by
$$
\mu_{ij}(v):=\sqrt{2}\re\left[(\omega^{1-i}-\omega^{1-j})v\right],
$$
so that the eigenvalue of $\Ad_{\para_0(0,z)}$ on the elementary matrix $E_{ij}\in\sl_3\mathbb{R}$ is 
$$
\lambda_{ij}(z)=e^{\mu_{ij}(z/|z|)|z|}.
$$ 
Also denote $\mu(v):=\max_{i, j}\mu_{ij}(v)=\max_{i\neq j}\mu_{ij}(v)$ for every $v\in\mathbb{S}^1$.

It is useful to visualize $\mu_{ij}(v)$ as follows: Consider an equilateral triangle on the complex plane, centered at the origin, with vertices labeled \emph{clockwise} by $1$, $2$ and $3$ such that the vertex $1$ is at the point $v$. Then $\mu_{ij}(v)$ is $\sqrt{2}$ times the difference between the real coordinates of the vertices $i$ and $j$. With this geometric interpretation, it is elementary to check the following properties:
\begin{enumerate}[(a)]
\item\label{item_mu1}
$\mu_{ij}(\omega v)=\mu_{i+1,j+1}(v)$ (here $i,j\in\{1,2,3\}$ are counted mod $3$). 
\item\label{item_mu2}
The maximum of $\mu(v)$ is $\sqrt{6}$ and is achieved when $v^3=\pm\ima$.
\item\label{item_mu3}
The minimum of $\mu(v)$ is $\frac{3\sqrt{2}}{2}$ and is achieved when $v^3=\pm1$.
\item\label{item_mu4}
If $v^3\neq\pm 1$, there is a unique $(i,j)$ with $\mu_{ij}(v)=\mu(v)$. In particular, for $v=e^{\frac{\pi\ima}{6}}$ (resp. $v=e^{-\frac{\pi\ima}{6}}$), the corresponding $(i,j)$ is $(1,3)$ (resp. $(1,2)$) 
\item\label{item_mu5}
If $v^3=\pm1$, there are two $(i,j)$'s with $\mu_{ij}(v)=\mu(v)$. For $v=1$ (resp. $v=e^{\frac{\pi\ima}{3}}$), the corresponding $(i,j)$'s are $(1,2)$ and $(1,3)$ (resp. $(1,3)$ and $(2,3)$).
\end{enumerate}

Let $\mathfrak{n}_v\subset\sl_3\mathbb{R}$ denote the eigenspace of $\Ad_{\para_0(0,v)}$ with largest eigenvalue, which is generated by those $E_{ij}$ with $\mu_{ij}(v)=\mu(v)$. By Property (\ref{item_mu1}), $\frak{n}_{\omega v}$ is obtained from $\mathfrak{n}_v$ by a permutation of basis, and it follows from (\ref{item_mu4}) and (\ref{item_mu5}) that $\mathfrak{n}_v$ has dimension $1$ or $2$. Moreover, when $\dim\mathfrak{n}_v=2$ (\ie when $v^3=\pm1$), the matrix product $xy$ vanishes for all $x,y\in\mathfrak{n}_v$. Therefore, $\frak{n}_v$ is an abelian Lie algebra and the corresponding closed subgroup of $\SL(3,\mathbb{R})$ can be written as
$$
\mathcal{N}_v:=\exp(\frak{n}_v)=\{\exp(x)=I+x\mid x\in\frak{n}_v\}.
$$ 

\subsection{Asymptotics of osculation map}\label{subsec_part1}
With the above notations, the asymptotic result on osculation maps needed for Proposition \ref{prop_reduction1} is:
\begin{proposition}\label{prop_osculation1}
Let $\overline{V}$ be a closed sector in $\mathbb{C}$ based at $0$ and $U$ be the $\varepsilon$-neighborhood of $\overline{V}$ for some $\varepsilon>0$. Let $g=e^u|\dz|^2$ be a solution to Wang's equation on $(U, \dz^3)$ such that
$$
0\leq u(z)\leq C|z|^\alpha e^{-\sqrt{6}|z|}\ \text{ for all } z\in U,\,|z|\geq r
$$
for some constants $\alpha, C,r>0$. Let $P:U\rightarrow\SL(\T_0\mathbb{C}\oplus\mathbb{R})\cong\SL(3,\mathbb{R})$ be the osculation map associated to $g$ at $0$. 
\begin{enumerate}
\item\label{item_osculation1}
If $\overline{V}$ does not contain any cubic root of $\pm\ima$, then $P(z)$ converges in $\SL(3,\mathbb{R})$ as $z$ goes to infinity in $\overline{V}$.
\item\label{item_osculation2}
If there is a unique $v\in\overline{V}$ such that $v^3=\pm\ima$, then there are continuous maps $\widehat{P}: \overline{V}\rightarrow \SL(3,\mathbb{R})$ and $R: \overline{V}\rightarrow \mathcal{N}_v$ such that
\begin{itemize}
\item $P(z)=\widehat{P}(z)R(z)$ for all $z\in \overline{V}$;
\item $\widehat{P}(z)$ converges in $\SL(3,\mathbb{R})$ as $z$ goes to infinity in $\overline{V}$;
\item $\|R(z)-I\|\leq C'|z|^{\alpha+1}$ for some constant $C'>0$ and all $z\in \overline{V}$, $|z|\geq r$, where $\|\cdot\|$ is a matrix norm.
\end{itemize}
\end{enumerate}
\end{proposition}
We give below the proof of Part (\ref{item_osculation2}) and then only outline that of Part (\ref{item_osculation1}), as the latter follows the same scheme and is simpler. 

To be specific, we fix the norm $\|\cdot\|$ on $3\times 3$ matrices as $\|M\|:=\max_{i,j}|M_{ij}|$. Note that $\|AB\|\leq 3\|A\|\cdot \|B\|$. Given a scalar valued $1$-form $\xi$ on $U$, let $|\xi|$ denote the pointwise norm of $\xi$ with respect to the metric $|\dz|^2$, while for a matrix valued $1$-form $\Xi$ we consider the pointwise norm $\|\Xi\|:=\max_{ij}|\Xi_{ij}|$.

\begin{proof}[Proof of Part (\ref{item_osculation2})]
As seen in Section \ref{subsec_setups}, the affine sphere connections $\D_0$ and $\D$ associated to $(|\dz|^2,\dz^3)$ and $(g,\dz^3)$ are expressed under the frame $(e_1,e_2,e_3)$ as
$$
\D_0=\dif+\sqrt{2}\re
\begin{pmatrix}
\dz&&\\
&\omega^2\dz&\\
&&\omega\,\dz
\end{pmatrix}
,\quad
\D=\D_0+\Xi,
$$ 
where $\Xi$ is a matrix valued $1$-form whose coefficients are linear combinations of $e^u-1$, $e^{-u}-1$, $\pa_xu$ and $\pa_yu$. Since $|e^u-1|$, $|e^{-u}-1|$ and $|\Delta u|=2|e^u-e^{-2u}|$ are all majorized by a constant multiple of $u$ when $u$ is bounded, using the Schauder estimates (see \eg \cite[Corollary 1.2.7]{jost}) to control $|\pa_x u|$ and $|\pa_yu|$, we obtain
\begin{equation}\label{eqn_xiestimate}
\|\Xi\|(z)\leq C_1|z|^{\alpha}e^{-\sqrt{6}|z|}\ \text{ for all } z\in\overline{V},\, |z|\geq r
\end{equation}
for some $C_1>0$ (the control of the derivatives at each point replies on the control of $u$ on a neighborhood of that point, this is why in the hypothesis we need $u$ to be controlled on a $\varepsilon$-neighborhood of $\overline{V}$). We proceed in the following steps.

\textbf{Step 1. Define $\widehat{P}(z)$ and $R(z)$.} 
Suppose $v\in\overline{V}$, $v^3=\pm\ima$. By Properties (\ref{item_mu2}) and (\ref{item_mu4}) in Section \ref{subsec_setups}, the largest eigenvalue of $\Ad_{\para_0(0,v)}$ on $\sl_3\mathbb{R}$ is $\sqrt{6}$, which has $1$-dimensional eigenspace $\mathfrak{n}_v$ generated by an elementary matrix $E_{i_0j_0}$, $i_0\neq j_0$. Let $\Xi_{\mathfrak{n}_v}:=\Xi_{i_0j_0}E_{i_0j_0}$ be the $\frak{n}_v$-component of $\Xi$ (where $\Xi_{ij}$ denotes the $(i,j)$-entry of $\Xi$) and consider the connection
$$
\D_1:=\D_0+
\Xi_{\mathfrak{n}_v}.
$$
As $\D_1$ is not flat, its parallel transport $\para_1$ depends on paths instead of just two points. For every $z\in\overline{V}$, consider the path $\Gamma_z$ in $\overline{V}$ from $0$ to $z$ which first goes straight from $0$ to $|z|v$ then goes from $|z|v$ to $z$ along the circle of radius $|z|$ centered at $0$. We define the required maps $\widehat{P}$ and $R$ as 
$$
\widehat{P}(z):=\para(0,z)\para_1(\Gamma_z), \quad R(z):=\para_1(\Gamma_z^{-1})\para_0(z,0).
$$
The definition immediately implies $P(z):=\para(0,z)\para_0(z,0)=\widehat{P}(z)R(z)$. 

\vspace{3pt}

\textbf{Step 2. Estimate of $R(z)$.} 
In order to control $R(z)$, we first consider a general $C^1$-path $\gamma:[0,T]\rightarrow U$ with $\gamma(0)=0$ instead of $\Gamma_z$. Put $U(t):=\para_1(\gamma_{[0,t]}^{-1})\para_0(\gamma(t),0)$ for $t\in[0,T]$, where $\gamma_{[0,t]}^{-1}$ is the reverse path of  $\gamma|_{[0,t]}$. Computing the derivative of $U(t)$ in view of  Appendix \ref{subsec_connection}, we get
$$
U'(t)=\frac{\dif}{\dif t}\left(\para_1(\gamma_{[0,t]}^{-1})\para_0(\gamma(t),0)\right)
$$
\begin{align*}
&=\para_1(\gamma_{[0,t]}^{-1})A_1(\dot\gamma(t))\para_0(\gamma(t),0)-\para_1(\gamma_{[0,t]}^{-1})A_0(\dot\gamma(t))\para_0(\gamma(t),0)\\
&=\para_1(\gamma_{[0,t]}^{-1})\para_0(\gamma(t),0)\Ad_{\para_0(0,\gamma(t))}(\D_1-\D_0)(\dot\gamma(t))=U(t)\lambda_{i_0j_0}(\gamma(t))\Xi_{\mathfrak{n}_v}(\dot\gamma(t)),
\end{align*}
where $A_0$ and $A_1$ are the matrix valued $1$-forms such that $\D_0=\dif+A_0$, $\D_1=\dif+A_1$, and $\lambda_{ij}(z)=e^{\mu_{ij}(z/|z|)|z|}$ is the eigenvalue of $\Ad_{\para_0(0,z)}$ on $E_{ij}$ (see Section \ref{subsec_setups}).
Write $X(t):=\lambda_{i_0j_0}(\gamma(t))\Xi_{\mathfrak{n}_v}(\dot\gamma(t))\in\mathfrak{n}_v$, so that $U(t)$ is the solution to the ODE $U'(t)=U(t)X(t)$ with initial condition $U(0)=I$. By virtue of the fact that the matrix product $xy$ vanishes for any $x,y\in\mathfrak{n}_v$, the solution is given by
$$
U(t)=\exp\left(\int_0^tX(\tau)\dif \tau\right)=I+\int_0^tX(\tau)\dif\tau=I+\int_{\gamma|_{[0,t]}} \lambda_{i_0j_0}\Xi_{\mathfrak{n}_v}.
$$
In particular, we have $U(T)=\int_\gamma\lambda_{i_0j_0}\Xi_{\mathfrak{n}_v}$. This generalizes to piecewise $C^1$-paths through approximation by $C^1$-paths. Applying to the path $\Gamma_z$, we get
$$
R(z)=I+\int_{\Gamma_z}\lambda_{i_0j_0}\Xi_{\mathfrak{n}_v}\in \mathcal{N}_v.
$$
Inequality (\ref{eqn_xiestimate}) and the fact that $\lambda_{ij}(z)=e^{\mu_{ij}(z/|z|)|z|}\leq e^{\sqrt{6}|z|}$ (see Section \ref{subsec_setups})
imply $\|\lambda_{i_0j_0}\Xi_{\mathfrak{n}_v}\|(z)\leq C_1|z|^{\alpha}$ for all $z\in\overline{V}$, $|z|\geq r$. By integration, we get 
\begin{equation}\label{eqn_restimate1}
\|R(z)-I\|\leq C'|z|^{\alpha+1}\ \text{ for all } z\in\overline{V},\,|z|\geq r
\end{equation}
for  a constant $C'>0$ depending on $C_1$, $\alpha$, $\overline{V}$ and the maximum of $\|\lambda_{i_0j_0}\Xi_{\mathfrak{n}_v}\|$ on $\{z\in\overline{V}\mid |z|\leq r\}$. Thus, $R(z)$ has the required properties.

\vspace{3pt}

\textbf{Step 3. Convergence of $\widehat{P}(tv)$.} 
Noting that $\Gamma_{tv}$ is the line segment from $0$ to $tv$, we compute the derivative of $\widehat{P}(tv)=\para(0,tv)\para_1(\Gamma_{tv})$ in $t$ similarly as in the above computation of $U'(t)$ and get
$$
\frac{\dif }{\dif t}\widehat{P}(tv)=\widehat{P}(tv)\Ad_{\para_1(\Gamma_{tv}^{-1})}(\D-\D_1)\big(\pa_t(tv)\big)=\widehat{P}(tv)\Ad_{\para_1(\Gamma_{tv}^{-1})}(\Xi-\Xi_{\mathfrak{n}_v})_{tv}(v).
$$
In order to analyze this ODE, we define an $\sl_3\mathbb{R}$-valued $1$-form $\Theta$ on $U$ by specifying $\Theta_z\in\Hom(\T_zU,\sl_3\mathbb{R})$ as
$$
\Theta_z:=\Ad_{\para_1(\Gamma_{z}^{-1})}(\Xi-\Xi_{\mathfrak{n}_v})_z=\Ad_{R(z)}\Ad_{\para_0(0,z)}(\Xi-\Xi_{\mathfrak{n}_v})_z.
$$
Then $\widehat{P}(tv)$ is the solution to the initial value problem
\begin{equation}\label{eqn_hatpt}
\frac{\dif}{\dif t}\widehat{P}(tv)=\widehat{P}(tv)\Theta_{tv}(v),\quad \widehat{P}(0)=I.
\end{equation}

We shall show that $\|\Theta\|(z)$ decays exponentially with respect to $|z|$. To this end, we first estimate $\Ad_{\para_0(0,z)}(\Xi-\Xi_{\mathfrak{n}_v})_z$. Recall from Section \ref{subsec_setups} that the eigenvalue of $\Ad_{\para_0(0,z)}$ on $E_{ij}$ is $\lambda_{ij}(z)=e^{\mu_{ij}(z/|z|)|z|}$. By assumption, for $z\in\overline{V}$ we have $(z/|z|)^3=\pm\ima$ if and only if $z/|z|=v$, so
Properties (\ref{item_mu2}) and (\ref{item_mu4}) of $\mu_{ij}$ in Section \ref{subsec_setups} imply that $\mu_{ij}(z/|z|)\leq\sqrt{6}$ for all $z\in\overline{V}$ and the equality is achieved if and only if $z/|z|=v$, $(i,j)=(i_0,j_0)$. Therefore, there is a constant $c>0$ such that
$$
\max_{z\in\overline{V},(i,j)\neq(i_0,j_0)}\mu_{ij}(z/|z|)=\sqrt{6}-c.
$$
whence $\lambda_{ij}(z)\leq e^{(\sqrt{6}-c)|z|}$ for all $z\in\overline{V}$ and $(i,j)\neq (i_0,j_0)$. Combining with (\ref{eqn_xiestimate}) and noting that $\Xi-\Xi_{\mathfrak{n}_v}$ is obtained from $\Xi$ by removing the $(i_0,j_0)$-entry, we obtain
\begin{equation}\label{eqn_adt0estimate}
\|\Ad_{\para_0(0,z)}\big(\Xi-\Xi_{\mathfrak{n}_v}\big)_z\|
\leq C_1|z|^\alpha e^{-c|z|}\ \text{ for all } z\in\overline{V},\,|z|\geq r.
\end{equation}

On the other hand, given $M=(M_{ij})\in\sl_3\mathbb{R}$ with $M_{i_0j_0}=0$ and $R=I+rE_{i_0j_0}\in \mathcal{N}_v$ ($r\in\mathbb{R}$), writing down $\Ad_{R}M$ explicitly, one sees that each of its entries is a sum of at most three terms of the form $r^kM_{ij}$ with $k=0,1$ or $2$. Therefore,
$$
\|\Ad_RM\|\leq 3\big(\max_{k=0,1,2}|r|^k\big)\|M\|=3\max\{1,\|R-I\|^2\}\|M\|.
$$
Using this inequality, we derive from (\ref{eqn_restimate1}) and (\ref{eqn_adt0estimate}) that
\begin{equation}\label{eqn_theta}
\|\Theta\|(z)\leq
3C_1\max\{1,C'{}^2|z|^{2\alpha+2}\}|z|^\alpha e^{-c|z|}\ \text{ for all } z\in\overline{V},\,|z|\geq r.
\end{equation}
In particular, $\int_0^{+\infty}\|\Theta_{t v}(v)\|\dif t<+\infty$. This allows us to apply a classical ODE result \cite[Lemma B.1(ii)]{dumas-wolf} to Eq.(\ref{eqn_hatpt}) and conclude that
$\widehat{P}(tv)$ has a limit $\widehat{P}(\infty)\in\SL(3,\mathbb{R})$ as $t\rightarrow+\infty$.

\textbf{Step 4. Compare $\widehat{P}(z)$ with $\widehat{P}(tv)$.} 
Every $z\in\overline{V}$ can be written as $z=te^{\theta\ima}v$, where $t\geq 0$ and $\theta$ belongs to an interval $[\theta_1,\theta_2]$ with $\theta_1<0<\theta_2$. We finally need to show that $\widehat{P}(te^{\theta\ima}v)$ converges to $\widehat{P}(\infty)$ uniformly in $\theta\in[\theta_1,\theta_2]$ as $t\rightarrow+\infty$. To this end, consider the function $Q:[\theta_1,\theta_2]\times[0,+\infty)\rightarrow\SL(3,\mathbb{R})$ defined by
\begin{align*}
Q(\theta,t)&:=\widehat{P}(tv)^{-1}\widehat{P}(te^{\theta\ima}v)=\para_1(\Gamma_{tv}^{-1})\para(tv,0)\para(0,te^{\theta\ima}v)\para_1(\Gamma_{te^{\theta\ima}v})\\[5pt]
&=\para_1(\Gamma_{tv}^{-1})\para(tv,te^{\theta\ima}v)\para_1(\Gamma_{te^{\theta\ima}v}).
\end{align*}

Using the identities (see Appendix \ref{subsec_connection})
$$
\frac{\pa}{\pa\theta}\para(tv,te^{\theta\ima}v)=\para(tv,te^{\theta\ima}v)A\left(\pa_\theta(te^{\theta\ima}v)\right),
$$
$$
\frac{\pa}{\pa\theta}\para_1(\Gamma_{te^{\theta\ima}v})=-A_1\left(\pa_\theta(te^{\theta\ima}v)\right)\para_1(\Gamma_{te^{\theta\ima}v}),
$$
we compute $\tfrac{\pa}{\pa\theta}Q(\theta,t)$ similarly as above and get 
$$
\frac{\pa}{\pa\theta}Q(\theta,t)=
Q(\theta,t)\Ad_{\para_1(\Gamma_{te^{\theta\ima}v}^{-1})}\big(\D-\D_1\big)\left(\pa_\theta(te^{\theta\ima}v)\right)=Q(\theta,t)\Theta(\pa_\theta(te^{\theta\ima}v)).
$$
View this equation as an ODE with parameter $t$ satisfied by every $Q(\cdot,t)$, which has initial value $Q(0,t)=I$. Inequality (\ref{eqn_theta}) implies that $\|\Theta(\pa_\theta(te^{\theta\ima}v))\|\rightarrow0$ uniformly in $\theta$ as $t\rightarrow+\infty$, so we can apply another ODE result \cite[Lemma B.1(i)]{dumas-wolf} and conclude that $\|Q(\theta,t)-I\|$ converges to $0$ uniformly in $\theta$, hence so does
$$
\|\widehat{P}(te^{\theta\ima}v)-\widehat{P}(tv)\|
=\|\widehat{P}(tv)(Q(\theta,t)-I)\|
\leq 3\|\widehat{P}(tv)\|\cdot\|Q(\theta,t)-I\|.
$$
It follows that $\widehat{P}(te^{\theta\ima}v)$ converges to $\widehat{P}(\infty)$ uniformly in $\theta$, completing the proof of Proposition \ref{prop_osculation1}  (\ref{item_osculation2}).
\end{proof}

Part (\ref{item_osculation1}) is proved by the same argument with simplifications. Namely, since there is no $v\in\overline{V}$ with $v^3=\pm1$, the spectral radius $e^{\mu(z/|z|)|z|}$ of $\Ad_{\para_0(0,z)}$ on the whole $\sl_3\mathbb{R}$ is less than $e^{(\sqrt{6}-c)|z|}$ for a constant $c>0$, so we have a more straightforward exponential decay property
 $$
\left\|\Ad_{\para_0(0,z)}\Xi\right\|\leq C_1|z|^\alpha e^{-c|z|}\ \text{ for all } z\in\overline{V},\,|z|\geq r.
 $$
This allows us to proceed with $P(z)$ exactly as how we treat $\widehat{P}(z)$ in the above proof and conclude that $P(te^{\theta\ima}v)$ converges to some $P(\infty)\in\SL(3,\mathbb{R})$ uniformly in $\theta$.

\vspace{5pt}
Combining the two parts of Proposition \ref{prop_osculation1} yields:
\begin{corollary}\label{coro_comparison}
Under the hypotheses of Proposition \ref{prop_osculation1}, for each connected component $W$ of $\overline{V}\setminus\{tv\mid t\geq 0,\,v^3=\pm\ima\}$, there exists $P_W(\infty)\in\SL(3,\mathbb{R})$ such that $P(z)$ converges to $P_W(\infty)$ when $z$ goes to infinity in any closed sector contained in $W\cup\{0\}$. If $W_1$ and $W_2$ are adjacent connected components separated by a ray $v\mathbb{R}_{\geq0}$, then $P_{W_1}(\infty)$ and $P_{W_2}(\infty)$ belong to the same left $\mathcal{N}_v$-coset in $\SL(3,\mathbb{R})$.
\end{corollary}
\begin{proof}
The first statement follows from Proposition \ref{prop_osculation1} (\ref{item_osculation1}) applied to closed sectors contained in $W\cup\{0\}$. As for the second one, let $\overline{S}\subset W_1\cup (v\mathbb{R}_{\geq0})\cup W_2$ be a closed sector containing $v$ in its interior, so that applying Proposition \ref{prop_osculation1} (\ref{item_osculation2}) to $\overline{S}$ gives $P(z)=\widehat{P}(z)R(z)$ on $\overline{S}$ with $\widehat{P}(z)$ converging to some $\widehat{P}(\infty)\in\SL(3,\mathbb{R})$ as $z$ goes to infinity in $\overline{S}$. Let $\Pi:\SL(3,\mathbb{R})\rightarrow\SL(3,\mathbb{R})/\mathcal{N}_v$ denote the projection, assigning to each element of $\SL(3,\mathbb{R})$ its left $\mathcal{N}_v$-coset. Then $\Pi(P(z))=\Pi(\widehat{P}(z))$ because $R(z)\in \mathcal{N}_v$. Given $v_i\in W_i$, $i=1,2$, we have $\lim_{t\rightarrow+\infty}P(tv_i)=P_{W_i}(\infty)$, hence
$$
\widehat{P}(\infty)=\lim_{t\rightarrow+\infty}\Pi(\widehat{P}(tv_i))=\lim_{t\rightarrow+\infty}\Pi(P(tv_i))=\Pi(P_{W_i}(\infty))
$$
for $i=1,2$. It follows that $\Pi(P_{W_1}(\infty))=\Pi(P_{W_2}(\infty))$, as required.
\end{proof}

\subsection{Proof of Proposition \ref{prop_reduction1}}\label{subsec_reduction1}
By a shift of coordinate, we can consider $g=e^u|\dz|^2$ and $\iota$ as defined on $\H_{-1}=\{z\in\mathbb{C}\mid \re(z)>-1\}$ instead of $\H$. The control of $u$ in the assumption is still valid for $u:\H_{-1}\rightarrow\mathbb{R}$ after modifying the constants $C$ and $r$. As at the end of Section \ref{subsec_reduction}, let $\para_0$ and $\para$ be the parallel transports of the affine sphere connections on $\T\H_{-1}\oplus\underline{\mathbb{R}}$ associated to $(|\dz|^2,\dz^3)$ and $(g,\dz^3)$, respectively, so that we can write
$$
\iota_0, \iota:\H_{-1}\rightarrow\T_0\mathbb{C}\oplus\mathbb{R}\cong\mathbb{R}^3,\ \ \iota_0(z)=\para_0(0,z)\underline{1}_z,\,\iota(z)=\para(0,z)\underline{1}_z=P(z)\iota_0(z),
$$
and hence $\delta(z)=P(z)\iota_0(z)$. Here 
$P(z):=\para(0,z)\para_0(z,0)$ is the osculation map. 

The idea of proof of the proposition is to define the values of the extension $\overline{\delta}$ on $\pa_\infty\H$ using the extension $\overline{\delta}_0$ of $\delta_0$ and the limits of $P(z)$ or $\widehat{P}(z)$ from Proposition \ref{prop_osculation1}, so that $\overline{\delta}$ has the built-in property 
\begin{equation}\label{eqn_limzn}\tag{$\star$}
\text{$\forall$ sequence $(z_n)$ in $\H_{-1}$,}\lim_{n\rightarrow\infty}z_n=z_\infty\in\pa_\infty\H\, \Longrightarrow \lim_{n\rightarrow\infty}\delta(z_n)=\overline{\delta}(z_\infty),
\end{equation}
as well as the metric preserving property.  Since $\H_{-1}\cup\pa_\infty\H$ is homeomorphic to the closed upper half-plane $\overline{\mathbb{H}}$ with $\pa_\infty\H$ corresponding to the boundary $\pa\mathbb{H}=\mathbb{R}$, we can use the following lemma to conclude that $\overline{\delta}$ is continuous and injective:
\begin{lemma}\label{lemma_topology}
Let $f:\overline{\mathbb{H}}\rightarrow\mathbb{RP}^2$ be a map such that the restrictions $f|_{\mathbb{H}}$ and $f|_\mathbb{R}$ are both continuous and injective, and that  for every sequence $(z_n)_{n=1,2,\cdots}$ in $\mathbb{H}$ converging to some $x\in\mathbb{R}$, we have $\lim_{n\rightarrow\infty}f(z_n)=f(x)$. Then $f$ is continuous and injective on $\overline{\mathbb{H}}$. 
\end{lemma}
\begin{proof}
By the last hypothesis and the continuity of $f|_\mathbb{R}$, for any sequence $(z_n)$ in $\overline{\mathbb{H}}$ converging to $x\in\mathbb{R}$ we still have $\lim_{n\rightarrow\infty}f(z_n)=f(x)$. So $f$ is continuous on the whole $\overline{\mathbb{H}}$. Given $z\in \mathbb{H}$ and $x\in\mathbb{R}$, we pick a neighborhood $U$ of $z$ in $\mathbb{H}$ and a sequence $(z_n)$ of points in $\mathbb{H}\setminus U$ such that $f|_U$ is injective and $\lim_{n\rightarrow\infty}z_n=x$. By injectivity of $f|_\mathbb{H}$, each $f(z_n)$ lies outside of $f(U)$, which is an open subset of $\mathbb{RP}^2$ by Brouwer Invariance of Domain. As a result, $\lim_{n\rightarrow\infty}f(z_n)=f(x)$ lies outside $f(U)$, hence is different from $f(z)$. This shows that $f$ is injective. 
\end{proof}

\begin{proof}[Proof of Proposition \ref{prop_reduction1}]
By Property (\ref{item_mu4}) in Section \ref{subsec_setups}, the unipotent subgroups $\mathcal{N}_v$ for $v=e^{\pm\pi\ima/6}$ are
$$
\mathcal{N}_+:=\mathcal{N}_{e^{\pi\ima/6}}=
\{I+rE_{13}\mid r\in\mathbb{R}\}
,\quad 
\mathcal{N}_-:=\mathcal{N}_{e^{-\pi\ima/6}}=
\{I+rE_{12}\mid r\in\mathbb{R}\}.
$$
Applying Corollary \ref{coro_comparison} to $\overline{V}=\overline{\H}\subset\H_{-1}$, we obtain $P_+, P_-, P_0\in\SL(3,\mathbb{R})$, with $P_\pm$ in the same left $\mathcal{N}_\pm$-coset as $P_0$, such that
$$
\lim_{t\rightarrow+\infty}P(te^{\theta\ima})=
\begin{cases}
P_-&\mbox{ if }-\frac{\pi}{2}<\theta<-\frac{\pi}{6},\\
P_0&\mbox{ if }-\frac{\pi}{6}<\theta<\frac{\pi}{6},\\
P_+&\mbox{ if }\frac{\pi}{6}<\theta<-\frac{\pi}{2}.
\end{cases}
$$
Since the actions of both $\mathcal{N}_+$ and $\mathcal{N}_-$ on $\mathbb{RP}^2$ fix $[1:0:0]$, we have 
$$
P_+[1:0:0]=P_0[1:0:0]=P_-[1:0:0].
$$ 
We then define the required extension $\overline{\delta}:\H_{-1}\cup\pa_\infty\H\rightarrow\mathbb{RP}^2$ by setting its values on $\pa_\infty\H=\pai{e^{\pi\ima/3}}\mathbb{C}\cup \pai{1}\mathbb{C}\cup\pai{e^{-\pi\ima/3}}\mathbb{C}$ as
$$
\overline{\delta}(\zeta):=
\begin{cases}
P_+\overline{\delta}_0(\zeta) &\mbox{ if }\zeta\in \pai{e^{\pi\ima/3}}\mathbb{C},\\[3pt]
P_+[1:0:0]=P_-[1:0:0]&\mbox{ if }\zeta=\pai{1}\mathbb{C},\\[3pt]
P_-\overline{\delta}_0(\zeta)&\mbox{ if }\zeta\in \pai{e^{-\pi\ima/3}}\mathbb{C}.
\end{cases}
$$
Here $\overline{\delta}_0:\mathbb{C}'\to\mathbb{RP}^2$ is the extension of $\delta_0=\mathbb{P}\circ\iota_0$ (see Section \ref{subsec_bord1}), sending $\pai{1}\mathbb{C}$ to $[1:0:0]$ and sending $\pai{e^{\pi\ima/3}}\mathbb{C}$ and $\pai{e^{-\pi\ima/3}}\mathbb{C}$ isometrically to the open segments 
$$
I_+:=\{[1:a:0]\in\mathbb{RP}^2\mid a>0\},\quad I_-:=\{[a:0:1]\in\mathbb{RP}^2\mid a>0\}.
$$

Since $P_\pm$ is in the same left $\mathcal{N}_\pm$-coset as $P_0$, we can write $P_-=P_+N_+N_-$ for $N_\pm\in \mathcal{N}_\pm$, and hence rewrite the map $\overline{\delta}|_{\pa_\infty\H}$ defined above as a composition
$$
\overline{\delta}|_{\pa_\infty\H}=P_+\circ N'\circ\overline{\delta}_0|_{\pa_\infty\H},
$$
where $N'$ is the continuous map $I_+\cup [1:0:0]\cup I_-\rightarrow\mathbb{RP}^2$ restricting to the identity on $I_+\cup [1:0:0]$ and to the projective transformation $N_+N_-$ on $I_-$. The segments $N_+N_-(I_-)$ and $I_+$ share the endpoint $[1:0:0]$ and are not collinear. Therefore, we conclude that $\overline{\delta}|_{\pa_\infty\H}$ is continuous and sends $\pai{e^{\pi\ima/3}}\mathbb{C}$ and $\pai{e^{-\pi\ima/3}}\mathbb{C}$ isometrically to open line segments not collinear with each other.

We now only need to verify Property (\ref{eqn_limzn}). First assume $z_\infty\in\pai{e^{\pi\ima/3}}\mathbb{C}$, so that $z_\infty=[\gamma]$ for a negative ray $\gamma:[0,+\infty)\rightarrow\H_{-1}$ of the form $\gamma(t)=\gamma(0)+e^{\pi\ima/3}t$ (see Section \ref{subsec_bord1}). 
The convergence $z_n\rightarrow[\gamma]$ means that $|z_n|\rightarrow+\infty$ as $n\rightarrow\infty$ while the distance from $z_n$ to $\gamma([0,+\infty))$ tends to $0$. As a consequence, for any $\varepsilon>0$, $z_n$ is contained in the sector 
$\{z\in\mathbb{C}\mid |\arg(z)-\frac{\pi}{3}|\leq\varepsilon\}$
for $n$ big enough. By Proposition \ref{prop_osculation1} (\ref{item_osculation1}) applied to this sector,  $P(z_n)$ converges to the above defined $P_+$. On the other hand, we have $\delta_0(z_n)\rightarrow\overline{\delta}_0(z_\infty)$ by continuity of $\overline{\delta}_0$ (Lemma \ref{lemma_deltaprime}), hence 
$\delta(z_n)=P(z_n)\delta_0(z_n)\rightarrow P_+\overline{\delta}_0(z_\infty)=\overline{\delta}(z_\infty)$, as required. 

If $z_\infty\in\pai{e^{-\pi\ima/3}}\mathbb{C}$, Property (\ref{eqn_limzn}) is proved in the same way. Finally, if $z_\infty$ is the point $\pai{1}\mathbb{C}$, we consider the sectors
$$
V_+:=\{0<\arg(z)<\tfrac{\pi}{3}\}, \ V_-:=\{-\tfrac{\pi}{3}<\arg(z)<0\},\ V:=\{-\tfrac{\pi}{3}<\arg(z)<\tfrac{\pi}{3}\}
$$ 
Neighborhoods of $\pai{1}\mathbb{C}$ in $\mathbb{C}'$ are generated by all subsets of the form $(x+V)\cup\pa_\infty(x+V)$, $x>0$, where $\pa_\infty(x+V)\subset\pa_\infty\mathbb{C}$.
 In order to show that the image of the sequence $(z_n)$ by $\delta$ converges to $\overline{\delta}(z_\infty)=P_+[1:0:0]=P_-[1:0:0]$, it is sufficient to show it for the subsequences $(z_n)\cap \{z\mid \im(z)\geq 0\}$ and $(z_n)\cap \{z\mid \im(z)\leq 0\}$ separately. Thus, we can assume either $\im(z_n)\geq 0$ for all $n$ or $\im(z_n)\leq 0$ for all $n$.

In the case $\im(z_n)\geq0$, the convergence $z_n\rightarrow z_\infty$ means that for any $x\geq0$, $z_n$ is contained in $x+\overline{V}_+$ for $n$ big enough. 
By part (\ref{item_osculation2}) of Proposition \ref{prop_osculation1}, there are continuous maps $\widehat{P}_+:\overline{V}_+\rightarrow\SL(3,\mathbb{R})$ and $R_+:\overline{V}_+\rightarrow \mathcal{N}_+$ such that 
\begin{enumerate}[(i)]
\item\label{item_proofred1}
$P(z)=\widehat{P}_+(z)R_+(z)$ for all $z\in V_+$;
\item\label{item_proofred2}
$\widehat{P}_+(z)$ converges to some $\widehat{P}_+(\infty)\in\SL(3,\mathbb{R})$ as $z\rightarrow\infty$ in $V_+$;
\item\label{item_proofred3}
$\|R_+(z)-I\|\leq C'|z|^{\alpha+1}$ for some $C'>0$ and all $z\in V_+$, $|z|\geq r$. 
\end{enumerate}
Letting $z$ go to infinity along a ray $e^{\theta\ima}\mathbb{R}_{\geq 0}$ with $\theta\in (0,\frac{\pi}{6})$ (or $\theta\in(\frac{\pi}{6},\frac{\pi}{3})$) in (\ref{item_proofred1}), we see that the limit $\widehat{P}_+(\infty)$ in (\ref{item_proofred2}) is in the same left $\mathcal{N}_+$-coset as $P_0$ (or $P_+$). Since $\mathcal{N}_+$ fixes $[1:0:0]$, we have 
$$
\widehat{P}_+[1:0:0]=P_0[1:0:0]=P_+[1:0:0]=\overline{\delta}(z_\infty).
$$

We claim that for any neighborhood $B$ of $[1:0:0]$ in $\mathbb{RP}^2$, there exists a big enough $x_0>0$ such that $x_0+\overline{V}_+$ is sent into $B$ by the map $z\mapsto R_+(z)\delta_0(z)$ from $\overline{V}_+$ to $\mathbb{RP}^2$. The proof is based on the coordinate expression
$$
R_+(z)\delta_0(z)=
\begin{pmatrix}
1&&r(z)\\
&1&\\
&&1
\end{pmatrix}
\begin{bmatrix}
e^{\sqrt{2}\re(z)}\\
e^{\sqrt{2}\re(\omega^2z)}\\
e^{\sqrt{2}\re(\omega z)}\\
\end{bmatrix}
=
\begin{bmatrix}
e^{\sqrt{2}\re[(1-\omega)z]}+r(z)\\
e^{\sqrt{2}\re[(\omega^2-\omega)z]}\\
1\\
\end{bmatrix},
$$
where $|r(z)|=\|R_+(z)-I\|\leq C'|z|^{\alpha+1}$ for $|z|\geq r$ by (\ref{item_proofred3}). For any $z\in\overline{V}_+$ we have
$$
\sqrt{2}\re[(1-\omega)z]=\sqrt{6}\re(e^{-\pi\ima/6}z)\geq \sqrt{6}\cdot\tfrac{\sqrt{3}}{2}|z|.
$$ 
Hence $e^{\sqrt{2}\re[(1-\omega)z]}$ glows exponentially with respect to $|z|$.
As a result, when $|z|$ is big enough, we have $|r(z)|\leq \frac{1}{2}e^{\sqrt{2}\re[(1-\omega)z]}$, so that 
$$
e^{\sqrt{2}\re[(1-\omega)z]}+r(z)\geq \tfrac{1}{2}e^{\sqrt{2}\re[(1-\omega)z]}.
$$
Since $e^{\sqrt{2}\re[(\omega^2-\omega)z]}=e^{\sqrt{6}\im(z)}\geq 1$, the claim follows from the above inequality and the following one, which holds for all $z\in x_0+\overline{V}_+$:
$$
e^{\sqrt{2}\re[(1-\omega)z]}/e^{\sqrt{2}\re[(\omega^2-\omega)z]}=e^{\sqrt{2}\re[(1-\omega^2)z]}=e^{\sqrt{6}\re(e^{\pi\ima/6}z)}\geq e^{\sqrt{6}\cdot\frac{\sqrt{3}}{2}x_0}.
$$

The claim implies that the image of the sequence $(z_n)$ by the map $z\mapsto R_+(z)\delta_0(z)$ converges to $[1:0:0]$. 
As a result, $\delta(z_n)=P(z_n)\delta_0(z_n)=\widehat{P}_+(z_n)R_+(z_n)\delta_0(z_n)$ converges to $\widehat{P}_+(\infty)[1:0:0]=\overline{\delta}(z_\infty)$, proving Property (\ref{eqn_limzn}) in the case $z_\infty=\pai{1}\mathbb{C}$ and $\im(z_n)\geq0$. The $\im(z_n)\leq 0$ case is proved in the same way. This establishes Property (\ref{eqn_limzn}) and completes the proof of Proposition \ref{prop_reduction1}. 
\end{proof}

\subsection{Asymptotics of osculation map with weaker exponential decay}\label{subsec_part2}
In order to prove Proposition \ref{prop_reduction2}, we need the following variation of Proposition \ref{prop_osculation1} under a weaker exponential decay assumption on $u$:
\begin{proposition}\label{prop_osculation2}
Pick $\theta_0\in(0,\frac{\pi}{6})$ and $v\in\mathbb{C}$ with $v^3=\pm1$, and let $U\subset\mathbb{C}$ be the $\varepsilon$-neighborhood, for some $\varepsilon>0$, of the sector
$$
\overline{V}:=\{te^{\theta\ima}v\in\mathbb{C}\mid t\geq0,\ |\theta|\leq\theta_0\}.
$$
Let $g=e^u|\dz|^2$ be a solution to Wang's equation on $(U,\dz^3)$ such that
$$
0\leq u(z)\leq Ce^{-\sqrt{6}\,c|z|}\ \text{ for all } z\in U,
$$
for constants $C>0$ and 
$c\in\left(\tfrac{1}{2}\left[\sin(\theta_0)+\sin\left(\tfrac{\pi}{3}+\theta_0\right)\right], \sin\left(\tfrac{\pi}{3}+\theta_0\right)\right)$.
Let $P:U\rightarrow\SL(\T_0\mathbb{C}\oplus\mathbb{R})\cong\SL(3,\mathbb{R})$ be the osculation map associated to $g$ at $0$.
Then there are continuous maps $\widehat{P}:\overline{V}\rightarrow\SL(3,\mathbb{R})$ and $R:\overline{V}\rightarrow \mathcal{N}_v$ such that 
\begin{itemize}
\item
$P(z)=\widehat{P}(z)R(z)$ for all $z\in\overline{V}$;
\item
$\widehat{P}(z)$ converges in $\SL(3,\mathbb{R})$ as $z$ goes to infinity in $\overline{V}$;
\item
$\|R(z)-I\|\leq C'e^{\sqrt{6}(\sin(\frac{\pi}{3}+\theta_0)-c)|z|}$ for some $C'>0$ and all $z\in \overline{V}$.
\end{itemize}
\end{proposition}
\begin{remark}
If the assumption on $u$ is strengthened to $0\leq u(x)\leq C e^{-\sqrt{6}\,c|z|}$ with $c>\sin\left(\tfrac{\pi}{3}+\theta_0\right)$,  one can show that $P(z)$ itself converges as $z\in\overline{V}$ goes to $\infty$. But we do not have such a strong control on $u$ in the case to which we apply the proposition (see the next section).
\end{remark}

\begin{proof}
Since the rotations by cubic roots of unity $U\mapsto \omega^k U$, $z\mapsto \omega^k z$ ($k=1,2$) preserve the cubic differential $\dz^3$, whereas any $v$ with $v^3=\pm1$ can be brought to $1$ or $e^{\pi\ima/3}$ by such rotations, we may suppose $v=1$ or $e^{\pi\ima/3}$ without loss of generality. We proceed with $v=e^{\pi\ima/3}$, while the proof for $v=1$ only has notational difference.

The proof is closely modeled on that of Proposition \ref{prop_osculation1}, so here we mainly emphasize the required modifications.
We start by writing $\D=\D_0+\Xi$ as in that proof, and 
obtain through the Schauder estimate
\begin{equation}\label{eqn_xiestimate2}
\|\Xi\|(z)\leq C_1e^{-\sqrt{6}\,c|z|}\ \text{ for all } z\in\overline{V}.
\end{equation}
\textbf{Step 1. Define $\widehat{P}(z)$ and $R(z)$.} By Property (\ref{item_mu5}) in Section \ref{subsec_setups}, $\mathfrak{n}_v$ is spanned by $E_{13}$ and $E_{23}$.  Let $\Xi_{\mathfrak{n}_v}:=\Xi_{13}E_{13}+\Xi_{23}E_{23}$ be the $\mathfrak{n}_v$-component of $\Xi$ and consider the auxiliary connection
$$
\D_1:=\D_0+\Xi_{\mathfrak{n}_v}.
$$
Let $\para_1$ be the parallel transport of $\D_1$ and consider the same path $\Gamma_z$ from $0$ to $z$ as in the proof of Proposition \ref{prop_osculation1}. We define the required maps as
$$
\widehat{P}(z):=\para(0,z)\para_1(\Gamma_z), \quad R(z):=\para_1(\Gamma_z^{-1})\para_0(z,0).
$$
The definition immediately implies $P(z)=\widehat{P}(z)R(z)$.

\textbf{Step 2. Estimate  of $R(z)$.} 
Similarly as in the proof of Proposition \ref{prop_osculation1}, we obtain
$$
R(z)=I+r_1(z)E_{13}+r_2(z)E_{23}\in \mathcal{N}_v\,,\text{ where } r_i(z)=\int_{\Gamma_z}\lambda_{i3}\Xi_{i3}\ \  (i=1,2).
$$
Here $\lambda_{ij}(z)=e^{\mu_{ij}(z/|z|)|z|}$ is the eigenvalue of $\Ad_{\para_0(0,z)}$ on $E_{ij}$ (see Section \ref{subsec_setups}). The $\mu_{ij}$'s appearing here can be written more explicit as
$$
\mu_{13}(e^{\theta\ima}v)=\sqrt{2}\re[(1-\omega)e^{\theta\ima}v]=\sqrt{6}\cos\left(\theta+\tfrac{\pi}{6}\right)=\sqrt{6}\sin\left(\tfrac{\pi}{3}-\theta\right),
$$
$$
\mu_{23}(e^{\theta\ima}v)=\sqrt{2}\re[(\omega^2-\omega)e^{\theta\ima}v]=\sqrt{6}\cos\left(\theta-\tfrac{\pi}{6}\right)=\sqrt{6}\sin\left(\tfrac{\pi}{3}+\theta\right).
$$
Given $z=te^{\theta\ima}v\in \overline{V}$, since $\theta$ is in the interval $[-\theta_0,\theta_0]$ with $0<\theta_0<\frac{\pi}{6}$, we have
$\sin\left(\frac{\pi}{3}-\theta_0\right)\leq\sin\left(\frac{\pi}{3}\pm\theta\right)\leq\sin\left(\frac{\pi}{3}+\theta_0\right)$, hence
\begin{equation}\label{eqn_lambda13}
e^{\sqrt{6}\sin\left(\tfrac{\pi}{3}-\theta_0\right)|z|} \leq\lambda_{i3}(z)\leq e^{\sqrt{6}\sin\left(\tfrac{\pi}{3}+\theta_0\right)|z|}\ \text{ for all }z\in\overline{V}\ (i=1,2).
\end{equation}
Combing the second inequality with (\ref{eqn_xiestimate2}), we get
\begin{equation}\label{eqn_lambdaxi13}
|\lambda_{i3}\Xi_{i3}|(z)\leq C_1 e^{\sqrt{6}\left(\sin(\tfrac{\pi}{3}+\theta_0)-c\right)|z|}\ \text{ for all } z\in\overline{V}\  (i=1,2).
\end{equation}
By integration, we get the required bound for $\|R(z)-I\|=\max\{|r_1(z)|,|r_2(z)|\}$.

\textbf{Step 3. Convergence of $\widehat{P}(tv)$.}
As in the proof of Proposition \ref{prop_osculation1}, if we consider the $\sl_3\mathbb{R}$-valued $1$-form $\Theta$ on $U$, defined at every $z\in U$ by
$$
\Theta_z:=\Ad_{\para_1(\Gamma_z^{-1})}(\Xi-\Xi_{\mathfrak{n}_v})_z=\Ad_{R(z)}\Ad_{\para_0(0,z)}(\Xi-\Xi_{\mathfrak{n}_v})_z,
$$
then $\widehat{P}(tv)$ solves the initial value problem of ODE 
\begin{equation}\label{eqn_hatpt2}
\frac{\dif }{\dif t}\widehat{P}(tv)=\widehat{P}(tv)\Theta_{tv}(v),\quad \widehat{P}(0)=I.
\end{equation}

We shall still show that $\|\Theta\|$ decays exponentially, although the proof is more delicate in the current setting. We first give an upper bound for each entry of
$$
\Ad_{\para_0(0,z)}(\Xi-\Xi_{\mathfrak{n}_v})_z
=
\begin{pmatrix}
\Xi_{11}&\lambda_{12}\Xi_{12}&0\\
\lambda_{21}\Xi_{21}&\Xi_{22}&0\\
\lambda_{31}\Xi_{31}&\lambda_{32}\Xi_{32}&\Xi_{33}
\end{pmatrix}_{\!\!\!z}.
$$
To this end, we need to control each $\lambda_{ij}$. By (\ref{eqn_lambda13}),  $\lambda_{31}(z)=\lambda_{13}(z)^{-1}$ and $\lambda_{32}(z)=\lambda_{23}(z)^{-1}$ are bounded from above by $e^{-\sqrt{6}\sin\left(\tfrac{\pi}{3}-\theta_0\right)|z|}$ for all $z\in\overline{V}$, whereas for $\lambda_{12}(z)=\lambda_{21}(z)^{-1}=e^{\mu_{12}(z/|z|)|z|}$, the expression
$$
\mu_{12}(e^{\theta\ima}v)=\sqrt{2}\re[(1-\omega^2)e^{\theta\ima}v]=-\sqrt{6}\sin(\theta)
$$
implies that $\lambda_{12}(z)$ and $\lambda_{21}(z)$ are bounded by $e^{\sqrt{6}\sin(\theta_0)}$ for all $z\in\overline{V}$. Combining these bounds  with the bound of the $|\Xi_{ij}|$'s given by (\ref{eqn_xiestimate2}), we obtain, for all $z\in\overline{V}$,
$$
|\Xi_{11}|(z),|\Xi_{22}|(z),|\Xi_{33}|(z)\leq C_1e^{-\sqrt{6}\,c|z|},
$$
$$
|\lambda_{31}\Xi_{31}|(z),|\lambda_{32}\Xi_{32}|(z)\leq C_1 e^{\sqrt{6}(-\sin(\frac{\pi}{3}-\theta_0)-c)|z|},
$$
$$
|\lambda_{12}\Xi_{12}|(z),|\lambda_{21}\Xi_{21}|(z)\leq C_1 e^{\sqrt{6}(\sin(\theta_0)-c)|z|}.
$$
 Note that the last bound is the biggest among the three, hence is an upper bound of $\|\Ad_{\para_0(0,z)}(\Xi-\Xi_{\mathfrak{n}_v})_z\|$. The fact that we have a finer bound for the $(3,1)$ and $(3,2)$ components will be important.


We proceed to estimate $\|\Theta\|$. Given any $R=I+r_1E_{13}+r_2E_{23}\in \mathcal{N}_v$ and $M=(M_{ij})\in\sl_3\mathbb{R}$ with $M_{13}=M_{23}=0$, we see from the expression
$$
\Ad_RM=\begin{pmatrix}
1&&r_1\\
&1&r_2\\
&&1
\end{pmatrix}
\begin{pmatrix}
M_{11}&M_{12}&0\\
M_{21}&M_{22}&0\\
M_{31}&M_{32}&M_{33}
\end{pmatrix}
\begin{pmatrix}
1&&-r_1\\
&1&-r_2\\
&&1
\end{pmatrix}
$$
$$
=\begin{pmatrix}
M_{11}+r_1M_{31}&M_{12}+r_1M_{32}&r_1M_{33}-r_1M_{11}-r^2_1M_{31}-r_2M_{12}-r_1r_2M_{32}\\
M_{21}+r_2M_{31}&M_{22}+r_2M_{32}&r_2M_{33}-r_1M_{21}-r_1r_2M_{31}-r_2M_{22}-r_2^2M_{32}\\
M_{31}&M_{32}&M_{33}-r_1M_{31}-r_2M_{32}
\end{pmatrix},
$$
that each matrix entry of $\Ad_RM$ is a sum of at most 5 terms of the form $M_{ij}$, $r_kM_{ij}$, $M_{33}$, $r_kM_{33}$ or $r_1^{l_1}r_2^{l_2}M_{3j}$, where $i,j,k\in\{1,2\}$ and $l_1+l_2\leq 2$ (remark that second order terms in $r_1$ and $r_2$ only appears with $M_{31}$ and $M_{32}$). Therefore,
$$
\|\Ad_RM\|\leq 5\max_{
\substack{
a=0,1,\,b=0,1,2\\j=1,2
}}
\big\{\|R-I\|^a\|M\|, \|R-I\|^b|M_{3j}|\big\}.
$$
Enlarging the constant $C'$ in the bound for $\|R(z)-I\|$ obtained above if necessary, we may assume that the bound is bigger than $1$ for all $z\in\overline{V}$. 
Using the last inequality, we deduce from the bounds for $\|R(z)-I\|$ and each entry of $\Ad_{\para_0(0,z)}(\Xi-\Xi_{\mathfrak{n}_v})_z$:
\begin{align*}
&\|\Theta\|(z)\leq5\cdot\max\Big\{C'e^{\sqrt{6}\left(\sin(\tfrac{\pi}{3}+\theta_0)-c\right)|z|} C_1 e^{\sqrt{6}\left(\sin(\theta_0)-c\right)|z|}\,,\\
&\hspace{110pt} \big(C'e^{\sqrt{6}\left(\sin(\tfrac{\pi}{3}+\theta_0)-c\right)|z|}\big)^2C_1|z|^\alpha e^{\sqrt{6}\left(-\sin(\frac{\pi}{3}-\theta_0)-c\right)|z|}\Big\}\\
&=5C'C_1\max\Big\{e^{\sqrt{6}\left(\sin(\frac{\pi}{3}+\theta_0)+\sin(\theta_0)-2c\right)|z|}\,, C'e^{\sqrt{6}\left(2\sin(\frac{\pi}{3}+\theta_0)-\sin(\frac{\pi}{3}-\theta_0)-3c\right)|z|}\Big\}
\end{align*}
for all $z\in\overline{V}$, $|z|\geq r'$.
Since by hypothesis we have 
$$
2c>\sin\big(\tfrac{\pi}{3}+\theta_0\big)+\sin(\theta_0)=2\sin\big(\tfrac{\pi}{3}+\theta_0\big)-\sin\big(\tfrac{\pi}{3}-\theta_0\big),
$$ 
there is $r>0$ such that when $|z|\geq r$ we have
$$
e^{\sqrt{6}\left(\sin(\frac{\pi}{3}+\theta_0)+\sin(\theta_0)-2c\right)|z|}\geq C'e^{\sqrt{6}\left(\sin(\tfrac{\pi}{3}+\theta_0)-3c\right)|z|},
$$
so that the above inequality for $\|\Theta\|$ amounts to
\begin{equation}\label{eqn_theta2}
\|\Theta\|(z)\leq C_2e^{\sqrt{6}\left(\sin(\frac{\pi}{3}+\theta_0)+\sin(\theta_0)-2c\right)|z|}\ \text{ for all } z\in\overline{V},\,|z|\geq r.
\end{equation}

The exponent is negative by assumption. So   $\int_0^{+\infty}\|\Theta_{tv}(v)\|\dif t<+\infty$, and we can apply \cite[Lemma B.1(ii)]{dumas-wolf} again to Eq.(\ref{eqn_hatpt2}) and conclude that $\widehat{P}(tv)$ converges to some $\widehat{P}(\infty)\in\SL(3,\mathbb{R})$ as $t\rightarrow+\infty$.

\textbf{Step 4. Compare $\widehat{P}(z)$ and $\widehat{P}(tv)$.} Consider the $\SL(3,\mathbb{R})$-valued function $Q$ on $[-\theta_0,\theta_0]\times[0,+\infty)$ given by $Q(\theta,t):=\widehat{P}(tv)^{-1}\widehat{P}(te^{\theta\ima}v)$. The computation of $\frac{\pa}{\pa \theta}Q(\theta,t)$ in the proof of Proposition \ref{prop_osculation1} is still valid and gives
$$
\frac{\pa}{\pa \theta}Q(\theta,t)=Q(\theta,t)\Theta(\pa_\theta(te^{\theta\ima}v)),\quad Q(0,t)=I.
$$
By (\ref{eqn_theta2}), $\|\Theta(\pa_\theta(te^{\theta\ima}v))\|\rightarrow0$ uniformly in $\theta\in[-\theta_0,\theta_0]$ as $t\rightarrow+\infty$. Therefore, $\|Q(\theta,t)-I\|\rightarrow0$ uniformly in $\theta$ by \cite[Lemma B.1(i)]{dumas-wolf}, and this implies $\widehat{P}(te^{\theta\ima}v)\rightarrow\widehat{P}(\infty)$ uniformly in $\theta$, completing the proof.
\end{proof}

\subsection{Proof of Propositions \ref{prop_reduction2} and \ref{prop_reduction3}}\label{subsec_reduction2}

\begin{proof}[Proofs of Propositions \ref{prop_reduction2}]
As in Section \ref{subsec_reduction1}, by a shift of coordinate, we consider $g=e^u|\dz|^2$ as defined on the translated sector $V_2:=V_1-v$, which contains $0$, and consider the affine spherical embeddings
$$
\iota_0,\iota:V_2\rightarrow \T_0\mathbb{C}\oplus\mathbb{R}\cong\mathbb{R}^3,\ \iota_0(z)=\para_0(0,z)\underline{1}_z,\ \iota(z)=\para(0,z)\underline{1}_z=P(z)\iota_0(z),
$$
where $P(z):=\para(0,z)\para_0(z,0)$ is the osculation map.

Since $\sin(\theta_1)>\frac{\sqrt{3}}{4}=\frac{1}{2}\sin(\frac{\pi}{3})$ by assumption, we can pick a small enough $\theta_0\in(0,\frac{\pi}{6})$  such that
$$
c':=\sin(\theta_1-\theta_0)>\tfrac{1}{2}[\sin(\theta_0)+\sin\left(\tfrac{\pi}{3}+\theta_0\right)].
$$
The idea of proof of the required proposition is to apply Proposition \ref{prop_osculation2} to $u$ on 
$$
\overline{V}:=\{te^{\theta\ima}v\in\mathbb{C}\mid t\geq0,\ |\theta|\leq\theta_0\}.
$$ 
This sector is contained in $\overline{V}_1$ with the property that the distance from any $z\in\overline{V}$ to $\pa V_1$ is at least $c'|z|$. So we can apply Theorem \ref{thm_diskfine} to $u$ on the disk $D(z,c'|z|)\subset V_2$ and obtain
$u(z)\leq C|z|^\frac{1}{2}e^{-\sqrt{6}\,c'|z|}$ for all $z\in\overline{V}$, $|z|\geq r$. Pick a constant $c$ such that
$$
c'>c>\tfrac{1}{2}[\sin(\theta_0)+\sin\left(\tfrac{\pi}{3}+\theta_0\right)].
$$ 
Since $|z|^\frac{1}{2}$ is controlled by $e^{(c'-c)|z|}$, enlarging $C$ if necessary, we may suppose
\begin{equation}\label{eqn_reduction21}
0\leq u(z)\leq C e^{-\sqrt{6}\,c|z|}\ \text{ for all } z\in\overline{V}.
\end{equation}
The same argument applies to the translated sector $\overline{V}-v$, which contains an $\varepsilon$-neighborhood of $\overline{V}$, and yields $0\leq u(z)\leq Ce^{-\sqrt{6}c|z+v|}$ for all $z\in \overline{V}-v$. But this implies that (\ref{eqn_reduction21}) still holds on $\overline{V}-v$ after modifying $C$, hence the hypothesis on $u$ in Proposition \ref{prop_osculation2} is fulfilled.

Proposition \ref{prop_osculation2} gives a factorization $P(z)=\widehat{P}(z)R(z)$ for all $z\in\overline{V}$, such that $\widehat{P}(z)\rightarrow\widehat{P}(\infty)\in\SL(3,\mathbb{R})$ as $|z|\rightarrow+\infty$, while $R(z)\in \mathcal{N}_v$ satisfies
\begin{equation}\label{eqn_reduction2r}
\|R(z)-I\|\leq C_1e^{\sqrt{6}\left(\sin(\frac{\pi}{3}+\theta_0)-c\right)|z|}\leq C_1e^{\frac{3\sqrt{2}}{4}|z|}\ \text{ for all } z\in\overline{V},
\end{equation}
where the last inequality is because
$$
\sin(\tfrac{\pi}{3}+\theta_0)-c<\sin(\tfrac{\pi}{3}+\theta_0)-\tfrac{1}{2}[\sin(\theta_0)+\sin\left(\tfrac{\pi}{3}+\theta_0\right)]=\tfrac{1}{2}\sin\left(\tfrac{\pi}{3}-\theta_0\right)<\tfrac{\sqrt{3}}{4}.
$$

We can now define the required extension $\overline{\delta}$ of $\delta=\mathbb{P}\circ\iota$ by setting
$$
\overline{\delta}|_{\pai{v}\mathbb{C}}=\widehat{P}(\infty)\circ\overline{\delta}_0|_{\pai{v}\mathbb{C}},
$$
where $\overline{\delta}_0$ is the extension of $\delta_0=\mathbb{P}\circ\iota_0$ from Section \ref{subsec_bord1}. As seen in Section \ref{subsec_bord1}, $\overline{\delta}_0$ maps $\pai{v}\mathbb{C}$ isometrically to an open edge $I_v$ of the triangle $\Delta$, hence $\overline{\delta}$  maps $\pai{v}\mathbb{C}$ isometrically to the open line segment $\widehat{P}(\infty)I_v$. 

As in the proof of Proposition \ref{prop_reduction1}, now it only remains to be shown that for any sequence $(z_n)$ in $V_2$ tending to some $z_\infty\in\pai{v}\mathbb{C}$, the image $\delta(z_n)$ converges in $\mathbb{RP}^2$ to $\overline{\delta}(z_\infty)=\widehat{P}(\infty)\overline{\delta}_0(z_\infty)$.
Suppose $z_\infty=[\gamma]$ for a negative ray $\gamma:[0,+\infty)\rightarrow V_2$ of the form $\gamma(t)=\gamma(0)+tv$. Then the convergence ``$z_n\rightarrow[\gamma]$'' means 
``$|z_n|\rightarrow+\infty$ and the distance from $z_n$ to the ray $\gamma([0,+\infty))=\gamma(0)+v\mathbb{R}_{\geq0}$ tends to $0$'' (see Section \ref{subsec_bord1}). This implies $z_n\in\overline{V}$ for $n$ big enough, so we may suppose $z_n\in\overline{V}$ for all $n$. Since $\widehat{P}(z_n)\rightarrow\widehat{P}(\infty)$, in order to prove the required convergence $\delta(z_n)=\widehat{P}(z_n)R(z_n)\delta_0(z_n)\rightarrow\widehat{P}(\infty)\overline{\delta}_0(z_\infty)$, it is sufficient to show that $R(z_n)\delta_0(z_n)\rightarrow\overline{\delta}_0(z_\infty)$. 

To this end, let us suppose, say, $v=e^{\frac{\pi\ima}{3}}$ (the proof for $v=-1$ and $e^{-\frac{\pi\ima}{3}}$ only has notational difference). Then $\mathcal{N}_{e^{\pi\ima/3}}=\{I+r_1E_{13}+r_2E_{23}\}_{r_1,r_2\in\mathbb{R}}$ (see Section \ref{subsec_setups}, Property (\ref{item_mu5})) and we can write
$$
R(z)\delta_0(z)=
\begin{pmatrix}
1&&r_1(z)\\
&1&r_2(z)\\
&&1
\end{pmatrix}
\begin{bmatrix}
e^{\sqrt{2}\re(z)}\\
e^{\sqrt{2}\re(\omega^2z)}\\
e^{\sqrt{2}\re(\omega z)}\\
\end{bmatrix}
=
\begin{bmatrix}
1+r_1(z)e^{\sqrt{2}\re((\omega-1) z)}\\
e^{\sqrt{2}\re((\omega^2-1)z)}+r_2(z)e^{\sqrt{2}\re((\omega-1) z)}\\
e^{\sqrt{2}\re((\omega-1) z)}\\
\end{bmatrix}
$$
(terms in square brackets are homogeneous coordinates of $\mathbb{RP}^2$). Since $|z_n|\rightarrow+\infty$ and the distance from $z_n$ to the ray $\gamma(0)+e^{\frac{\pi\ima}{3}}\mathbb{R}_{\geq0}$ tends to $0$, we have approximately $z_n\approx |z_n|e^{\frac{\pi\ima}{3}}$ when $n$ is large, hence
$$
\re((\omega-1)z_n)=\re(\sqrt{3}\,e^{\frac{5\pi\ima}{6}}z_n)\approx \re(\sqrt{3}\,e^{\frac{5\pi\ima}{6}}e^{\frac{\pi\ima}{3}}|z_n|)=
-\tfrac{3}{2}|z_n|
$$ 
(``$\approx$'' means more precisely that the difference between the two sides is bounded). Combining with the bound (\ref{eqn_reduction2r}) of $\|R(z)-I\|=\max\{|r_1(z)|,\,|r_2(z)|\}$, we get
$$
\big|r_i(z_n)e^{\sqrt{2}\re((\omega-1) z_n)}\big|\leq C_2\, e^{-\frac{3\sqrt{2}}{4}|z_n|}\to0\ \ (i=1,2).
$$ 
Therefore, we see from the above expression of $R(z)\delta_0(z)$ that $R(z_n)\delta_0(z_n)$ has the same limit in $\mathbb{RP}^2$ as $\delta_0(z_n)$, which is $\overline{\delta}_0(z_\infty)$. This establishes the required convergence and completes the proof of Proposition \ref{prop_reduction2}.
\end{proof}
\begin{proof}[Proof of Proposition \ref{prop_reduction3}]
Again, we shift the coordinate and view $g$ and $\iota$ as defined on $\tau\mathbb{H}_{-1}=\{\tau z\mid \im(z)>-1\}$. Denoting the vector bundle $\T(\tau\mathbb{H}_{-1})\oplus\underline{\mathbb{R}}$ over $\tau\mathbb{H}_{-1}$ simply by $E$, we let  $E_z:=\T_z(\tau\mathbb{H}_{-1})\oplus\mathbb{R}$ be the fiber at $z\in \tau\mathbb{H}_{-1}$, and $L_z$ be the linear isomorphism $E_z\rightarrow E_{z+\tau}$ induced by the translation $z\mapsto z+\tau$.


Since the metric $g$ and cubic differential $\dz^3$ are preserved by the translation, 
the affine sphere connection $\D$ associated to $(g,\dz^3)$ is invariant under the bundle map induced by the translation, hence the parallel transport $\para$ of $\D$ satisfies
\begin{equation}\label{eqn_tauz}
L_{z_1}\circ\para(z_1,z_2)=\para(z_1+\tau,z_2+\tau)\circ L_{z_2}
\end{equation}
for all $z_1,z_2\in \tau\mathbb{H}_{-1}$ (both sides are linear isomorphisms $E_{z_2}\to E_{z_1}$). The parallel transport $\para_0$ of the \c{T}i\c{t}eica affine sphere connection $\D_0$ has the same property. The linear transformation $H\in \SL(E_0)\cong\SL(3,\mathbb{R})$ is then naturally defined as
$$
H:=L_{-\tau}\circ \para(-\tau,0)=\para(0,\tau)\circ L_0.
$$

The required equivariance property $\iota(z+\tau)=H(\iota(z))$ is verified with (\ref{eqn_tauz}) and the fact that the section $\underline{1}$ of $\underline{\mathbb{R}}\subset E$ is preserved by the translation:
\begin{align*}
H(\iota(z))=H\para(0,z)\underline{1}_z=L_{-\tau}\para(-\tau,0)\para(0,z)\underline{1}_z=L_{-\tau}\para(-\tau,z)\underline{1}_z\\
=\para(0,z+\tau) L_z\underline{1}_z=\para(0,z+\tau)\underline{1}_{z+\tau}=\iota(z+\tau).
\end{align*}
Also, $H$ is the unique linear transformation of $E_0$ satisfying the equivariance property because $\iota(z)$ linearly spans $E_0$ as $z$ runs over $\tau\mathbb{H}_{-1}$.

To prove the statement on eigenvalues of $H$, we compare $H$ with the linear transformation $H_0\in\SL(E_0)$ defined in the same way as $H$ but using $\para_0$ instead. 
Using the matrix expression of $\para_0$ (see Section \ref{subsec_titeica}), we get
$$
H_0=
\begin{pmatrix}
e^{\sqrt{2}\re(\tau)}\hspace{-7pt}&&\\
&e^{\sqrt{2}\re(\omega^2\tau)}\hspace{-7pt}&\\
&&e^{\sqrt{2}\re(\omega\tau)}
\end{pmatrix},
\mbox{ where } \omega=e^{2\pi\ima/3}.
$$
On the other hand, using (\ref{eqn_tauz}) (and the same equality for $\para_0$), we get the following relation between $H$ and $H_0$ involving the osculation map $P(z)=\para(0,z)\para_0(z,0)$:
\begin{align}
P(z+\tau)H_0P(z)^{-1}&=\para(0,z+\tau)\para_0(z+\tau,0)\para_0(0,\tau)L_0\,\para_0(0,z)\para(z,0)\label{eqn_red3}\\
&=\para(0,z+\tau)\para_0(z+\tau,\tau)\para_0(\tau,z+\tau)L_{z}\para(z,0)\nonumber\\
&=\para(0,z+\tau)\para(z+\tau,\tau)L_0=\para(0,\tau)L_0=H.\nonumber
\end{align}
We shall eliminate the dependence on $z$ of the relation by letting $z$ go to infinity. To achieve this, consider the cubic root $v$ of $-1$ in the assumption, which is contained in the un-shifted half-plane $\tau\mathbb{H}$. The assumption on the distance from $v$ To $\tau\mathbb{R}$ implies that we can take a small enough $\theta_0$, in the same way as in the above proof of Propositions \ref{prop_reduction2}, such that on the sector
$\overline{V}:=\{te^{\theta\ima}v\mid t\geq0,\,|\theta|\leq\theta_0\}$, the function $u$ satisfies the assumption of Proposition \ref{prop_osculation2}. That proposition then yields a factorization $P(z)=\widehat{P}(z)R(z)$ for $z\in\overline{V}$, with $R(z)\in \mathcal{N}_v$ and $\lim_{z\in\overline{V},|z|\rightarrow\infty}\widehat{P}(z)=\widehat{P}(\infty)\in\SL(3,\mathbb{R})$. Thus, when $z$ and $z+\tau$ are both contained in $\overline{V}$, (\ref{eqn_red3}) can be written as
$\widehat{P}(z+\tau)R(z+\tau)H_0R(z)^{-1}\widehat{P}(z)^{-1}=H$.
Equivalently, we have
$$
H_0^{-1}R(z+\tau)H_0R(z)^{-1}=H_0^{-1}\widehat{P}(z+\tau)^{-1}H\widehat{P}(z).
$$ 
Here we put $H_0^{-1}$ on both sides because it makes the left-hand side contained in $\mathcal{N}_v$, while the right-hand side converges as $z\to+\infty$. We conclude that the limit
$$
R_0:=H_0^{-1}\widehat{P}(\infty)^{-1}H\widehat{P}(\infty)
$$ 
is in $\mathcal{N}_v$ as well, and $H$ can be written as
\begin{equation}\label{eqn_hh0}
H=\widehat{P}(\infty)H_0R_0\widehat{P}(\infty)^{-1}.
\end{equation}

We can now verify the required statement about the eigenvalues of $H$. First note that the extension $\overline{\delta}$ is so constructed that $\overline{\delta}|_{\pai{v}\mathbb{C}}=\widehat{P}(\infty)\circ\overline{\delta}_0|_{\pai{v}\mathbb{C}}$(see the proof of Proposition \ref{prop_reduction2}). As seen in Section \ref{subsec_bord1}, $\overline{\delta}_0$ maps $\pai{v}\mathbb{C}$ to an open edge $I_v$ of the triangle $\Delta$, sending the endpoints $\pai{e^{\pi\ima/3}v}\mathbb{C}$ and $\pai{e^{-\pi\ima/3}v}\mathbb{C}$ of $\pai{v}\mathbb{C}$ to the endpoints $x^+_0$ and $x^-_0$ of $I_v$, respectively, which are ordered counterclockwise in Figure \ref{figure_continuity}. A case-by-case check for $v=e^{\pi\ima/3}$, $-1$ and $e^{\pi\ima/3}$ shows that $H_0R_0$ has eigenvalues $\exp\big(\sqrt{2}\re(e^{\mp\frac{\pi\ima}{3}}v^{-1}\tau)\big)$ at $x^\pm_0$ for any $R_0\in \mathcal{N}_v$ (for example, if $v=e^{\pi\ima/3}$ then $x_0^-=[1:0:0]$, $x_0^+=[0:1:0]$ and $$
H_0R_0=
\begin{pmatrix}
e^{\sqrt{2}\re(h)}\hspace{-8pt}&0&*\\[4pt]
&e^{\sqrt{2}\re(\omega^2h)}\hspace{-8pt}&*\\[4pt]
&&e^{\sqrt{2}\re(\omega h)}
\end{pmatrix}
$$
has eigenvalues $e^{\sqrt{2}\re(h)}=e^{\sqrt{2}\re(e^{\frac{\pi\ima}{3}}v^{-1}\tau)}$ and $e^{\sqrt{2}\re(\omega^2h)}=e^{\sqrt{2}\re(e^{-\frac{\pi\ima}{3}}v^{-1}\tau)}$ at $x_0^-$ and $x_0^+$, respectively). Therefore, as $H$ has the form (\ref{eqn_hh0}), it has eigenvalue $\exp\big(\sqrt{2}\re(e^{\mp\frac{\pi\ima}{3}}v^{-1}\tau)\big)$ at the endpoint $x^\pm=\widehat{P}(\infty)x^\pm_0$ of $I:=\overline{\delta}(\pai{v}\mathbb{C})=\widehat{P}(\infty)I_v$. We conclude the proof of Proposition \ref{prop_reduction3} by noting that $x^\pm$ is the limit of $\overline{\delta}(\zeta)=\widehat{P}(\infty)\overline{\delta}_0(\zeta)$ as $\zeta\in\pai{v}\mathbb{C}$ tends to $\pai{e^{\pm\frac{\pi\ima}{3}}v}\mathbb{C}$.
\end{proof}

\section{Cubic differential around (piecewise) geodesic boundary}\label{sec_7} 
In this chapter, we finish the proof of Theorem \ref{thm_intro2} by establishing the ``if'' direction in Parts (\ref{item_introthm2}) and (\ref{item_introthm1}). 

\subsection{Blaschke curvature near boundary lines and corners}\label{subsec_bnb}
By a \emph{corner} of a properly convex domain $\Omega\subset\mathbb{RP}^2$, we mean a point $p\in\pa\Omega$ at which $\pa\Omega$ is not $C^1$,  \ie $\pa\Omega$ admits two non-collinear tangent rays at $p$. Let $g_\Omega$ and $\ve{\phi}_\Omega$ be the metric and cubic differential on $\Omega$ given by the Blaschke metric and normalized Pick differential on the affine sphere $M_\Omega\subset\mathbb{R}^3$ projecting to $\Omega$ (see Sections \ref{subsec_yau}). Using the Hausdorff continuity property in Section \ref{subsec_benoisthulin}, we can control the curvature $\kappa_\Omega$ of $g_\Omega$ near boundary lines and corners:
\begin{lemma}\label{lemma_boundline}
Let $\Omega\subset\mathbb{RP}^2$ be a properly convex domain. Fix $0<\varepsilon<1$.
\begin{enumerate}
\item\label{item_boundline1}
Suppose $\pa\Omega$ contains a line segment $I$. Then every point $p$ in the interior of $I$ has a neighborhood $B$ in $\mathbb{RP}^2$ such that $\kappa_{\Omega}>-\varepsilon$ in $B\cap\Omega$.
\item\label{item_boundline2}
Suppose $p\in\pa\Omega$ is a corner and fix an open triangle $T\subset\Omega$ with a vertex at $p$. Then $p$ has a neighborhood $B$ in $\mathbb{RP}^2$ such that $\kappa_{\Omega}>-\varepsilon$ in $B\cap T$. 
\end{enumerate}
\end{lemma}
Note that in the second part, for each of the two tangent rays to $\pa\Omega$ at $p$, the boundary $\pa\Omega$ can either contain an initial segment of the ray or meet the ray only at $p$. Figure \ref{figure_prooffomega2} below illustrates an example where both situations occur. However, if both rays are in the former situation, the lemma just says that $\kappa_\Omega>-\varepsilon$ holds in $B\cap\Omega$ for a neighborhood $B$ of $p$. This is actually the case needed in the next section. 
\begin{proof}
(\ref{item_boundline1}) 
We may suppose $I$ is a maximal segment on $\pa\Omega$, \ie the intersection of $\pa\Omega$ with a line in $\mathbb{RP}^2$. Let $\Delta\subset\Omega$ be the open triangle spanned by $I$ and a point $x_0$ in $\Omega$, see the first picture in Figure \ref{figure_prooffomega1}. Choose a point $x_1\in\Delta$. Since the affine sphere projecting to $\Delta$ has flat Blaschke metric (see Section \ref{subsec_titeica}), the curvature $\kappa_\Delta$ of the metric $g_\Delta$ on $\Delta$ vanishes. By Corollary \ref{coro_khausdorff}, there exists a metric ball $\mathcal{B}\subset\mathfrak{C}$ (under the Hausdorff distance given by a metric on $\mathbb{RP}^2$) centered at $\Delta$ such that every $\Omega'\in\mathcal{B}$ contains $x_1$ and satisfies $\kappa_{\Omega'}(x_1)>-\varepsilon$.

For any $x\in \Delta$, let $f_{x}\in\SL(3,\mathbb{R})$ denote the projective transformation which preserves $\Delta$ and brings $x$ to $x_1$. Since $\kappa_{\Omega}(x)=\kappa_{f_{x}(\Omega)}(x_1)$ (see Section \ref{subsec_benoisthulin}), a sufficient condition for $\kappa_\Omega(x)>-\varepsilon$ is
$f_{x}(\Omega)\in\mathcal{B}$.
Therefore, we only need to find, for every point $p$ in the interior of $I$, a neighborhood $B$ of $p$ in $\mathbb{RP}^2$ such that 
$f_{x}(\Omega)\in\mathcal{B}$ for all $x\in B\cap\Delta$.

To this end, we work with the affine chart $\mathbb{R}^2\subset\mathbb{RP}^2$ in which $\Delta$ is the first quadrant and $x_0=(0,0)$, $x_1=(1,1)$, so that the restriction of $f_x$ to the chart is given by $f_{x}(y)=\big(\frac{y^1}{x^1},\frac{y^2}{x^2}\big)$, where $y^1$ and $y^2$ denote the coordinates of $y\in\mathbb{R}^2$. Suppose $p$ is the point at infinity of lines in $\mathbb{R}^2$ with a fixed slope $k>0$. Choose $k_0>0$ satisfying  $k_0<k<k^{-1}_0$ and consider the sector $\Delta'=\{x\in\mathbb{R}^2\mid k_0<x^2/x^1<k_0^{-1}\}$. Then $\Delta'$ is a triangle contained in $\Delta$ with an edge $I'\subset I$ containing $p$, see Figure \ref{figure_prooffomega1}.
\begin{figure}[h]
\centering\includegraphics[width=4in]{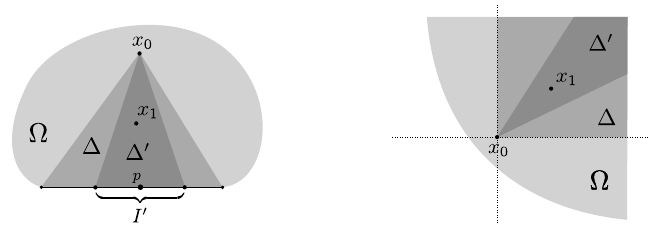}
\caption{Proof of Lemma \ref{lemma_boundline} (\ref{item_boundline1})}
\label{figure_prooffomega1}
\end{figure}

Given $\lambda\in(0,1)$ and a point $a$ in the third quadrant $-\Delta$, we let $L_{a,\lambda}$ denote the region containing $\Delta$ and bounded by the rays issuing from $a$ with slopes $-\lambda$ and $-\lambda^{-1}$. Such a region is a triangle in $\mathbb{RP}^2$, hence an element of $\mathfrak{C}$, with the following key property: there are constants $r_0>0$ and $\lambda_0>0$ such that $L_{a,\lambda}\in\mathcal{B}$ whenever $\|a\|:=((a^1)^2+(a^2)^2)^\frac{1}{2}< r_0$ and $\lambda<\lambda_0$. 

Since $I$ is a maximal segment, for any $\lambda>0$ we can take $a$ sufficiently far away from $0$ such that $\Omega\subset L_{a,\lambda}$. In particular, we fix $a_0\in-\Delta$ such that $\Omega\subset L_{a_0,k_0\lambda_0}$. Using the formula $f_{x}(L_{a,\lambda})=L_{f_{x}(a),\,\lambda x^1\!/x^2}$, which follows from the expression of $f_x$, we check that 
\begin{equation}\label{eqn_fxl}
f_x(L_{a_0,k_0\lambda_0})\in\mathcal{B}\,\mbox{ for all }x\in\Delta' \mbox{ such that } x^1,x^2> \|a_0\|/r_0.
\end{equation}
Indeed, since $k_0\leq\frac{x^2}{x^1}\leq k_0^{-1}$ for $x\in\Delta'$, the $\lambda$-parameter of $f_x(L_{a_0,k_0\lambda_0})$ is less than $\lambda_0$, while $x^1,x^2\geq \|a_0\|/r_0$ implies that the base point $f_x(a_0)$ of $f_x(L_{a_0,k_0\lambda_0})$ has norm less than $a_0$. Thus, Property (\ref{eqn_fxl}) is implied by the definition of $r_0$ and $\lambda_0$.

Since $\mathcal{B}$ is a Hausdorff metric ball centered at $\Delta$ while we have the inclusion relations $\Delta=f_x(\Delta)\subset f_x(\Omega)\subset f_x(L_{a_0,k_0\lambda_0})$, if $f_x(L_{a_0,k_0\lambda_0})$ is contained in $\mathcal{B}$, then so does  $f_x(\Omega)$. Therefore, Property (\ref{eqn_fxl}) implies that $f_x(\Omega)\in\mathcal{B}$ for all $x\in U:=\{(x^1,x^2)\in\Delta'\mid x^1,x^2>\|a_0\|/r_0\}$. So $B:=U\cup I'$ is a neighborhood of $p$ in $\overline{\Delta}$ fulfilling the requirement.

(\ref{item_boundline2}) Take an open triangle $\Delta$ containing $\Omega$ with two edges tangent to $\pa\Omega$ at $p$ as in Figure \ref{figure_prooffomega2}. Fix $x_1\in\Delta$. Defining $\mathcal{B}$ and $f_{x}$ in the same way as in the proof of Part (\ref{item_boundline1}), we only need to show $f_x(\Omega)\in\mathcal{B}$ for all $x\in T$ close enough to $p$.  
\begin{figure}[h]
\centering\includegraphics[width=4.4in]{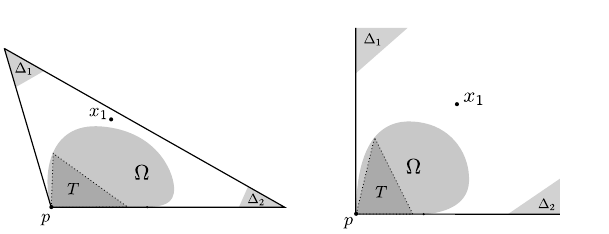}
\caption{Proof of Lemma \ref{lemma_boundline} (\ref{item_boundline2})}
\label{figure_prooffomega2}
\end{figure}

To this end, take small enough sub-triangles $\Delta_1, \Delta_2\subset\Delta$ as in the picture, such that for any convex domain $\Omega'$ in $\Delta$ with boundary containing $p$, we have $\Omega'\in\mathcal{B}$ whenever $\Omega'$ meets both $\Delta_1$ and $\Delta_2$. It is then sufficient to show that $f_x(T)\subset f_{x}(\Omega)$ meets $\Delta_1$ and $\Delta_2$, or equivalently, $T$ meets $f_{x}^{-1}(\Delta_1)$ and $f_{x}^{-1}(\Delta_2)$, when $x\in T$ is close enough to $p$. But the latter statement is elementary to check by working on the affine chart in which $\Delta$ is the first quadrant and $x_1=(1,1)$, as in the second picture in Figure \ref{figure_prooffomega2}.
\end{proof}

\subsection{(Piecewise) geodesic boundaries and poles}\label{subsec_piecewise}
Given a convex $\mathbb{RP}^2$-surface $S$, we identify $S$ as the quotient of a properly convex domain $\Omega\subset\mathbb{RP}^2$ by projective automorphisms. Recall from  Section \ref{subsec_yau} that the Blaschke metric and normalized Pick differential of $S$ are the quotients of $g_\Omega$ and $\ve{\phi}_\Omega$ on $\Omega$, respectively. The information on the curvature of $g_\Omega$ obtained above allows us to show:
\begin{theorem}[\textbf{(Piecewise) geodesic boundaries are poles}]
\label{thm_ptoc}
Let $S'$ be an oriented convex $\mathbb{RP}^2$-surface with geodesic or piecewise geodesic boundary such that the boundary $\pa S'$ is homeomorphic to a circle. Let $S$ be the interior of $S'$, viewed as an open convex $\mathbb{RP}^2$-surface, and let $g$ and $\ve{\phi}$ be the Blaschke metric and normalized Pick differential on $S$, respectively. Then the end of $S$ at $\pa S'$ is conformally a puncture (for the conformal structure underlying $g$) at which $\ve{\phi}$ has a pole of order at least $3$.
\end{theorem}
\begin{proof}
Identify $S'$ as the quotient of a properly convex set $\Omega'\subset\mathbb{RP}^2$ by projective transformations, so that $S$ is the quotient of the interior $\Omega$ of $\Omega'$. Then $g$ and $\ve{\phi}$ are given by the metric $g_\Omega$ and cubic differential $\ve{\phi}_\Omega$ on $\Omega$. 

The boundary $\pa S'$ being geodesic or piecewise geodesic means that every point of $\pa S'$ is the image, by the quotient map $\Omega'\rightarrow S'$, of either an interior point of a line segment in $\pa\Omega$ or the junction point of two line segments in $\pa\Omega$. By Lemma \ref{lemma_boundline}, every such point has a neighborhood $B$ in $\mathbb{RP}^2$ such that $\kappa_\Omega=-1+\|\ve{\phi}_\Omega\|^2_{g_\Omega}\geq -1/2$ on $B\cap \Omega$. As a consequence, $\pa S'$ has a neighborhood $U'$ in $S'$ such that 
\begin{equation}\label{eqn_proofptoc}
\kappa_g=-1+\|\ve{\phi}\|^2_{g}\geq -1/2\ \mbox{ in }\, U:=U'\setminus\pa S'.
\end{equation}
After shrinking $U'$ if necessary, we can identify the closure $\overline{U}$ of $U$ in $(S,\ac{J})$ conformally with the annulus $\{z\in\mathbb{C}\mid r<|z|\leq 1\}$ for some $r\in[0,1)$ and write $\ve{\phi}=\phi(z)\dz^3$ in $\overline{U}$, where $\phi$ is holomorphic and nowhere vanishing. By (\ref{eqn_proofptoc}),  the conformal ratio between the flat metric $|\ve{\phi}|^\frac{2}{3}=|\phi(z)|^\frac{2}{3}|\dz|^2$ and the metric $g$ is bounded from below by $2^{-\frac{1}{3}}$ in $\overline{U}$, so the completeness of $g$ implies that the restriction of $|\ve{\phi}|^\frac{2}{3}$ to $\overline{U}$ is complete. Using Lemma \ref{lemma_huber}, we conclude that $r=0$ and $\phi$ has a pole of order at least $3$ at $z=0$, as required.
\end{proof}
By construction of the one-to-one correspondence between convex $\mathbb{RP}^2$-structures and cubic differentials explained in Sections \ref{subsec_convexity} and \ref{subsec_yau}, Theorem \ref{thm_ptoc} implies that under the hypotheses of Theorem \ref{thm_intro2}, if $X$ admits an extension $X'$ as stated in Part (\ref{item_introthm2}) or (\ref{item_introthm1}), then $p$ is a pole of order at least $3$. Combining with the ``only if'' statements of (\ref{item_introthm2}) and (\ref{item_introthm1}) already proved in Section \ref{sec_6}, we get the required ``if'' statements.

\appendix

\section{Connections and parallel transports}\label{subsec_connection}
In this section, we fix some notations about connections and parallel transports used in Chapters \ref{sec_pre} and \ref{sec_6}. See \eg \cite{taubes} for a survey of the background materials.
 
Let $(E,\mu,\D)$ be an $\SL(n,\mathbb{R})$-vector bundle over a manifold $M$, namely, $E$ is a real vector bundle of rank $n$, $\D$ is a connection on $E$ and $\mu$ is a volume form on $E$ preserved by $\D$.
The \emph{parallel transport} $\para$ of $\D$ assigns to every oriented $C^1$-path $\gamma:[a,b]\rightarrow M$ a linear isomorphism between fibers
$$
\para(\gamma): E_{\gamma(a)}\rightarrow E_{\gamma(b)}
$$
respecting the volume form $\mu$. It has the properties
 $$
\para(\gamma^{-1})=\para(\gamma)^{-1},\quad \para(\beta)\para(\gamma)=\para(\beta\gamma),
$$
where $\gamma^{-1}$ is the reverse path of $\gamma$, and $\beta$ is any path starting from the point where $\gamma$ terminates. Here and below, the reader should be cautious about the order in the notations. See the next remark. 

In particular, when $M$ is simply connected and $\D$ is flat, we let $\para(y,x)$ denote the parallel transport of $\D$ along any path from the point $x$ to $y$, which does not depend on the choice of the path. In this case, the above properties become
 $$
\para(x,y)=\para(y,x)^{-1},\quad \para(z,y)\para(y,x)=\para(z,x).
$$
\begin{remark}\label{remark_order}
In our notation, the concatenated path $\beta\gamma$ first runs $\gamma$ and then $\beta$. Namely, the order is the same as in the notation for composition of maps, and might be different from the convention in algebraic topology. Accordingly, in the notation $\para(y,x)$, the starting point $x$ of the parallel transport is put in the \emph{second} slot.
\end{remark}

In terms of the matrix expression of the connection
$$
\D=\dif+A
$$ 
under a unimodular frame of $(E,\mu)$ on an open set containing $\gamma$, the parallel transport $\para(\gamma)$ is obtained by solving an ODE as follows: For every $t\in[a,b]$, let $\gamma_{[a,t]}$ denote the restriction of $\gamma$ to the interval $[a,t]$ and view the parallel transport $\para(\gamma_{[a,t]})$ from $E_{\gamma(a)}$ to $E_{\gamma(t)}$ as a matrix in $\SL(n,\mathbb{R})$ by means of the frame. Then we have
$$
\frac{\dif}{\dif t}\para(\gamma_{[a,t]}) +A(\dot{\gamma}(t))\para(\gamma_{[a,t]})=0.
$$
This equation together with the obvious initial condition $\para(\gamma_{[a,a]})=I$ determines $\para(\gamma)=\para(\gamma_{[a,b]})$.

When $M$ is simply connected and $\D$ is flat, writing $\D=\dif+A$ and viewing $\para(y,x)$ as a matrix in $\SL(n,\mathbb{R})$ under a unimodular frame as above, we get a $2$-point parallel transportation map
$$
M\times M\rightarrow \SL(3,\mathbb{R}), \quad (y,x)\mapsto \para(y,x).
$$
It is characterized by the following properties:
\begin{itemize}
\item
$\para(x,x)=I$, $\para(x,y)=\para(y,x)^{-1}$;
\item
Let $\pa_X^{(2)}\para(y,x)$ (resp. $\pa_Y^{(1)}\para(y,x)$) denote the derivative of $\para(y,x)$ with respect to the variable $x$ (resp. $y$) along the tangent vector $X\in\T_XM$ (resp. $Y\in\T_yM$). Then
$$
\pa^{(2)}_X\para(y,x)=\para(y,x)A(X),\quad\pa^{(1)}_Y\para(y,x)=-A(Y)\para(y,x).
$$
\end{itemize}

\section{Ends of flat surfaces}\label{sec_flatend}
A \emph{flat surface} is a surface endowed with a flat Riemannian metric. In this appendix, we collect some results, used in Sections \ref{subsec_nonexistence}, \ref{subsec_finitvolume}, \ref{subsec_normalform} and \ref{subsec_piecewise}, about ends of flat surfaces.  To formulate the results, we call a flat surface homeomorphic to the half-open annulus $\mathbb{S}^1\times[0,1)$ a \emph{flat end}, and study its properties ``at infinity'', \ie properties invariant under the following equivalence relation: flat ends $F_1$ and $F_2$ are said to be \emph{equivalent} if they have compact subsets $C_1$ and $C_2$ such that there is an orientation preserving isometry from $F_1\setminus C_1$ to $F_2\setminus C_2$.
\subsection{Regular flat ends}\label{subsec_regularity}
First note that any flat end can be realized as an annulus 
$$
\Omega_r:=\{z\in\mathbb{C}\mid r<|z|\leq 1\}
$$ 
($r\in[0,1)$) equipped with a metric given by a holomorphic function:
\begin{proposition}
	Every flat end is equivalent to one of the form
	\begin{equation}\label{eqn_flatend}
	(\Omega_r,\ |f(z)|^{2\alpha}|\dz|^2),
	\end{equation}
	where $0\leq r<1$, $\alpha>0$ and $f$ is a nowhere vanishing holomorphic function defined on a larger annulus $\{r<|z|<1+\varepsilon\}$.
\end{proposition}
\begin{proof}
A conformal metric $e^{2u(z)}|\dz|^2$ on a domain in $\mathbb{C}$ has curvature $-e^{-2u}\Delta u$, hence is flat if and only if $u$ is harmonic. Therefore, any flat end, after removing a compact part around the boundary if necessary, has the form
$(\Omega_r, e^{2u}|\dz|^2)$
for some $r\in[0,1)$ and some harmonic function $u$ on a larger annulus $\{r<|z|<1+\varepsilon\}$.
But it is a classical fact (see \cite[Theorem 9.15]{axler}) that given such a $u$, there exist $\beta\in\mathbb{R}$ and a holomorphic function $h$ on $\{r<|z|<1+\varepsilon\}$ such that 
$$
u(z)=\re h(z)+\beta\log|z|.
$$
So we can write $e^{2u}|\dz|^2=|f(z)|^{2\alpha}|\dz|^2$ with $f(z)=ze^{h(z)}$, $\alpha=1$ if $\beta=0$, and $f(z)=ze^\frac{h(z)}{\beta}$ (resp. $z^{-1}e^{-\frac{h(z)}{\beta}}$), $\alpha=\beta$ (resp. $-\beta$) if $\beta>0$ (resp. $<0$).
\end{proof}

\begin{definition}[\textbf{Regular flat ends}]\label{def_regularflat}
A flat end is \emph{regular} if it is equivalent to (\ref{eqn_flatend}) such that $r=0$ and $f(z)$ has a non-essential singularity at $0$.
\end{definition}

A flat end is said to be \emph{complete} if it is complete as a Riemannian manifold with boundary, \ie if any path leaving every compact subset has infinite length. The works of A. Huber \cite{huber_1, huber} in more general settings imply:
\begin{lemma}
\label{lemma_huber}
The flat end (\ref{eqn_flatend}) is complete if and only if $r=0$ and $f(z)$ has a pole of order $\geq 1/\alpha$ at $z=0$. In particular, complete flat ends are regular.
\end{lemma}
\begin{proof}
The ``if'' part is elementary. The ``only if'' part follows from the following formulation of Huber's results (see \cite[\S 1]{hulin-troyanov}), which is a generalization of the fact for harmonic functions in the previous proof: given a $C^2$-function $u$ on $\Omega_r$, if the metric $e^{2u(z)}|\dz|^2$ is complete and $\int_{\Omega_r}|\Delta u|<+\infty$ (\ie the metric has finite total curvature), then $r=0$ and $u$ can be written as
$$
u(z)=\re h(z)+\beta\log|z|-\frac{1}{2\pi}\int_{\Omega_0}\log|z-\zeta|\Delta u(\zeta)\dif\mu(\zeta)
$$
for a constant $\beta\leq-1$ and a holomorphic function $h$ on the whole disk $\{|z|\leq 1\}$, where $\mu$ denotes the Lebesgue measure.
\end{proof}

The next lemma shows that when $r=0$, the finiteness of volume also implies regularity. For applications in Section \ref{subsec_finitvolume}, we allow $f$ to have zeros:
\begin{lemma}\label{lemma_area}
Let $f$ be a holomorphic function on the punctured disk $\Omega_0$ and $\alpha>0$ be a constant. Then we have $\int_{\Omega_0}|f(z)|^{2\alpha}\dif\mu(z)<+\infty$ (where $\mu$ is the Lebesgue measure on $\mathbb{C}$) if and only if $f(z)$ has a pole of order $<1/\alpha$ at $z=0$. 
\end{lemma}
\begin{proof}
The statement is elementary if $f(z)$ have non-essential singularity at $z=0$, so we only need to show that if $|f(z)|^{2\alpha}$ has finite integral then the singularity of $f$ at $z=0$ is not essential.  To this end, we work with the coordinate $w=1/z$, so that the former condition becomes
\begin{equation}\label{eqn_proofarea}
\int_{0<|z|\leq1}|f(z)|^{2\alpha}\dif\mu(z)=\int_{|w|\geq 1}\frac{1}{|w|^4}|f(\tfrac{1}{w})|^{2\alpha}\dif\mu(w)<+\infty.
\end{equation}
We claim that for any nonnegative subharmonic function $F(w)$ on $\{|w|\geq 1\}$, if 
$$
\int_{|w|\geq 1}\frac{1}{|w|}F(w)\dif\mu(w)<+\infty,
$$
then $F$ is bounded. To prove the claim, we consider the harmonic function $H_R$ on the disk $\{|w|<R\}$ whose boundary values coincide with the values of $F$ on $\{|w|=R\}$, for every $R\geq1$. The Poisson formula gives 
$$
H_R(w)=\frac{1}{2\pi R}\int_{|\zeta|=R}\frac{R^2-|w|^2}{|w-\zeta|^2}F(\zeta)|\dif\zeta|\geq0.
$$
As a consequence, for any $1\leq r<R$ we have
$$
\max_{|w|=r}H_R(w)\leq \frac{R+r}{2\pi R(R-r)}\int_{|\zeta|=R}F(\zeta)|\dif\zeta|.
$$
Integrating this inequality with respect to $R\in[2r,+\infty)$ yields
$$
\int_{2r}^{+\infty}\max_{|w|=r}H_R(w)\dif R\leq \frac{3}{2\pi}\int_{2r}^\infty\frac{\dif R}{R}\int_{|\zeta|=R}F(\zeta)|\dif\zeta|= \frac{3}{2\pi}\int_{|w|\geq{2r}}\frac{F(w)}{|w|}\dif\mu(w).
$$
By assumption, the last integral converges, hence so does the first one. This implies 
\begin{equation}\label{eqn_proofarea2}
\liminf_{R\rightarrow+\infty}\max_{|w|=r}H_R(w)=0.
\end{equation}
On the other hand, $H_R$ plus the constant $\max_{|w|=1}F(w)$ majorizes $F$ on the boundary of the annulus  $\{1\leq|w|\leq R\}$, hence also on the whole annulus. It follows that
$$
\max_{|w|=r}F(w)\leq \max_{|w|=r}H_R(w)+\max_{|w|=1}F(w).
$$
With a fixed $r\geq1$, this holds for all $R>r$. Therefore, taking Inequality (\ref{eqn_proofarea2}) into account, we get $\max_{|w|=r}F(w)\leq \max_{|w|=1}F(w)$, hence the claim.

Therefore, Condition (\ref{eqn_proofarea}) implies that $F(w)=|w|^{-3}|f(\frac{1}{w})|^{2\alpha}=|z|^3|f(z)|^{2\alpha}$ is bounded as $z$ runs over $\Omega_0$. It then follows from the Big Picard Theorem that $z=0$ is not an essential singularity of $f$, as required.
\end{proof}


\subsection{Invariants and classification}\label{subsec_classification}
An orientation preserving isometry $\sigma:z\mapsto az+b$ ($|a|=1$) of the Euclidean plane $\mathbb{E}^2:=(\mathbb{C},|\dz|^2)$ is determined up to conjugation by two invariants:
\begin{itemize}
	\item the \emph{rotation} $\rho(\sigma):=a\in\mathbb{S}^1:=\{a\in\mathbb{C}\mid |a|=1\}$;
	\item the \emph{translation length} $t(\sigma):=\min_{z\in\mathbb{C}}|\sigma(z)-z|\in\mathbb{R}_{\geq0}$, which equals $0$ (resp. $|b|$) when $a\neq1$ (resp. $a=1$).
\end{itemize}

Viewing a flat metric on a surface as an $\mathbb{E}^2$-structure (in the sense of \cite{goldman_gx}), we can consider monodromy invariants of a flat end $F$, similarly as for $\frac{1}{k}$-translation ends (see Section \ref{subsec_normalform}). In particular, we define the \emph{monodromy rotation} and \emph{monodromy translation length} of $F$, denoting them by $\rho(F)\in\mathbb{S}^1$ and
$t(F)\in\mathbb{R}_{\geq0}$, respectively, with a slight abuse of notation. 

While these invariants capture the monodromy of $F$, the monodromy does not completely determine $F$ up to equivalence. For example, the following nonequivalent  flat ends have the same monodromy, namely the identity:
$$
F_1=(\{0<|z|\leq 1\},|\dz|^2),\ F_2=\text{the double cover of }F_1,\ F_3=(\{|z|\geq 1\},|\dz|^2).
$$ 
Nevertheless, there is an invariant $\theta(F)\in\mathbb{R}$, the \emph{total boundary curvature}, such that
\begin{itemize}
\item $\theta(F)$ covers $\rho(F)$, in the sense that $e^{\theta(F)\ima}=\rho(F)$;
\item $\theta(F)$ and $t(F)$ determine $F$ up to equivalence as long as $F$ is regular.
\end{itemize}
The definition is as follows: If the boundary $\pa F$ is of class $C^2$, so that the geodesic curvature $\kappa_g$ of $\pa F$, with respect to the inward pointing normal vector field on $\pa F$, is a continuous function $\pa F\rightarrow\mathbb{R}$,  and we put
$$
\theta(F):=\int_{\pa F}\kappa_g\dif s,
$$ 
where $s$ is the length parameter. Note that our choice of normal vectors makes $\kappa_g>0$ (resp. $<0$) at $p\in\pa F$ if and only if $F$ is locally convex (resp. concave) at $p$.  For certain $\pa F$ with less regularity, we can still define the measure $\kappa_g\dif s$ and take its total mass as $\theta(F)$. 
For instance, if $\pa F$ is piecewise geodesic, then $\kappa_g\dif s$ is supported at the vertices, and the mass at a vertex $p$ is the angle of $F$ at $p$ minus $\pi$. 

\begin{remark}\label{remark_gaussbonnet}
The fact  that $\theta(F)$ is invariant under the equivalence relation can be proved with the following observation: for any flat end $F'\subset F$ such that $F\setminus F'$ is a topological annulus bounded by $\pa F$ and $\pa F'$, we have $\theta(F)=\theta(F')$ by the Gauss-Bonnet formula applied to that annulus.
\end{remark}
When a developing map of $F$ is known, an effective way of computing $\theta(F)$ is to take $F'\subset F$ as in the above remark, with piecewise geodesic boundary, so that $\theta(F')$ can be read off geometrically. For example, for above $F_1$, $F_2$ and $F_3$, we get
$$
\theta(F_1)=2\pi,\ \theta(F_2)=4\pi,\ \theta(F_3)=-2\pi.
$$

The aforementioned fact that $\theta(F)$ and $t(F)$ determine $F$ up to equivalence can be stated more precisely as:
\begin{theorem}[\textbf{Classification of regular flat ends}]
\label{thm_classification}
For any regular flat end $F$, the pair $(\theta(F),t(F))$ belongs to one of the four cases in the first column of the table below. Conversely, given $(\theta,t)$ belonging to these cases, there exists a regular flat end $F$, unique up to equivalence, such that $(\theta(F),t(F))=(\theta,t)$. An concrete example of such an $F$ is given in the second column (see Figures \ref{figure_geo1} and \ref{figure_geo2} for the construction), and the third column displays a flat metric on $\{0<|w|<\varepsilon\}$ which also yields such an $F$.

\vspace{5pt}

\begin{tabular}{|c|c|c|}
	\hline
	invariants&geometric model&analytic model\\
	\hline
	\rule[-4pt]{0pt}{16pt} $\theta>0$, $t=0$&cone of angle $\theta$& \\
	\cline{1-2}
	\rule[-4pt]{0pt}{16pt} $\theta<0$, $t=0$&funnel of angle $-\theta$&\raisebox{9pt}[0pt]{$|w|^{2\left(\frac{\theta}{2\pi}-1\right)}|\dif w|^2$}\\
	\hline
	\rule[-6pt]{0pt}{20pt} $\theta=0$, $t>0$&cylinder of perimeter $t$&$\left(\frac{t}{2\pi}\right)^2\frac{|\dif w|^2}{|w|^2}$\\
	\hline
	\rule[-6pt]{0pt}{20pt}  $\theta=-2\pi l$ ($l\in\mathbb{Z}_+$), $t>0$&$t$-grafted funnel of angle $2\pi l$&$\left|\frac{1}{w^l}+\frac{t}{2\pi}\right|^2\frac{|\dif w|^2}{|z|^2}$\\
	\hline
\end{tabular}
\end{theorem}

\begin{figure}[h]
	\includegraphics[width=4.7in]{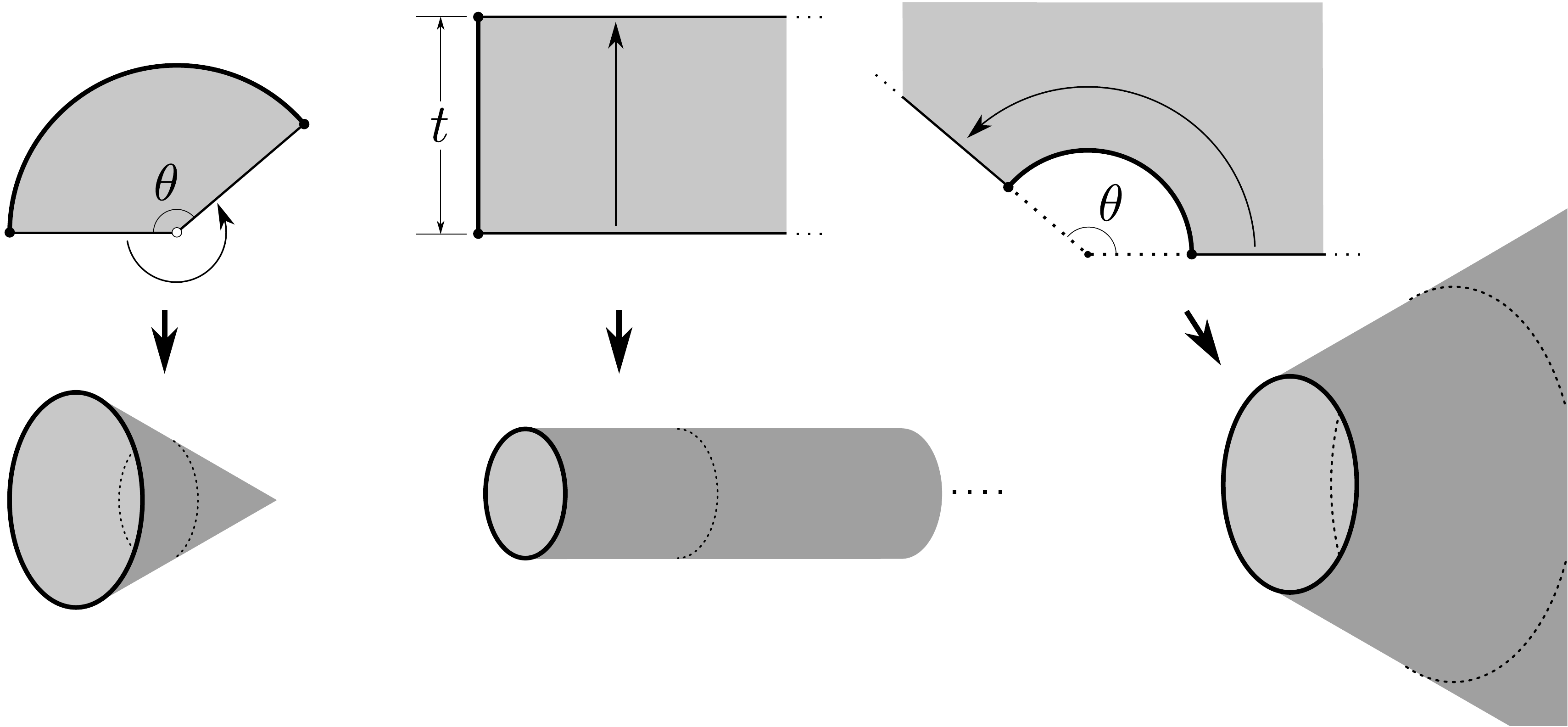}
	\caption{From left to right: cone of angle $\theta$,  cylinder of perimeter $t$ and funnel of angle $\theta$. Each of them is constructed from a flat surface with boundary by gluing boundary rays. 
	}\label{figure_geo1}
\end{figure}
\begin{figure}[h]
	\includegraphics[width=2.8in]{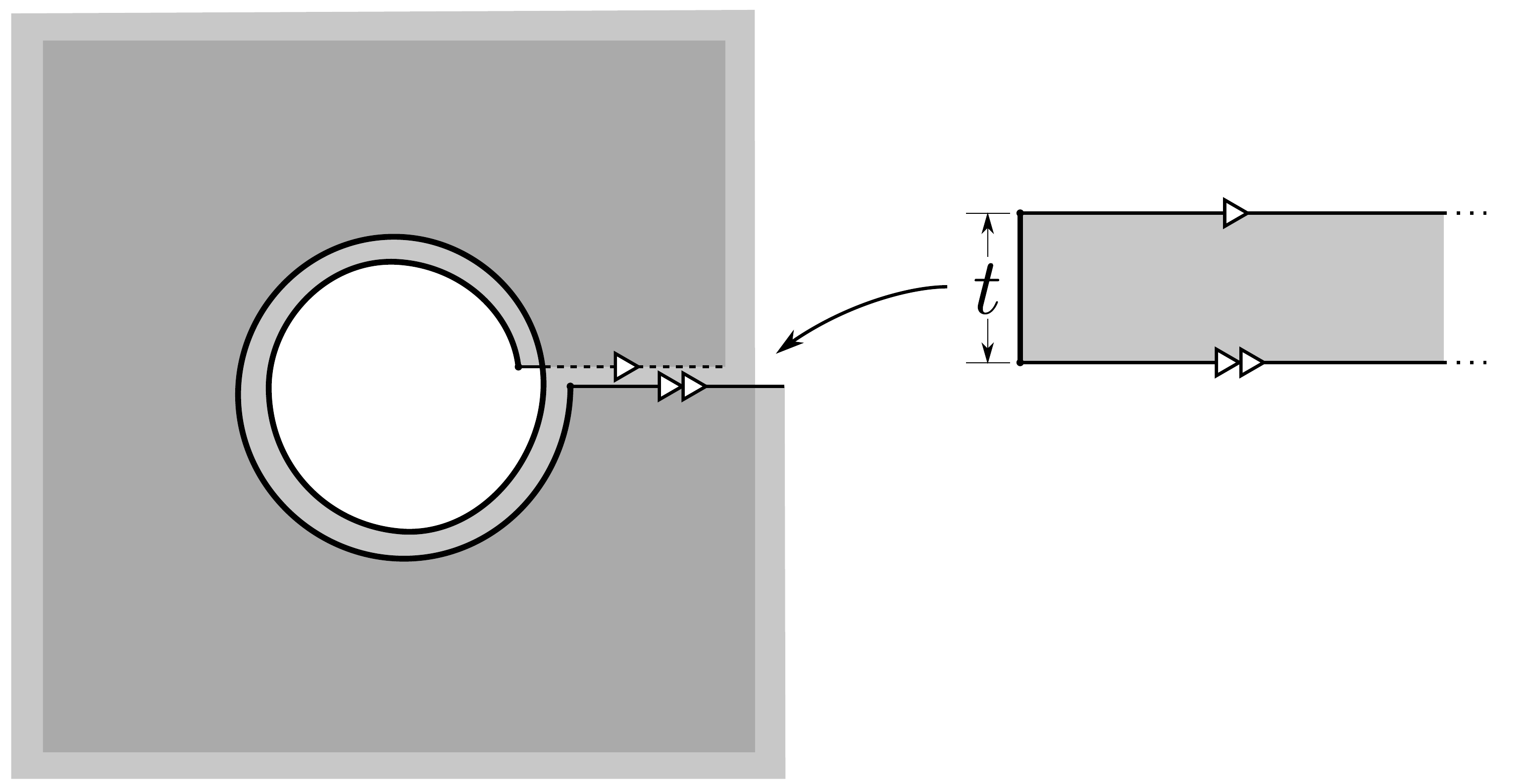}
	\caption{$t$-grafting of a funnel of angle $2\pi l$ when $l=2$: On the left is a funnel of angle $4\pi$, cut open along a ray. We add to it a half-infinite strip of width $t$ by gluing along boundary rays.}
	\label{figure_geo2}
\end{figure}
\begin{proof}
	We first check that the geometric and analytic models in the table do give regular flat ends with the prescribed invariants. For the geometric models, the regularity follows from Lemmas \ref{lemma_huber} and \ref{lemma_area}, and the invariants can be seen directly from the construction. The analytic models are regular by definition, and for each of them we can construct an explicit developing map from the universal cover of $\{0<|w|<\varepsilon\}$ to $\mathbb{E}^2$ similarly as in the proof of Lemma \ref{lemma_invariants}, so as to read off the invariants. In fact, these developing maps can be interpreted as isomorphisms from the analytic models to the geometric models.
	
	To prove the theorem, we now only need to show that any regular flat end is equivalent to some $(\{0<|w|<\varepsilon\},g)$, where $g$ is one of the metrics in the third column. This is a consequence of the following lemma.
\end{proof}
\begin{lemma}\label{lemma_normalform}
Given any $\alpha,\beta\in\mathbb{R}$ and  nowhere vanishing holomorphic function ${f}(z)$ defined on a neighborhood of $0$ in $\mathbb{C}$, there exists a conformal local coordinate $w=w(z)$ centered at $0$ such that the flat metric $|z|^{2\alpha}|f(z)|^{2\beta}|\dz|^2$, define on a punctured neighborhood of $0$, is written in the coordinate $w$ as
$$
|z|^{2\alpha}|f(z)|^{2\beta}|\dz|^2=
\begin{cases}
|w|^{2\alpha}|\dif w|^2&\text{if }\alpha \notin\{-1,-2,\cdots\},\\[5pt]
R|w|^{-2}|\dif w|^2\ \ (R>0)&\text{if }\alpha =-1,\\[5pt]
|w|^{-2}|w^{-l}+A|^{2}|\dif w|^2\ \ (A\geq0)&\text{if } \alpha=-(l+1)\text{ with }l\in\mathbb{Z}_+.
\end{cases}
$$
\end{lemma}
The lemma gives a local normal form for flat metrics around punctures. The proof below is adapted from \cite[\S 6]{strebel}, where a similar normal form for quadratic differentials is given (\cf Section \ref{subsec_normalform}).
\begin{proof}
Fix a sufficiently small disk $D=\{|z|<\varepsilon\}$, on which we pick a single-valued branch of $\log{f}(z)$ so as to define arbitrary powers of $f(z)$. In particular, ${f}(z)^\beta$ is well defined on $D$ and is given by a convergent series
$$
{f}(z)^\beta=a_0+a_1z+a_2z^2+\cdots\quad (a_0\neq 0).
$$

\textbf{Case 1: $\alpha\notin\{-1,-2,\cdots\}$.}
We define a holomorphic function ${g}(z)$ on $D$ with ${g}(0)\neq0$ as the convergent power series
$$
{g}(z):=\frac{a_0}{\alpha +1}+\frac{a_1z}{\alpha+2}+\frac{a_2z^2}{\alpha +3}+\cdots.
$$
Fix a sufficiently small sub-disk $D'=\{|z|<\varepsilon'\}\subset D$, on which  we pick a single-valued branch of $\log{g}(z)$ so as to define arbitrary powers of $g(z)$. Then 
$$w(z):=z{g}(z)^\frac{1}{\alpha+1}$$ 
is a conformal local coordinate on $D'$. Viewing $z^{\alpha}$ and $w(z)^{\alpha}$ as functions on the universal cover of $D'\setminus\{0\}$, we can write
$$
w(z)^{\alpha+1}=z^{\alpha +1}{g}(z)=\frac{a_0z^{\alpha +1}}{\alpha +1}+\frac{a_1z^{\alpha +2}}{\alpha +2}+\frac{a_2z^{\alpha +3}}{\alpha +3}+\cdots.
$$
Differentiating this series and comparing with the series of $f(z)^\beta$, we get $\left(w(z)^{\alpha+1}\right)'=z^{\alpha}{f}(z)^\beta$. But on the other hand
$\left(w(z)^{\alpha+1}\right)'=(\alpha  +1)w(z)^{\alpha}w'(z)$. Therefore,
$$
|\alpha+1||w|^{\alpha}|\dif w|=|\alpha+1||w(z)|^{\alpha }|w'(z)||\dz|=|z|^{\alpha}|{f}(z)|^\beta|\dz|.
$$
We then get the required coordinate after scaling $w$ by a constant.

\textbf{Case 2: $\alpha=-1$.} Put
$$
w(z):=z\exp\left(\frac{a_1}{a_0}z+\frac{a_2}{2a_0}z^2+\frac{a_3}{3a_0}z^3+\cdots\right).
$$
Then $w$ is a conformal local coordinate around $0$ and satisfies
$a_0w(z)^{-1}w'(z)=\big(a_0\log w(z)\big)'=z^{-1}{f}(z)^\beta$. We have $|z|^{-1}|{f}(z)|^\beta|\dz|=|a_0||w|^{-1}|\dif w|$, as required.

\textbf{Case 3: $\alpha=-(l+1)$ with $l\in\mathbb{Z}_+$.} Consider the following function on $D\setminus\{0\}$:
$$
\frac{{f}(z)^\beta}{z^{l+1}}=\frac{a_0}{z^{l+1}}+\cdots +\frac{a_l}{z}+a_{l+1}+a_{l+2}z+\cdots.
$$
Its primitive (defined on the universal cover of $D\setminus\{0\}$) is
$$
a_l\log z+\frac{1}{z^l}(b_0+b_1z+b_2z^2+\cdots)=:a_l\log z+\frac{1}{z^l}\eta(z),
$$
where $b_l=0$,
$b_n=a_n/(n-l)$ (for $n\neq l$). Using Implicit Function Theorem, we find a holomorphic function $h(z)$ on a neighborhood of $z=0$ such that 
$$
a_lz^lh(z)+e^{-lh(z)}=\eta(z).
$$
Setting $w(z):=ze^{h(z)}$, we have
$a_l\log w+\frac{1}{w^l}=a_l\log z+\frac{1}{z^l}\eta(z)$, hence
$$
\left(\frac{a_l}{w(z)}-\frac{l}{w(z)^{l+1}}\right)w'(z)=\left(a_l\log z+\frac{1}{z^l}\eta(z)\right)'=\frac{{f}(z)^\beta}{z^{l+1}}.
$$
Therefore,
$$
\left|\frac{l}{w^l}-a_l\right|\frac{|\dif w|}{|w|}=\left|\frac{a_l}{w(z)}-\frac{l}{w(z)^{l+1}}\right||w'(z)||\dz|=\frac{|{f}(z)|^\beta}{|z|^{l+1}}|\dz|=|z|^\alpha|f(z)|^\beta|\dz|.
$$
We then get the required coordinate after scaling $w$ by a constant.
\end{proof}

\bibliographystyle{amsalpha} \bibliography{ends}

\providecommand{\bysame}{\leavevmode\hbox to3em{\hrulefill}\thinspace}
\providecommand{\MR}{\relax\ifhmode\unskip\space\fi MR }
\providecommand{\MRhref}[2]{%
  \href{http://www.ams.org/mathscinet-getitem?mr=#1}{#2}
}
\providecommand{\href}[2]{#2}
\begin{thebibliography}{{Nie}15}

\bibitem[ABR01]{axler}
Sheldon Axler, Paul Bourdon, and Wade Ramey, \emph{Harmonic function theory},
  second ed., Graduate Texts in Mathematics, vol. 137, Springer-Verlag, New
  York, 2001. \MR{1805196}

\bibitem[BH13]{benoist-hulin}
Yves Benoist and Dominique Hulin, \emph{Cubic differentials and finite volume
  convex projective surfaces}, Geom. Topol. \textbf{17} (2013), no.~1,
  595--620. \MR{3039771}

\bibitem[BH14]{benoist-hulin_2}
\bysame, \emph{Cubic differentials and hyperbolic convex sets}, J. Differential
  Geom. \textbf{98} (2014), no.~1, 1--19. \MR{3238310}

\bibitem[CY75]{cheng-yau_3}
S.~Y. Cheng and S.~T. Yau, \emph{Differential equations on {R}iemannian
  manifolds and their geometric applications}, Comm. Pure Appl. Math.
  \textbf{28} (1975), no.~3, 333--354. \MR{0385749}

\bibitem[CY77]{cheng-yau_1}
Shiu~Yuen Cheng and Shing~Tung Yau, \emph{On the regularity of the
  {M}onge-{A}mp\`ere equation {${\rm det}(\partial \sp{2}u/\partial
  x\sb{i}\partial sx\sb{j})=F(x,u)$}}, Comm. Pure Appl. Math. \textbf{30}
  (1977), no.~1, 41--68. \MR{0437805 (55 \#10727)}

\bibitem[CY86]{cheng-yau_4}
Shiu~Yuen Cheng and Shing-Tung Yau, \emph{Complete affine hypersurfaces. {I}.
  {T}he completeness of affine metrics}, Comm. Pure Appl. Math. \textbf{39}
  (1986), no.~6, 839--866. \MR{859275}

\bibitem[DW15]{dumas-wolf}
David Dumas and Michael Wolf, \emph{Polynomial cubic differentials and convex
  polygons in the projective plane}, Geom. Funct. Anal. \textbf{25} (2015),
  no.~6, 1734--1798. \MR{3432157}

\bibitem[Eva98]{evans}
Lawrence~C. Evans, \emph{Partial differential equations}, Graduate Studies in
  Mathematics, vol.~19, American Mathematical Society, Providence, RI, 1998.
  \MR{1625845}

\bibitem[FG06]{fock-goncharov_1}
Vladimir Fock and Alexander Goncharov, \emph{Moduli spaces of local systems and
  higher {T}eichm{\"u}ller theory}, Publ. Math. Inst. Hautes {\'E}tudes Sci.
  (2006), no.~103, 1--211.

\bibitem[FG07]{fock-goncharov_2}
V.~V. Fock and A.~B. Goncharov, \emph{Moduli spaces of convex projective
  structures on surfaces}, Adv. Math. \textbf{208} (2007), no.~1, 249--273.
  \MR{2304317}

\bibitem[Gol90]{goldman_convex}
William~M. Goldman, \emph{Convex real projective structures on compact
  surfaces}, J. Differential Geom. \textbf{31} (1990), no.~3, 791--845.
  \MR{1053346 (91b:57001)}

\bibitem[Gol10]{goldman_gx}
\bysame, \emph{Locally homogeneous geometric manifolds}, Proceedings of the
  {I}nternational {C}ongress of {M}athematicians. {V}olume {II}, Hindustan Book
  Agency, New Delhi, 2010, pp.~717--744. \MR{2827816}

\bibitem[GT01]{gilbarg-trudinger}
David Gilbarg and Neil~S. Trudinger, \emph{Elliptic partial differential
  equations of second order}, Classics in Mathematics, Springer-Verlag, Berlin,
  2001, Reprint of the 1998 edition. \MR{1814364}

\bibitem[HT92]{hulin-troyanov}
D.~Hulin and M.~Troyanov, \emph{Prescribing curvature on open surfaces}, Math.
  Ann. \textbf{293} (1992), no.~2, 277--315. \MR{1166122}

\bibitem[Hub57]{huber_1}
Alfred Huber, \emph{On subharmonic functions and differential geometry in the
  large}, Comment. Math. Helv. \textbf{32} (1957), 13--72. \MR{0094452 (20
  \#970)}

\bibitem[Hub67]{huber}
\bysame, \emph{Vollst\"andige konforme {M}etriken und isolierte
  {S}ingularit\"aten subharmonischer {F}unktionen}, Comment. Math. Helv.
  \textbf{41} (1966/1967), 105--136. \MR{0224036 (36 \#7083)}

\bibitem[Jos13]{jost}
J{\"u}rgen Jost, \emph{Partial differential equations}, third ed., Graduate
  Texts in Mathematics, vol. 214, Springer, New York, 2013. \MR{3012036}

\bibitem[Lab07]{labourie_cubic}
Fran{\c{c}}ois Labourie, \emph{Flat projective structures on surfaces and cubic
  holomorphic differentials}, Pure Appl. Math. Q. \textbf{3} (2007), no.~4,
  part 1, 1057--1099. \MR{2402597 (2009c:53046)}

\bibitem[Li90]{li_calabi1}
An~Min Li, \emph{Calabi conjecture on hyperbolic affine hyperspheres}, Math. Z.
  \textbf{203} (1990), no.~3, 483--491. \MR{1038713}

\bibitem[Li92]{li_calabi2}
\bysame, \emph{Calabi conjecture on hyperbolic affine hyperspheres. {II}},
  Math. Ann. \textbf{293} (1992), no.~3, 485--493. \MR{1170522}

\bibitem[Li19]{li_on}
Qiongling Li, \emph{On the uniqueness of vortex equations and its geometric
  applications}, J. Geom. Anal. \textbf{29} (2019), no.~1, 105--120.
  \MR{3897005}

\bibitem[Lof01]{loftin_amer}
John~C. Loftin, \emph{Affine spheres and convex {$\Bbb{RP}\sp n$}-manifolds},
  Amer. J. Math. \textbf{123} (2001), no.~2, 255--274. \MR{1828223
  (2002c:53018)}

\bibitem[Lof04]{loftin_compactification}
\bysame, \emph{The compactification of the moduli space of convex {$\Bbb R\Bbb
  P\sp 2$} surfaces. {I}}, J. Differential Geom. \textbf{68} (2004), no.~2,
  223--276. \MR{2144248 (2006i:32014)}

\bibitem[Lof10]{loftin_survey}
John Loftin, \emph{Survey on affine spheres}, Handbook of geometric analysis,
  {N}o. 2, Adv. Lect. Math. (ALM), vol.~13, Int. Press, Somerville, MA, 2010,
  pp.~161--191. \MR{2743442}

\bibitem[Lof19]{loftin_neck}
\bysame, \emph{Convex {$\Bbb{RP}^2$} structures and cubic differentials under
  neck separation}, J. Differential Geom. \textbf{113} (2019), no.~2, 315--383.
  \MR{4023294}

\bibitem[Mar12]{marquis}
Ludovic Marquis, \emph{Surface projective convexe de volume fini}, Ann. Inst.
  Fourier (Grenoble) \textbf{62} (2012), no.~1, 325--392. \MR{2986273}

\bibitem[{Nie}15]{nie_mero}
X.~{Nie}, \emph{{Meromorphic cubic differentials and convex projective
  structures}}, ArXiv e-prints (2015).

\bibitem[NS94]{nomizu}
Katsumi Nomizu and Takeshi Sasaki, \emph{Affine differential geometry},
  Cambridge Tracts in Mathematics, vol. 111, Cambridge University Press,
  Cambridge, 1994, Geometry of affine immersions. \MR{1311248 (96e:53014)}

\bibitem[Ric63]{richards}
Ian Richards, \emph{On the classification of noncompact surfaces}, Trans. Amer.
  Math. Soc. \textbf{106} (1963), 259--269. \MR{0143186}

\bibitem[Str84]{strebel}
Kurt Strebel, \emph{Quadratic differentials}, Ergebnisse der Mathematik und
  ihrer Grenzgebiete (3) [Results in Mathematics and Related Areas (3)],
  vol.~5, Springer-Verlag, Berlin, 1984. \MR{743423 (86a:30072)}

\bibitem[Tau11]{taubes}
Clifford~Henry Taubes, \emph{Differential geometry}, Oxford Graduate Texts in
  Mathematics, vol.~23, Oxford University Press, Oxford, 2011, Bundles,
  connections, metrics and curvature. \MR{3135161}

\bibitem[WA94]{au-wan}
Tom Yau-Heng Wan and Thomas Kwok-Keung Au, \emph{Parabolic constant mean
  curvature spacelike surfaces}, Proc. Amer. Math. Soc. \textbf{120} (1994),
  no.~2, 559--564. \MR{1169052}

\bibitem[Wan92]{wan}
Tom Yau-Heng Wan, \emph{Constant mean curvature surface, harmonic maps, and
  universal {T}eichm\"uller space}, J. Differential Geom. \textbf{35} (1992),
  no.~3, 643--657. \MR{1163452}

\bibitem[Yau73]{yau_ahlfors}
S.~T. Yau, \emph{Remarks on conformal transformations}, J. Differential
  Geometry \textbf{8} (1973), 369--381. \MR{0339007}

\bibitem[Zor06]{zorich}
Anton Zorich, \emph{Flat surfaces}, Frontiers in number theory, physics, and
  geometry. {I}, Springer, Berlin, 2006, pp.~437--583. \MR{2261104}

\end{thebibliography}

\end{document}